\documentclass[11pt,reqno]{amsart}
\usepackage{amsmath}

\usepackage{amssymb}
\usepackage{amsthm}
\usepackage{amsfonts}

\usepackage{tikz}
\usepackage{tikz-cd}
\usepackage{etoolbox, textcomp}
\usepackage[colorlinks=true,citecolor=red,linkcolor=blue]{hyperref}
\usepackage{comment}
\usepackage{cancel}
\usepackage{ulem}

\title[$N=1$ super Virasoro tensor categories]{$N=1$ super Virasoro tensor categories}
\author[T. Creutzig]{Thomas Creutzig}
\address{(Thomas Creutzig) Department Mathematik, Friedrich-Alexander Universit\"at Erlangen-N\"urnberg, 
91058 Erlangen, Germany}\email{thomas.creutzig@fau.de}
\author[R. McRae]{Robert McRae} 
\address{(Robert Mcrae) Yau Mathematical Sciences Center, Tsinghua University, Beijing 100084, China}\email{rhmcrae@tsinghua.edu.cn}
\author[F.~Orosz Hunziker]{Florencia Orosz Hunziker}
\address{(Florencia Orosz Hunziker)
  Department of Mathematics\\
  University of Colorado Boulder\\
  Boulder, CO 80309-0395, United States of America.}
\email{florencia.orosz@colorado.edu}

\author[J. Yang]{Jinwei Yang}
\address{(Jinwei Yang) School of Mathematical Sciences, Shanghai Jiao Tong University, Shanghai 200240, China}
\email{jinwei2@sjtu.edu.cn}

\newtheorem{theorem}{Theorem}[section]
\newtheorem{lemma}[theorem]{Lemma}
\newtheorem{prop}[theorem]{Proposition}
\newtheorem{cor}[theorem]{Corollary}

\theoremstyle{definition}
\newtheorem{defn}[theorem]{Definition}
\newtheorem{remark}[theorem]{Remark}

\newtheorem{nota}[theorem]{Notation}

\usepackage[margin=1.25in]{geometry}

\def\CC{\mathbb{C}}

\def\NN{\mathbb{N}}

\def\QQ{\mathbb{Q}}
\def\RR{\mathbb{R}}

\def\ZZ{\mathbb{Z}}

\def\calY{\mathcal{Y}}

\def\vac{\mathbf{1}}

\newcommand\cA{\mathcal{A}}

\newcommand\cC{\mathcal{C}}

\newcommand\cE{\mathcal{E}}
\newcommand\cF{\mathcal{F}}

\newcommand\cJ{\mathcal{J}}

\newcommand\cM{\mathcal{M}}

\newcommand\cO{\mathcal{O}}

\newcommand\cR{\mathcal{R}}
\newcommand\cS{\mathcal{S}}

\newcommand\cW{\mathcal{W}}

\newcommand\cY{\mathcal{Y}}

\newcommand\tilC{\widetilde{C}}

\newcommand\tilW{\widetilde{W}}

\newcommand\id{\textup{id}}

\newcommand\End{\textup{End}}

\newcommand{\btimes}{\boxtimes}

\newcommand\quash[1]{}

\newcommand\ot{\otimes}
\newcommand\one{\mathbf{1}}

\renewcommand\a\alpha
\renewcommand\b\beta
\newcommand\g\gamma
\renewcommand\d\delta
\newcommand\D\Delta

\newcommand{\om}{\omega}

\newcommand{\NS}{\mathfrak{ns}}

\newcommand{\oc}{\mathcal{O}_{c}}
\newcommand{\ocfin}{ \oc^{\textup{fin}}}

\newcommand{\U}{\mathcal{U}}

\newcommand{\W}{\widetilde{W}}

\date{\today}
\subjclass[2020]{Primary 17B69, 17B65, 18M15, 81R10, 81T40}

\begin{document}

\numberwithin{equation}{section}

\begin{abstract}
    We show that the category of $C_1$-cofinite modules for the universal $N=1$ super Virasoro vertex operator superalgebra $\mathcal{S}(c,0)$ at any central charge $c$ is locally finite and admits the vertex algebraic braided tensor category structure of Huang-Lepowsky-Zhang. For central charges $c^{\mathfrak{ns}}(t)=\frac{15}{2}-3(t+t^{-1})$ with $t\notin\mathbb{Q}$, we show that this tensor category is semisimple, rigid, and slightly degenerate, and we determine its fusion rules. For central charge $c^{\mathfrak{ns}}(1)=\frac{3}{2}$, we show that this tensor category is rigid and that its simple modules have the same fusion rules as $\mathrm{Rep}\,\mathfrak{osp}(1\vert 2)$, in agreement with earlier fusion rule calculations of Milas. Finally, for the remaining central charges $c^{\mathfrak{ns}}(t)$ with $t\in\QQ^\times$, we show that the simple $\mathcal{S}(c^{\mathfrak{ns}}(t),0)$-module $\mathcal{S}_{2,2}$ of lowest conformal weight $h^{\mathfrak{ns}}_{2,2}(t)=\frac{3(t-1)^2}{8t}$ is rigid and self-dual, except possibly when $t^{\pm 1}$ is a negative integer or when $c^{\mathfrak{ns}}(t)$ is the central charge of a rational $N=1$ superconformal minimal model.
    As $\mathcal{S}_{2,2}$ is expected to generate the category of $C_1$-cofinite $\mathcal{S}(c^{\mathfrak{ns}}(t),0)$-modules under fusion, rigidity of $\mathcal{S}_{2,2}$ is the first key step to proving rigidity of this category for general $t\in\mathbb{Q}^\times$.
\end{abstract}

\maketitle
\thanks{\it Thomas Creutzig, Department Mathematik, Friedrich-Alexander Universit\"at Erlangen-N\"urnberg, 
91058 Erlangen, Germany, 0000-0002-7004-6472, thomas.creutzig@fau.de.

Robert McRae,  Yau Mathematical Sciences Center, Tsinghua University, Beijing 100084, China, rhmcrae@tsinghua.edu.cn.

Florencia Orosz Hunziker
Department of Mathematics, University of Colorado Boulder, Boulder, CO 80309-0395, United States of America, 0000-0003-2593-3694, florencia.orosz@colorado.edu (corresponding author).

Jinwei Yang, School of Mathematical Sciences, Shanghai Jiao Tong University, Shanghai 200240, China, jinwei2@sjtu.edu.cn.}

\vskip2cm
\noindent\textbf{Acknowledgments.}
We thank Andrew Linshaw for helpful discussions.
TC's work  is supported by a Natural Sciences and Engineering Research Council of Canada grant and a Deutsche Forschungsgemeinschaft grant.
RM's work is partially supported by a startup grant from Tsinghua University and by a research fellowship from the Alexander von Humboldt Foundation. RM also thanks
 the Universit\"{a}t Hamburg for hospitality while part of this research was performed. 
 FOH's work is supported by the United States National Science Foundation
under Grant No. DMS-2102786.
JY's work is supported by National Natural Science Foundation of China Grant No. 12371030 and a startup grant from Shanghai Jiao Tong University. 

\tableofcontents

\section{Introduction}

A vertex operator (super)algebra (VOA) is the chiral algebra of a two-dimensional (super)conformal field theory (CFT). Thus VOAs and their representation categories provide a bridge between representation theory and the physics of two-dimensional CFT and related physical theories such as string theory. 
The axiomatization of two-dimensional rational CFT led Moore and Seiberg to propose what are now called modular tensor categories \cite{MS88}. The VOA of a rational CFT is called strongly rational, and the highlight of the theory of strongly rational VOAs is Yi-Zhi Huang's theorem that their representation categories are indeed modular tensor categories \cite{HMod}, and that Verlinde's conjecture \cite{V} is true \cite{HVer}. In particular, strongly rational VOAs have only finitely many non-isomorphic simple modules, and every module is completely reducible. Most VOAs are not strongly rational, as they usually have indecomposable but reducible modules as well as an infinite number of inequivalent simple modules. Nonetheless, one still expects that non-rational VOAs often admit categories of modules which form rigid braided tensor categories that satisfy a non-semisimple variant of Verlinde's formula. It is an ongoing major effort to put this expectation on a solid foundation, in particular to show that Verlinde's formula holds under a certain natural setup \cite{Cr2}.

This work is concerned with the important example of the universal vertex operator superalgebras associated to the $N=1$ superconformal algebra, also known as the $N=1$ super Virasoro algebra. While this Lie superalgebra was used earlier by physicists, the corresponding VOAs were introduced by Kac and Wang \cite{KW2}, and by Barron \cite{B1, B2} shortly afterwards.
The universal $N=1$ super Virasoro VOAs are strongly and freely generated by a Virasoro field of some central charge $c$ together with a single odd field of conformal weight $\frac{3}{2}$.
For a discrete set of central charges, namely for
\begin{align} \label{cpq}
c=c^\NS_{p,q}:= \frac{15}{2}-3\left(\frac{p}{q}+\frac{q}{p}\right)
\end{align} 
for $p,q\in \ZZ_{\geq 2}$ such that $p-q\in 2\mathbb{Z}$ and $\gcd(\frac{p-q}{2},q)=1,$ 
Adamovi\'{c} \cite{A1} showed that the simple quotients of the universal $N=1$ superconformal VOAs are strongly  rational. Adamovi\'{c} also classified the simple modules for these rational VOAs and showed that they correspond to the $N=1$ superconformal minimal models of physics. Later, Huang and Milas studied the tensor category structure on the representation categories of the rational $N=1$ superconformal VOAs \cite{HM1, HM2}. Another proof of their rationality and of the classification of their simple modules was given in \cite{BMRW}. Beyond the rational minimal models, the existence of tensor structure on module categories for the $N=1$ superconformal algebra has remained open, though some results on fusion rules were obtained in \cite{M} for $c=\frac{3}{2}$ and in \cite{AM-super-trip} for $c=\frac{15}{2}-3(p+\frac{1}{p})$, $p\in\ZZ_{\geq 3}$ odd.

In the current work, we apply the logarithmic tensor category theory of Huang, Lepowsky, and Zhang \cite{HLZ1}-\cite{HLZ8} to the universal $N=1$ super Virasoro VOA $\cS(c,0)$ for any central charge $c$. These vertex operator superalgebras are  neither rational nor $C_2$-cofinite, as they admit infinitely many simple modules as well as non-semisimple modules. The first problem to solve is the mere existence of tensor category structure on a suitable category of $\cS(c,0)$-modules. That is, we need a category of $\cS(c,0)$-modules that satisfies the rather extensive list of sufficient conditions in \cite{HLZ1}-\cite{HLZ8} for the tensor category construction there to work. In the series of papers \cite{CHY, CJORY,CY,  Mc2}, we have found improved sufficient conditions that guarantee a tensor category structure on the category of $C_1$-cofinite modules for a VOA. Here,
since we are interested in modules for vertex operator \textit{super}algebras, we first ensure that such conditions also apply  in the superalgebra setting:
\begin{theorem} [Theorem \ref{thm:vosa-btc}]\label{intro:vosa-btc}
Let $\cC_1(V)$ be the category of $C_1$-cofinite grading-restricted generalized modules for a vertex operator superalgebra $V$. If $\cC_1(V)$ is closed under contragredient modules, then $\cC_1(V)$ admits the braided tensor category structure of \cite{HLZ1}-\cite{HLZ8}. In particular, $\cC_1(V)$ is a braided tensor category if the following two conditions hold:
\begin{enumerate}
    \item The contragredient $W'$ of any simple $C_1$-cofinite $V$-module $W$ is $C_1$-cofinite.
    \item $\cC_1(V)$ is equal to the category of finite-length $V$-modules whose composition factors are $C_1$-cofinite.
\end{enumerate}
\end{theorem}
To apply this theorem to $\cS(c,0)$, we carefully study its category of $C_1$-cofinite modules in Section \ref{sec:tensor}. Our main result is:
\begin{theorem} [Theorem \ref{thm:mainthm}]\label{intro:mainthm}
The category $\cC_1(\cS(c,0))$ of $C_1$-cofinite grading-restricted generalized $\cS(c,0)$-modules is the same as the category $\cO_c^{\mathrm{fin}}$ of finite-length grading-restricted generalized $\cS(c,0)$-modules whose composition factors are not isomorphic to Verma modules or their parity reversals.
\end{theorem}
As a consequence, we obtain the existence of the Huang-Lepowsky-Zhang braided tensor category structure:
\begin{cor}[Corollary \ref{cor:tensor}]\label{intro-cor:tensor}
    For any central charge $c\in\CC$, the category $\cO_c^{\mathrm{fin}}$ of finite-length grading-restricted generalized $V$-modules with $C_1$-cofinite composition factors admits the braided tensor category structure of \cite{HLZ1}-\cite{HLZ8}.
\end{cor}
Having obtained vertex algebraic braided tensor categories of modules for the $N=1$ super Virasoro algebra, one next wants to understand fusion rules (decompositions of tensor products as direct sums of indecomposable modules) and rigidity (existence of duals in the tensor categorical sense). In principle, both could be computed directly, but often indirect methods turn out to be much more efficient. We use the fact that the vector space tensor product of $\cS(c,0)$ with the vertex operator superalgebra of one free fermion is a conformal extension of the tensor product of two Virasoro VOAs \cite{CGL}. In Proposition \ref{prop:decomp}, the precise decomposition of the extension is described for all central charges of the form $c^\NS(t) := \frac{15}{2} - 3(t+t^{-1})$ for $t \in \mathbb C \setminus \mathbb Q$.  
Then, due to the theory of VOA extensions in direct limit completions of vertex algebraic tensor categories \cite{CMY-completions}, fusion rules and rigidity for $\cS(c,0)$-modules are inherited from those of the Virasoro subalgebras \cite{CJORY}: 
\begin{theorem} [Theorems \ref{thm:t-irrational-simple} and \ref{thm:t-irrational-properties}]
    Let $c=c^{\NS}(t)$ for $t\in\CC\setminus \QQ$. Then:
    \begin{itemize}
    \item [(1)] The category $\cO_c^{\rm fin}$ is semisimple with simple objects $\cS_{r,s}$, for $r,s\in\ZZ_{\geq 1}$ such that $r-s\in 2\ZZ$, and their parity reversals, where $\cS_{r,s}$ is the simple highest-weight module for the $N=1$ super Virasoro algebra of central charge $c$ and lowest conformal weight $h_{r,s}^\NS(t):=\frac{r^2-1}{8}t-\frac{rs-1}{4}+\frac{s^2-1}{8}t^{-1}$. 
        \item[(2)] The category $\cO_{c}^{\rm fin}$ is rigid and slightly degenerate, that is, its M\"{u}ger center is semisimple with only two simple objects, namely $\cS(c, 0)$ itself and its parity reversal.
        \item[(3)] The fusion rules for simple modules in $\cO_c^{\rm fin}$ are given by \eqref{eqn:irr-fus-rules}.
        \end{itemize}
    \end{theorem}

    This theorem provides an essentially complete description of $\cO_{c^{\NS}(t)}^{\rm fin}$ for irrational $t$.
For rational $t$, we give complete results only in the special case $t=1$, that is, $c=\frac{3}{2}$, in Theorem \ref{thm:t=1-properties}. Using the fact that $\cS(\frac{3}{2}, 0)$ is the $SO(3)$-orbifold of the vertex operator superalgebra of three free fermions, we show that although $\cO_{3/2}^{\rm fin}$ is not semisimple, it is rigid and its simple objects generate a semisimple tensor subcategory with the same fusion rules as $\mathrm{Rep}\,\mathfrak{osp}(1\vert 2)$.
This recovers the fusion rules of simple $\cS(\frac{3}{2},0)$-modules calculated in \cite{M}.

For $t\in\QQ\setminus\lbrace 0,1\rbrace$, computing fusion rules and studying rigidity of $\cO_{c^\NS(t)}^{\rm fin}$ will require case-by-case analysis depending on whether $t$ is positive or negative, and depending on whether $t$ or $t^{-1}$ is an integer or not. We mainly leave such analysis to future work, contenting ourselves here with studying rigidity of the simple module $\cS_{2, 2}$ of lowest conformal weight $h^\NS_{2,2}(t)=\frac{3(t-1)^2}{8t}$. When $t$ is irrational, $\cS_{2,2}$ generates the entire category $\cO_{c^\NS(t)}^{\rm fin}$ under tensor products, and we expect the same when $t$ is rational. Thus, proving that $\cS_{2,2}$ is rigid should be the key first step to proving that all of $\cO_{c^\NS(t)}^{\rm fin}$ is rigid.

Proving rigidity in tensor categories of modules for a VOA is a true challenge in general. For the module $\cS_{2,2}$ in the tensor category $\cO_{c^\NS(t)}^{\rm fin}$, proving rigidity directly
amounts to computing a certain $4$-point correlation function explicitly. It is possible to use singular vectors that vanish in $\cS_{2,2}$ to derive a fourth-order BPZ-type differential equation for this correlation function,
but solving this differential equation directly is highly technical (see \cite{Na1} for recent work in this direction that appeared after the first version of this paper was completed and that was applied to the different but related super-triplet VOA).
 Fortunately, for irrational $t$, we know how the tensor product of $\cS_{2,2}$ with the vertex operator superalgebra of one free fermion decomposes as a module for two commuting Virasoro VOAs (see Theorem \ref{thm:t-irrational-properties}(4)).
 Using this decomposition, we can express the required super Virasoro correlation function as a linear combination of products of two Virasoro correlation functions.
 These Virasoro correlation functions satisfy second-order BPZ differential equations that can be solved explicitly in terms of hypergeometric functions, and hence we get the $N=1$ super Virasoro correlation function explicitly for irrational $t$. Extending to rational $t$ by analytic continuation, we can then prove that $\cS_{2,2}$ is rigid for almost all $t$:
\begin{theorem}\textup{(Theorem \ref{thm:S22-rigid})}\label{intro-thm:S22-rigid}
    Assume that $t^{\pm1}\notin\ZZ_{\leq 0}$ and that $t\neq \frac{p}{q}$ for any $p,q\in\ZZ_{\geq 2}$ such that $p-q\in 2\ZZ$ and $\gcd(\frac{p-q}{2},q)=1$. Then $\cS_{2,2}$ is rigid and self-dual in $\cO_{c^\NS(t)}^{\mathrm{fin}}$.
\end{theorem}

For the values of $t$ excluded in this theorem, note that for $p,q\in\ZZ_{\geq 2}$ such that $p-q\in 2\ZZ$ and $\gcd(\frac{p-q}{2},q)=1$, $c^\NS(\frac{p}{q})=c^\NS_{p,q}$ is a rational $N=1$ superconformal minimal model central charge. Thus the tensor category of modules for the rational simple quotient of $\cS(c^\NS_{p,q},0)$ is rigid, but $\cO_{c^\NS_{p,q}}^{\rm fin}$ will not be since $\cS(c^\NS_{p,q},0)$ is not simple and thus not self-contragredient as a module for itself.
For $c=c^\NS(t)$, $t\in\ZZ_{\leq -1}$, we expect that $\cO^{\rm fin}_{c}$ will be rigid. But extra work will be required to prove that $\cS_{2,2}$ is rigid in this case, since it is not clear how to get a coevaluation map $\mathcal{S}(c,0)\rightarrow\cS_{2,2}\boxtimes\cS_{2,2}$. This is related to the fact that when $t\in\ZZ_{\leq -1}$, $\cS_{2,2}\boxtimes\cS_{2,2}$ might contain direct summands with negative integer lowest conformal weights.

Note that both our fusion rule and rigidity results use the fact that the tensor product of $\cS(c, 0)$ with the vertex operator superalgebra of one free fermion is a conformal extension of a specific type of two Virasoro VOAs. This extension was found in \cite{CGL} as a particularly nice example of isomorphisms of corner VOAs; see \cite{CG, FG} for the much more general picture and its connection to the quantum geometric Langlands program. 
Such correspondences of seemingly different VOAs are useful because they allow one to transport structure from the representation theory of one of the VOAs involved to another one.
As another example, a correspondence between the affine VOA of $\mathfrak{sl}_2$ and the $N=2$ superconformal algebra was recently instrumental in getting braided tensor category structure on weight modules of the affine VOA of $\mathfrak{sl}_2$ at any admissible level \cite{C}. In two-dimensional CFT, correspondences are usually exploited to obtain relations between correlation functions \cite{HS}, and this is exactly what we do here to prove the rigidity of $\cS_{2,2}$ in $\cO_c^{\rm fin}$. 

We conclude this introduction with a discussion of future directions, motivated by recent progress on tensor categories of modules for the Virasoro algebra. In this paper, we have achieved the same (and more) for the $N=1$ super Virasoro algebra as was achieved in \cite{CJORY} for the Virasoro algebra. In \cite{CJORY}, it was shown that the category of $C_1$-cofinite modules for the universal Virasoro VOA at any central charge is a vertex algebraic braided tensor category, and this result has enabled a series of works \cite{MY1, MY2, MS, LMY} exploring the detailed tensor structure of these categories at various central charges. These results on Virasoro tensor categories provide a rigorous mathematical foundation for the study of logarithmic minimal models in CFT, and based on Corollary \ref{intro-cor:tensor} and Theorem \ref{intro-thm:S22-rigid}, similar results on the detailed tensor structure of $\cO_{c^\NS(t)}^{\rm fin}$ for $t\in\QQ^\times$ should be possible. In particular, we conjecture that $\cO_{c^\NS(t)}^{\rm fin}$ is rigid for all $t\in\QQ^\times$ except for those giving the rational minimal model central charges $c^\NS_{p,q}$. It should also be possible to compute fusion rules in $\cO_{c^\NS(t)}^{\rm fin}$, $t\in\QQ^\times$, and construct logarithmic $\cS(c^\NS(t),0)$-modules that would be projective in a suitable subcategory of $\cO_{c^\NS(t)}^{\rm fin}$.

It would be especially interesting to explore the super Virasoro central charges $c^\NS_m:=c^\NS(2m+1)$, $m\in\ZZ_{\geq 1}$. The structure of $\cO^{\rm fin}_{c^\NS_m}$ should be analogous to that of the tensor category of $C_1$-cofinite modules for the Virasoro algebra at central charge $c_{p,1}:=13-6p-6p^{-1}$, $p\in\ZZ_{\geq 2}$. This Virasoro tensor category has a natural tensor subcategory containing all simple objects \cite{MY1} that is equivalent to the category of finite-dimensional weight modules for Lusztig's big quantum group of $\mathfrak{sl}_2$ at a $2p$th root of unity \cite{GN}. The Virasoro VOA at central charge $c_{p,1}$ also has two interesting extensions: the triplet VOA $\cW(p)$, which contains the Virasoro algebra as an $SO(3)$-orbifold, and the singlet VOA $\cM(p)$, which is the $U(1)$-orbifold of $\cW(p)$.
The triplet algebra $\cW(p)$ is non-rational and $C_2$-cofinite and has a non-semisimple modular tensor category of representations \cite{TW, GN} which is equivalent to a quasi-Hopf modification of the category of finite-dimensional representations for the small quantum group of $\mathfrak{sl}_2$ at a $2p$th root of unity \cite{GN, CLR1, CLR2}. The singlet algebra $\cM(p)$ is not rational or $C_2$-cofinite but has a rigid braided tensor category of representations \cite{CMY-singlet, CMY-fullsinglet} which is equivalent to the category of finite-dimensional weight modules for the unrolled small quantum group of $\mathfrak{sl}_2$ at a $2p$th root of unity \cite{CLR2, GN}.

It turns out that $\cS(c^\NS_m,0)$ has super-triplet and super-singlet extensions $\cS\cW(m)$ and $\cS\cM(m)$, analogous to $\cW(p)$ and $\cM(p)$, such that $\cS\cW(m)$ is $C_2$-cofinite but non-rational \cite{AM-super-trip}. Since the first version of this paper was completed, Nakano \cite{Na1} has given a description of the category of $\cS\cW(m)$-modules, including structure of projective modules, fusion rules, and rigidity. Note that while we here establish braided tensor structure on the category $\ocfin$ of $C_1$-cofinite $\cS(c,0)$-modules for every central charge $c$, the super-triplet extensions studied in \cite{Na1} occur only at central charge $c^{\NS}_m$ for $m\in\ZZ_{\geq 1}$, and because $\cS\cW(m)$ is $C_2$-cofinite, the category of $\cS\cW(m)$-modules already admits braided tensor structure due to \cite{HuC2}. 
Perhaps the most interesting difference between our work here and that in \cite{Na1} is in the proofs of rigidity. Here, we prove that the basic simple module $\cS_{2,2}$ in $\ocfin$ is rigid for almost all $c=c^{\NS}(t)$ (including $c^{\NS}_m$) using analytic continuation of $N=1$ super Virasoro correlation functions that are computed in the $t\notin\QQ$ case, while in \cite{Na1}, rigidity of the analogous simple $\cS\cW(m)$-module is proved by directly finding connection formulas for the above-mentioned fourth-order BPZ-type differential equation. 
Probably the direct methods of \cite{Na1} could also be used to prove rigidity of $\cS_{2,2}$ in $\ocfin$, although this would be complicated since these methods involve technical analysis of Dotsenko-Fateev integrals. On the other hand, our results on $\ocfin$ here at $c=c^{\NS}_m$ could also be combined with the VOA extension theory of \cite{CKM-exts} to recover the results of \cite{Na1} on $\cS\cW(m)$. 


Nakano also showed in \cite{Na1} that as abelian categories, the category of $\cS\cW(m)$-modules is equivalent to the category of modules for the small quantum group of $\mathfrak{sl}_2$ at a $(2m+1)$st root of unity, thus making important progress towards the conjecture in \cite{AM-super-trip} that these categories are tensor equivalent.
We expect that a detailed understanding of the tensor structure of $\ocfin$ at $c=c^\NS_m$ would enable a full proof of this conjectural tensor equivalence and would in addition pave the way for a  complete description of the representation category of the super-singlet $\cS\cM(m)$, similar to the work in  \cite{CMY-singlet, CMY-fullsinglet} for the singlet algebra $\cM(p)$. It would then also be interesting to find quantum group correspondences for $\cS\cM(m)$ and $\cS(c^\NS_m,0)$.

It would also be interesting to study the Ramond sector of parity-twisted $\cS(c,0)$-modules at any central charge $c$. This will be challenging, however, because our proof of Theorem \ref{intro:mainthm} heavily uses the embedding diagrams of Verma modules for the $N=1$ super Virasoro algebra from \cite{IK3}, and these diagrams are more complicated in the Ramond sector (see \cite[Theorem 4.4]{IK3}). Especially, there are non-injective homomorphisms between some Ramond sector Verma modules, which would cause serious trouble in the proof of Proposition \ref{desclarge} below. Moreover, since Ramond sector modules are not ordinary untwisted $\cS(c,0)$-modules, we cannot apply Theorem \ref{intro:vosa-btc} to get braided tensor category structure in this case.
Indeed, in view of \cite{McR-G-equiv}, we expect that $\cO^{\rm fin}_c$ together with a suitable subcategory of parity-twisted $\cS(c,0)$-modules forms a braided $\ZZ/2\ZZ$-crossed tensor category rather than a braided tensor category. At central charge $c^\NS_m$ for $m\in\ZZ_{\geq 1}$, we also expect $\ZZ/2\ZZ$-crossed tensor structure involving the twisted modules for the super-triplet algebra $\cS\cW(m)$ classified in \cite{AM-twisted}.

Finally, we note that the literature is full of examples of VOAs that are conformal extensions of the $N=1$ super Virasoro algebra. These include supersymmetric affine VOAs \cite{KT}, $W$-superalgebras obtained via quantum Hamiltonian reduction associated to an odd nilpotent element in an $\mathfrak{osp}_{1|2}$-subalgebra \cite{KRW, MRS, LSS} (which especially include the $N=2, 3, 4$ superconformal algebras), the algebras of global sections of chiral de Rham complexes \cite{BHS}, and more recently VOAs coming from 4d/2d-correspondences \cite{BMR}. Usually these conformal extensions are not objects of  $\cO^{\rm fin}_{c}$, but it would be very interesting to explore special examples that are, such as the already-mentioned super-singlet and super-triplet algebras.


Somewhat more generally, one could look for VOAs that are conformal extensions of an $N=1$ super Virasoro algebra tensored with a Heisenberg VOA, and that decompose into $C_1$-cofinite modules for the tensor product VOA. These would be analogous to certain special subregular $W$-algebras of $\mathfrak{sl}_n$ and principal $W$-superalgebras of $\mathfrak{sl}_{n|1}$ studied in \cite{C1, CRW, ACGY, ACKR}, which are conformal extensions of singlet VOAs tensored with a Heisenberg algebra.
One expects similar extensions for the $N=1$ super Virasoro algebra, namely subregular $W$-algebras of $\mathfrak{so}_{2n+1}$ and principal $W$-superalgebras of $\mathfrak{osp}_{2|2n}$ as conformal extensions of the super-singlet algebras times a Heisenberg VOA. For example the $p=3$ super-singlet tensored with a Heisenberg conformally
extends to $L_{-2/3}(\mathfrak{sl}_2)$ \cite{ACR}.
It would be interesting to establish and further explore these possible extensions. 


\section{Preliminaries}

In this section, we recall the $N=1$ super Virasoro algebra and its Verma modules, the vertex operator superalgebras associated to the $N=1$ super Virasoro algebra, and the Huang-Lepowsky-Zhang construction of braided tensor categories of modules for a vertex operator (super)algebra.

\subsection{The \texorpdfstring{$N=1$}{N=1} super Virasoro algebra and its representations}
We introduce the $N=1$ super Virasoro (Neveu-Schwarz) Lie superalgebra and its highest weight modules following the exposition in \cite{IK2}.
\begin{defn} \label{Liedef}
The $N=1$ \textit{super Virasoro algebra} (or \textit{Neveu-Schwarz algebra}) $\NS$ is the Lie superalgebra 
\begin{equation}
\NS: = \bigoplus_{n \in \ZZ}\CC L_n \oplus \bigoplus_{m \in \frac{1}{2}+ \ZZ}\CC G_m \oplus \CC \underline{c},
\end{equation}
where the parity of the basis elements is given by
\begin{align*}
|L_n| & = |\underline{c}| = 0,\; n \in \mathbb{Z} \\
|G_m| & =1,\; m \in \frac{1}{2}+\mathbb{Z},
\end{align*}
and the (anti)-commutation relations of the basis elements are given by
\begin{align*}
[L_m,L_n] & =(m-n)L_{m+n}+\frac{m^3-m}{12}\d_{m+n,0}\underline{c},\\
 [G_m, L_n] & =\left(m - \frac{n}{2}\right)G_{m+n},\\
\{G_m,G_n\} & =2L_{m+n}+\frac{1}{3}\left(m^2-\frac{1}{4}\right)\d_{m+n,0}\underline{c},\\
[\NS, \underline{c}] & =0.
\end{align*}
\end{defn}

Fix a triangular decomposition $\NS=\NS^-\oplus \NS^0\oplus \NS^+$, where
\begin{align*}
\NS^\pm & :=\bigoplus_{n\in \ZZ_{\geq 1}} \mathbb{C}L_{\pm n}\oplus \bigoplus_{ m\in \frac{1}{2}+\ZZ_{\geq 0}}  \mathbb{C}G_{\pm m},\\
\NS^0 & := \mathbb{C}L_0 \oplus \mathbb{C}\underline{c},
\end{align*}
We also define the non-negative subalgebra $\NS^{\geq0}:=\NS^+\oplus \NS^{0}$. For $(c,h)\in \mathbb{C}^2$,
the Verma module $M^\NS(c,h)$ of highest weight $(c,h)$ is given by
\begin{align*} 
M^\NS(c,h):=\mathcal{U}(\NS)\otimes_{\mathcal{U}\left(\NS^{\geq 0}\right)} \mathbb{C}\vac_{c,h},
\end{align*}
where $\mathcal{U}(\cdot)$ denotes the universal enveloping superalgebra and $\mathbb{C}\vac_{c,h}$ is the $1$-dimensional $\NS^{\geq 0}$ module determined by
\begin{equation*}
\NS^+.\vac_{c,h}=0,\;\;\; L_0. \vac_{c,h}=h\vac_{c,h}, \;\;\; \underline{c}.\vac_{c,h}=c\vac_{c,h}.
\end{equation*}
The operator $L_0$ acts semisimply on $M^\NS(c,h)$ with eigenspace decomposition
\begin{align*} 
M^\NS(c,h)=\bigoplus_{i=0}^{\infty} M^\NS(c,h)_{({\frac{i}{2}})},
\end{align*}
where $L_0$ acts by the scalar $h+\frac{i}{2}$ on $M^\NS(c,h)_{({\frac{i}{2}})}$. Every highest-weight module with highest weight $(c,h)$ is a quotient of the Verma module $M^\NS(c,h)$. We denote by $J^\NS(c,h)$ the (possibly $0$) maximal proper $\mathbb{Z}/2\ZZ$-graded submodule of $M^\NS(c,h)$ and by $L^\NS(c,h)$ its irreducible quotient $L^\NS(c,h):=M^\NS(c,h)/J^\NS(c,h)$.

For $t \in \mathbb{C}\setminus\{0\}$ and $r,s \in \mathbb{Z}$, set
\begin{align} \label{ct}
c^\NS(t) & := \frac{15}{2}-3(t+t^{-1}),\\
h^\NS_{r,s}(t) & :=\frac{1}{8}(r^2-1)t-\frac{1}{4}(rs-1)+\frac{1}{8}(s^2-1)t^{-1}.\label{hrs}
\end{align}
Note that $h^\NS_{r,s}=h^\NS_{-r,-s}$ for all $r,s\in \ZZ$.
From the Kac-Wakimoto determinant formula, given in equation $(6.2)$ in \cite{KW1} (see also Theorem 3.1 and Lemma 4.1 in \cite{IK3}), we have the following result:
 \begin{prop}[\cite{KW1}] \label{Redu}
The Verma module $M^\NS(c^\NS(t), h)$ is reducible if and only if $h=h^\NS_{r,s}(t)$ for some $r,s\in \mathbb{Z}_{\geq 1}$ such that $r-s\in 2\mathbb{Z}$. 
\end{prop}
This gives us, for any fixed central charge $c\in \mathbb{C}$, the complete set $H_c$ of the conformal weights $h\in \CC$ such that $M^\NS(c,h)$ is reducible, that is,
\begin{align}
H_c=\{h\in \CC\mid \textrm{$M^\NS(c,h)$ is not irreducible} \}.   \label{Hc}
\end{align}
Recall that a \textit{singular vector} in a Verma module is an $L_0$-eigenvector $\mathbf{w}$ such that $L_n\mathbf{w}=G_{m-\frac{1}{2}}\mathbf{w}=0$ for all $m,n\in\ZZ_{\geq 1}$. The following theorem controls the existence and shape of singular vectors in $M^\NS(c,h)$:
\begin{theorem}[\cite{As}] \label{AsTh}
For any $n\in \frac{1}{2}\mathbb{Z}_{\geq 0}$, there is at most one singular vector $\mathbf{w}\in M^\NS(c,h)_{(n)}$ up to scaling. If such a $\bf{w}$ exists, then it is given by
\begin{align}
{\bf w}=(G_{-\frac{1}{2}})^{2n}{\vac}_{c,h}+\sum_{\substack{2i_{1}+\cdots +2i_{k}+j_{1}+\cdots +j_l= 2n,\\ i_k\geq \cdots \geq i_{1} \geq 2   \text{ with $i_s$ even}   \\  or\; j_l \geq \cdots  \geq j_1 \geq 1 \;\text{with }j_l\geq 3\\  \text{ and $j_s$ odd}}} P^{(n)}_{i_{1}, \cdots i_{k}, j_1, \cdots, j_l}(c,h) 
L_{-i_k}\cdots L_{-i_1} G_{-\frac{j_{l}}{2}} \cdots G_{-\frac{j_1}{2}}\vac_{c,h} \label{svector}
\end{align}
up to scaling, where the $P^{(n)}_{i_{1}, \cdots i_{k}, j_1, \cdots, j_l}(c,h)$ are polynomials in $c$ and $h$. 
\end{theorem}
\begin{remark}
For integer values of the degree $n$, the formula \eqref{svector} becomes
\begin{equation*}
{\bf w}=(L_{-1})^{n}{\vac}_{c,h}+\sum_{\substack{2i_{1}+\cdots +2i_{k}+j_{1}+\cdots +j_l= 2n,\\ i_k\geq \cdots \geq i_{1} \geq 2   \text{ with $i_s$ even }  \\  \text{or}\; j_l \geq \cdots  \geq j_1 \geq 1 \;\text{with }j_l\geq 3\\  \text{ and $j_s$ odd}}} P^{(n)}_{i_{1}, \cdots i_{k}, j_1, \cdots, j_l}(c,h) 
L_{-i_k}\cdots L_{-i_1} G_{-\frac{j_{l}}{2}} \cdots G_{-\frac{j_1}{2}}\vac_{c,h}
\end{equation*}
because $G_{-\frac{1}{2}}^2=L_{-1}$ on $M^\NS(c,h)$. But for half-integer degree $n$, the more general formula \eqref{svector} is needed to describe the degree $n$ singular vector.
\end{remark}

As with the Virasoro algebra, any non-zero homomorphism between $N=1$ super Virasoro Verma modules is injective (see \cite[Proposition 3.3]{IK3}). For brevity, we generally denote the Verma module $M^\NS(c^\NS(t), h^\NS_{r,s}(t))$ by $M^\NS_{r,s}$ in the rest of the paper, when $t$ is understood.

\begin{prop} [\cite{As}, \cite{IK3}] \label{vermadiag}
The following embedding diagrams show all embeddings of $N=1$ super Virasoro Verma modules $M^\NS_{r,s}$ for $r,s\in\ZZ$:
\begin{enumerate}
\item
If $t=\frac{p}{q}\in \mathbb{Q}_{>0}$ with $p-q\in 2\mathbb{Z}$, {\rm gcd}$(\frac{p-q}{2},q)=1$:
\begin{enumerate}
\item 
For $0<r<q$ and $0<s<p$, we have the embedding diagram
\begin{equation*}
\begin{tikzpicture}[->,>=latex,scale=1.4]
\node (b0) at (-0.5,1/2){$M^\NS_{r,s}$};
\node (b1) at (1, 0){$M^\NS_{2q-r,s}$};
\node (a1) at (1, 1){$M^\NS_{q+r,p-s}$};
\node (b2) at (3, 0){$M^\NS_{3q-r,p-s}$};
\node (a2) at (3, 1){$M^\NS_{2q+r,s}$};
\node (b3) at (5, 0){$M^\NS_{4q-r, s}$};
\node (a3) at (5, 1){$M^\NS_{3q+r, p-s}$};
\node (b4) at (7, 0){$M^\NS_{5q-r,p-s}$};
\node (a4) at (7, 1){$M^\NS_{4q+r,s}$};
\node (b5) at (9, 0){$\cdots$};
\node (a5) at (9, 1){$\cdots$};
\draw[] (b1) -- node[left]{} (b0);
\draw[] (a1) -- node[left]{} (b0);
\draw[] (b2) -- node[left]{} (b1);
\draw[] (b2) -- node[left]{} (a1);
\draw[] (a2) -- node[left]{} (b1);
\draw[] (a2) -- node[left]{} (a1);
\draw[] (b3) -- node[left]{} (b2);
\draw[] (b3) -- node[left]{} (a2);
\draw[] (a3) -- node[left]{} (b2);
\draw[] (a3) -- node[left]{} (a2);
\draw[] (b4) -- node[left]{} (b3);
\draw[] (b4) -- node[left]{} (a3);
\draw[] (a4) -- node[left]{} (b3);
\draw[] (a4) -- node[left]{} (a3);
\draw[] (b5) -- node[left]{} (b4);
\draw[] (b5) -- node[left]{} (a4);
\draw[] (a5) -- node[left]{} (b4);
\draw[] (a5) -- node[left]{} (a4);
\end{tikzpicture}
\end{equation*}
\item
For $r=q$ and $0<s<p$, we have the embedding diagram
\begin{equation*}
M^\NS_{q,s} \longleftarrow M^\NS_{2q,p-s} \longleftarrow M^\NS_{3q,s} \longleftarrow M^\NS_{4q,p-s} \longleftarrow M^\NS_{5q,s} \longleftarrow \cdots
\end{equation*}
\item For $0<r<q$ and $s=p$, we have the embedding diagram 
\begin{equation*}
M^\NS_{r,p} \longleftarrow M^\NS_{2q+r,p} \longleftarrow M^\NS_{4q-r,p} \longleftarrow M^\NS_{4q+r,p} \longleftarrow M^\NS_{6q-r,p} \longleftarrow \cdots
\end{equation*}
\item For $(r,s)=(iq,p)$ where $i=1,2$, we have the embedding diagram
\begin{equation*}
M^\NS_{iq,p}  \longleftarrow M^\NS_{(i+2)q,p} \longleftarrow M^\NS_{(i+4)q,p} \longleftarrow M^\NS_{(i+6)q,p} \longleftarrow M^\NS_{(i+8)q,p},\longleftarrow \cdots
\end{equation*}
where the $i=2$ case only occurs if $\gcd(p,q)=2$.

\end{enumerate}

\item If $t=-\frac{p}{q}\in \mathbb{Q}_{<0}$ with $p-q \in 2\mathbb{Z}$, $\gcd(\frac {p-q}{2},q)=1$ then the embedding diagram for $M^\NS_{r,s}$, $r,s\in\ZZ_{\geq 1}$, has one of the following two forms:
If $r$ and $s$ are not multiples of $q$ and $p$, respectively, then the diagram has the form
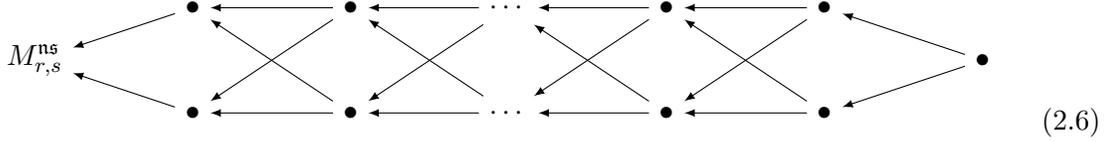
\begin{equation}
\begin{tikzpicture}[->,>=latex,scale=1.4]
\node (b0) at (-0.5,1/2){$M^\NS_{r,s}$};
\node (b1) at (1, 0){$\bullet$};
\node (a1) at (1, 1){$\bullet$};
\node (b2) at (2.5, 0){$\bullet$};
\node (a2) at (2.5, 1){$\bullet$};
\node (bm) at (4,0){$\cdots$};
\node (am) at (4,1) {$\cdots$};
\node (b3) at (5.5, 0){$\bullet$};
\node (a3) at (5.5, 1){$\bullet$};
\node (b4) at (7, 0){$\bullet$};
\node (a4) at (7, 1){$\bullet$};
\node (a5) at (8.5, 0.5){$\bullet$};
\draw[] (b1) -- node[left]{} (b0);
\draw[] (a1) -- node[left]{} (b0);
\draw[] (b2) -- node[left]{} (b1);
\draw[] (b2) -- node[left]{} (a1);
\draw[] (a2) -- node[left]{} (b1);
\draw[] (a2) -- node[left]{} (a1);
\draw[] (bm) -- node[left]{} (b2);
\draw[] (bm) -- node[left]{} (a2);
\draw[] (am) -- node[left]{} (b2);
\draw[] (am) -- node[left]{} (a2);
\draw[] (b3) -- node[left]{} (bm);
\draw[] (b3) -- node[left]{} (am);
\draw[] (a3) -- node[left]{} (bm);
\draw[] (a3) -- node[left]{} (am);
\draw[] (b4) -- node[left]{} (b3);
\draw[] (a4) -- node[left]{} (b3);
\draw[] (b4) -- node[left]{} (a3);
\draw[] (a4) -- node[left]{} (a3);
\draw[] (a5) -- node[left]{} (b4);
\draw[] (a5) -- node[left]{} (a4);
\end{tikzpicture}
\end{equation}
If $r$ is a multiple of $q$ or $s$ is a multiple of $p$, then the diagram has the form
\begin{equation}
{M^\NS_{r,s}} \longleftarrow \bullet \longleftarrow \bullet \longleftarrow \bullet \longleftarrow \cdots \longleftarrow \bullet
\end{equation}
\item If $t\notin \mathbb{Q}$, then we have the embedding diagram $M^\NS_{r,s} \longleftarrow M^\NS_{-r, s}.$
 \end{enumerate}
 \end{prop}
\begin{remark}
    The restriction gcd$(\frac{p-q}{2}, q)=1$ above is taken to avoid repeating values of $t$. Namely, if gcd$(\frac{p-q}{2}, q)=m>1$, it follows that $m$ divides $p$ and then  $(p,q)$ and $(\frac{p}{m}, \frac{q}{m})$ yield the same value of $t=\frac{p}{q}.$
 \end{remark}
 \begin{remark}
    In case (1) of Proposition \ref{vermadiag}, we have used the symmetries $h^\NS_{r,s}(\frac{p}{q})=h^\NS_{-r,-s}(\frac{p}{q})$ and $h^\NS_{r,s}(\frac{p}{q})=h^\NS_{q+r, p+s}(\frac{p}{q})$ for $r,s\in \ZZ$ to rewrite the embedding diagrams from \cite{IK3}. As a result, it is clear that every Verma module $M_{r,s}^\NS$ appears in exactly one of the embedding diagrams in case (1).
\end{remark}

The Verma module embedding diagrams contain all information about the structure of $M^\NS_{r,s}$ since every submodule of a Verma module is generated by its singular vectors by \cite[Theorem D]{As} and \cite[Theorem 4.2]{IK3}. The following results then follow easily from the embedding diagrams:
\begin{cor} \label{coro}
\begin{enumerate}
    \item Every non-zero $\NS$-submodule of a Verma module is either a Verma module or the sum of two Verma submodules.
    \item Every highest-weight $\NS$-module which is isomorphic to a quotient of a Verma module by a non-zero submodule has finite length.
    \item Every Verma module $M^\NS(c,h)$ appearing in the embedding diagram of a reducible Verma module has conformal weight $h\in H_c$ except possibly for the socle which, if non-zero, is of the form $M^\NS(c,h')$ for $h'\notin H_c$.   
    \item If a submodule $J$ of a Verma module is the sum of two Verma submodules, then the two singular vectors which generate $J$ have a common descendant.    
    \end{enumerate}
\end{cor}

We say that an $\NS$-module $W$ is \textit{restricted} if for any $w\in W$, $L_n w =0$ and $G_{m+\frac{1}{2}} w=0$ for $m,n\in\ZZ$ sufficiently positive.

\subsection{Vertex operator superalgebras and their modules}

For several slightly different definitions of vertex operator superalgebra, see for example \cite{DL, KW2, Xu, CKL}. Here we present what is essentially the definition of \cite{CKL}:
\begin{defn} \label{SVOADef}
A \textit{vertex operator superalgebra} is a $\frac{1}{2}\mathbb{Z}$-graded superspace $V= V^{\bar{0}}\oplus V^{\bar{1}}=\bigoplus_{n\in \frac{1}{2}\mathbb{Z}}V_{(n)}$, together with an even linear map called the vertex operator,
\begin{align*}
Y: V\otimes V & \rightarrow V((x))\\
 u\otimes v & \mapsto Y(u,x)v=\sum_{n\in \mathbb{Z}}u_nv\,x^{-n-1}, 
\end{align*}
and two distinguished vectors $\one \in V_{(0)}\cap V^{\bar{0}}$ called the vaccum vector and $\om \in V_{(2)}\cap V^{\bar{0}}$ called the conformal vector. These data satisfy the following axioms:
\begin{enumerate}
\item For all $n\in\frac{1}{2}\ZZ$, $\dim V_{(n)}<\infty$ and $V_{(n)}=(V^{\bar{0}}\cap V_{(n)})\oplus (V^{\bar{1}}\cap V_{(n)})$. Moreover, $V_{(n)}=0$ for all sufficiently negative $n\in\frac{1}{2}\ZZ$.

\item $Y(\one,x)=\id_V$, and for any $v \in V$, $Y(v,x)\one \in V[[x]]$ with constant term $v$.
\item The Jacobi identity: For any parity homogeneous $u,v\in V$,
\begin{align*}
 x_0^{-1}\delta\bigg(\frac{x_1-x_2}{x_0}\bigg)Y(u,x_1)Y(v, x_2)
&-(-1)^{|u||v|}x_0^{-1}\delta\bigg(\frac{x_2-x_1}{-x_0}\bigg)Y(v, x_2)Y(u,x_1)\\
& \qquad = x_2^{-1}\delta\bigg(\frac{x_1-x_0}{x_2}\bigg)Y(Y(u,x_0)v, x_2),
\end{align*}
where $\delta(x)=\sum_{n\in\ZZ} x^n$.

\item If we write $Y(\omega,x)=\sum_{n \in \mathbb{Z}}L_n\,x^{-n-2}$, then
\begin{align*}
[L_m,L_n]=(m-n)L_{m+n}+\frac{m^3-m}{12}\delta_{m+n,0}c\cdot\id_V,
\end{align*}
where $c\in \mathbb{C}$ is called the central charge of $V$. Moreover, $L_0v = nv$ for any $n \in \frac{1}{2}\ZZ$ and $v\in V_{(n)}$; in this case, we say that $n$ is the conformal weight of $v$.

\item The $L_{-1}$-derivative property:  $Y(L_{-1}v,x) = \frac{d}{dx} Y(v,x)$ for $v\in V$.
\end{enumerate}
\end{defn}

\begin{remark}
In this paper, we will mainly consider vertex operator superalgebras such that $V^{\bar{i}}=\bigoplus_{n\in\frac{i}{2}+\ZZ} V_{(n)}$ for $i=0,1$. Such vertex operator algebras are said to have ``correct statistics'' in \cite{CKL}.
\end{remark}

Let $V$ be a vertex operator superalgebra. We next recall the definitions of various types of modules for $V$:
\begin{itemize}
\item A \textit{weak $V$-module} is a superspace $W=W^{\bar{0}}\oplus W^{\bar{1}}$ with an even vertex operator map
\begin{align*}
Y_W:  V\otimes W & \rightarrow W((x))\\
v\otimes w &\mapsto Y_W(v,x)w =\sum_{n \in \ZZ}v_nw\, x^{-n-1}
\end{align*}
satisfying the following properties:
\begin{enumerate}
\item $Y_W(\one, x) = \id_W$.
\item The Jacobi identity: For any parity homogeneous $u,v\in V$,
\begin{align*}
 x_0^{-1}\delta\bigg(\frac{x_1-x_2}{x_0}\bigg)Y_W(u,x_1)Y_W(v, x_2) & -(-1)^{|u||v|}x_0^{-1}\delta\bigg(\frac{x_2-x_1}{-x_0}\bigg)Y_W(v, x_2)Y_W(u,x_1)\\
& = x_2^{-1}\delta\bigg(\frac{x_1-x_0}{x_2}\bigg)Y_W(Y(u,x_0)v, x_2).
\end{align*}

\item The $L_{-1}$-derivative property: $Y_W(L_{-1}v,x) = \frac{d}{dx} Y_W(v,x)$ for $v \in V$.

\end{enumerate}

\item A weak $V$-module $W$ is \textit{$\frac{1}{2}\NN$-gradable} if there exists a $\frac{1}{2}\NN$-grading $W=\bigoplus_{i=0}^\infty W\left(\frac{i}{2}\right)$
 such that for $v \in V$ of conformal weight wt $v$,
$$v_m\cdot W(n)\subset W({\rm wt} \  v+n-m-1)$$
for $m \in \ZZ$, $n\in\frac{1}{2}\NN$.

\item A weak $V$-module $W$ is a \textit{generalized  $V$-module} if it has a $\CC$-grading $W = \bigoplus_{h \in \CC}W_{[h]}$ such that:
\begin{enumerate}
    \item For any $h\in\CC$, $W_{[h]}=(W^{\bar{0}}\cap W_{[h]})\oplus (W^{\bar{1}}\cap W_{[h]})$.
    \item For any $h\in\CC$, $W_{[h]}$ is the generalized $L_0$-eigenspace with generalized eigenvalue $h$, where $L_0$ is the coefficient of $x^{-2}$ in $Y_W(\omega,x)$.
\end{enumerate}

\item A generalized $V$-module $W$ is \textit{lower bounded} if $W_{[h]}=0$ for $\mathrm{Re}(h)$ sufficiently negative.

\item A lower-bounded generalized $V$-module is \textit{grading restricted} if $\dim W_{[h]}<\infty$ for all $h\in\CC$.

\item A generalized module $W$ has \textit{finite length} if there is a filtration of generalized submodules $0 =W_0 \subset W_1 \subset \cdots \subset W_{n-1} \subset W_n = W$ such that $W_i/W_{i-1}$, $i = 1, \dots, n$ are irreducible grading restricted $V$-modules.
\end{itemize}
From now on, for brevity, we will use \textit{$V$-module} with no further qualification to mean a grading-restricted generalized $V$-module.

\begin{remark}\label{rem:half-N-grading}
Any lower-bounded generalized module $W$ for a vertex operator superalgebra $V$ is $\frac{1}{2}\NN$-gradable. Indeed, given $W=\bigoplus_{h\in \CC}W_{[h]}$,
let $I$ be the set of cosets $i\in\CC/(\frac{1}{2}\ZZ)$ such that $W_{[h]}\neq 0$ for some $h\in i$.
Then
\begin{align} \label{ngrading}
W=\bigoplus_{i\in I}\bigoplus_{n=0}^{\infty}W_{[h_i+\frac{n}{2}]},
\end{align}
where $h_i$ is the conformal weight with minimal real part in the coset $i\in I$. Each summand $W_i=\bigoplus_{n=0}^\infty W_{[h_i+\frac{n}{2}]}$ is a $V$-submodule of $W$. In particular, if $W\neq 0$ is indecomposable, then $|I|=1$, and if $W$ is finitely generated, then $I$ is finite. Now the decomposition \eqref{ngrading} implies that $W$ has a $\frac{1}{2}\NN$-grading  $W=\bigoplus_{n=0}^\infty W\left(\frac{n}{2}\right)$ where $W\left(\frac{n}{2}\right)=\bigoplus_{i\in I} W_{[h_i+\frac{n}{2}]}$.
\end{remark}

Following \cite[Definition 1.5]{KW2}, we say a vertex operator superalgebra $V$ is \textit{rational} if every $\frac{1}{2}\NN$-gradable $V$-module is a direct sum of irreducible $\frac{1}{2}\NN$-gradable $V$-modules. Moreover, $V$ is \textit{$C_2$-cofinite} if $\dim V/C_2(V)<\infty$, where 
    \[
    C_2(V) = \mathrm{span} \{u_{-2}v\mid u,v \in V\}.    
    \]

If $V$ is a vertex operator superalgebra and $W=\bigoplus_{h\in\CC} W_{[h]}$ is a generalized $V$-module, then $W'=\bigoplus_{h\in\CC} W_{[h]}^*$ is the graded dual superspace. If $W$ is lower bounded, then $W'$ admits the structure of lower-bounded generalized $V$-module, called the \textit{contragredient} of $W$, by a superalgebra generalization of the proof for vertex operator algebras in \cite[Section 5]{FHL}. Specifically, the vertex operator $Y_{W'}$ can be defined by
\begin{equation}\label{eqn:contra}
\langle Y_{W'}(v,x)w', w\rangle =(-1)^{|v||w'|} \langle w', Y_W(e^{xL_1}(e^{\pi i} x^{-2})^{L_0} v,x^{-1})w\rangle
\end{equation}
for $v\in V$, $w\in W$, and $w'\in W'$. Actually, there are several definitions of contragredient modules for a vertex operator superalgebra in the literature (see \cite[Remark 3.5]{CKM-exts} for a brief discussion), but all definitions yield isomorphic contragredient modules.

If $W_1$ and $W_2$ are (weak) modules for a vertex operator superalgebra $V$, there are two options for defining $V$-module homomorphisms $f: W_1\rightarrow W_2$. First, we can require $f$ to be an even linear map (that is, $f(W_1^{\bar{i}})\subset W_2^{\bar{i}}$ for $i=0,1$) such that 
\begin{equation*}
f(Y_{W_1}(v,x)w_1)=Y_{W_2}(v,x)f(w_1)
\end{equation*}
for $v\in V$ and $w_1\in W_1$. Alternatively, we can also allow odd linear maps $f$ such that
\begin{equation*}
    f(Y_{W_1}(v,x)w_1)=(-1)^{\vert v\vert} Y_{W_2}(v,x)f(w_1)
\end{equation*}
for parity-homogeneous $v\in V$. Then a $V$-module homomorphism would be any sum of an even and an odd homomorphism.

In this paper, to avoid excessive sign factors, we will only allow even $V$-module homomorphisms. This means we will need to distinguish between $V$-modules and their parity reversals. Namely, if $W$ is a (weak) $V$-module, then $(\Pi(W),Y_{\Pi(W)})$ is the (weak) $V$-module defined by $\Pi(W)^{\bar{i}}=W^{\overline{i+1}}$ for $i=0,1$, and $Y_{\Pi(W)}=Y_W$. The parity involution
\begin{equation*}
    P_W=\id_{W^{\bar{0}}}\oplus(-\id_{W^{\bar{1}}})
\end{equation*}
defines an odd $V$-module isomorphism $W\cong\Pi(W)$, but there is usually no even $V$-module isomorphism between $W$ and $\Pi(W)$. If $f: W_1\rightarrow W_2$ is an even $V$-module homomorphism, then $f$ also defines an even homomorphism from $\Pi(W_1)$ to $\Pi(W_2)$. Thus $\Pi$ is a functor, which we call the parity-reversal functor.

\begin{remark}
    If we allow both even and odd homomorphisms between $V$-modules, then the category of all weak $V$-modules is a $\Pi$-supercategory in the terminology of \cite{BE}, that is, a category enriched in vector superspaces and equipped with a parity-reversal functor. If we allow only even homomorphisms, then we get the ``underlying category'' of this supercategory, which is a $\Pi$-category in the terminology of \cite{BE}.
\end{remark}

\begin{remark}\label{rem:odd-to-even-morphism}
    If $f: W_1\rightarrow W_2$ is an odd $V$-module homomorphism, then $P_{W_2}\circ f: W_1\rightarrow\Pi(W_2)$ is an even $V$-module homomorphism. Thus we do not lose anything essential by allowing only even $V$-module homomorphisms.
\end{remark}

\subsection{Vertex operator superalgebras from the \texorpdfstring{$N=1$}{N=1} super Virasoro algebra}

We now recall the vertex operator superalgebras and their modules based on representations of the Neveu-Schwarz algebra. Neveu-Schwarz vertex operator superalgebras and their modules have been studied extensively in both the mathematics and physics literatures and were formulated mathematically by Kac-Wang \cite{KW2} and Barron \cite{B1, B2}.

To begin, $G_{-\frac{1}{2}}\vac_{c,0}$ is a singular vector of the Verma module $M^\NS(c,0)$ for any $c\in \mathbb{C}$. Then
\begin{align*}
\cS(c,0):=M^\NS(c,0)/\langle G_{-\frac{1}{2}}\vac_{c,0}\rangle
\end{align*}
has an $L_0$-eigenspace decomposition 
\begin{align*}
\cS(c,0)=\bigoplus_{n\in \frac{1}{2}\mathbb{Z}} \cS(c,0)_{(n)},
\end{align*}
where
$$\cS(c,0)_{(n)} = \textrm{span} \{L_{-i_{1}}\cdots L_{-i_{k}}G_{-j_{1}}\cdots G_{-j_{l}}\vac_{c,0}\mid  i_1+\cdots + i_k+j_1+\cdots +j_l = n\}.$$
Moreover, $\cS(c,0)$ is a vector superspace such that
$$|L_{-i_{1}}\cdots L_{-i_{k}}G_{-{j_{1}}}\cdots G_{-j_{l}}\vac_{c,0}|  =0\;\;\; (\textrm{respectively} \;1),$$ when $l$ is even (respectively odd).

Define the conformal vector $\omega: = L_{-2}\vac_{c,0}$ and Neveu-Schwarz vector $\tau: =G_{-\frac{3}{2}}\vac_{c,0}$ in $\cS(c,0)$, as well as the corresponding fields
\begin{align*}
L(x):=\sum_{n\in \mathbb{Z}}L_n\,x^{-n-2},\qquad G(x):= \sum_{n\in \mathbb{Z}}G_{n+\frac{1}{2}}\,x^{-n-2}.
\end{align*}
The vertex operator $Y: \cS(c,0) \rightarrow \End(\cS(c,0))[[x, x^{-1}]]$ is the unique linear map such that 
\begin{equation*}
Y(\vac_{c,0}, x)=\id_V,\qquad Y(\tau,x)=G(x), \qquad Y(\omega,x)=L(x),
\end{equation*}
and the Jacobi identity holds. In particular,
\begin{align}\label{operator}
&Y(L_{-i_{1}-2}\cdots L_{-i_{k}-2}G_{-{j_{1}}-\frac{3}{2}}\cdots G_{-j_{l}-\frac{3}{2}}\vac_{c,0},x) \nonumber\\
&\quad\quad\quad \quad\quad\quad  =\frac{1}{i_1!\cdots i_k!j_1!\cdots j_l!} :\partial^{i_1}L(x)\cdots \partial^{i_k} L(x) \partial^{j_1} G(x) \cdots\partial^{j_l} G(x):
\end{align}
for $i_1, \dots, i_k, j_1, \dots, j_l \in \ZZ_{\geq 0}$, where $k,l \in \ZZ_{\geq 0}$.
By \cite[Theorem~3.1]{KW2}, the map defined in \eqref{operator} gives $\cS(c,0)$ and its simple quotient $L^\NS(c,0)$ the structure of vertex operator superalgebras.

For fixed $c \in \mathbb{C}$, a weak $\cS(c,0)$-module is the same as a restricted $\NS$-module of central charge $c$. Thus the Verma and simple $\NS$-modules $M^\NS(c,h)$ and $L^\NS(c,h)$ are (grading-restricted generalized) $\cS(c,0)$-modules for all $h\in\CC$, and every irreducible $\cS(c,0)$-module is isomorphic to either $L^\NS(c,h)$ or its parity reversal $\Pi(L^\NS(c,h))$ for some $h\in\CC$. Both $L^\NS(c,h)$ and its parity reversal are generated by a vector $\one_{c,h}$ of conformal weight $h$; our convention will be that $L^\NS(c,h)$ denotes the module in which $\one_{c,h}$ is even, while $\Pi(L^\NS(c,h))$ denotes the module in which $\one_{c,h}$ is odd. These simple modules are self-contragredient,
\begin{equation*}
L^\NS(c,h)'\cong  L^\NS(c,h),\qquad \Pi(L^\NS(c,h))'\cong\Pi(L^\NS(c,h)),
\end{equation*}
because (using the $v=\omega$ case of \eqref{eqn:contra}) both $L^\NS(c,h)$ and $L^\NS(c,h)'$ are simple $\cS(c,0)$-modules with even lowest conformal weight space of $L_0$-eigenvalue $h$.

For a simple $\cS(c,0)$-module $W=\Pi^i(L^\NS(c,h))$, $i=0,1$, let $f: W\rightarrow W'$ be an $\cS(c,0)$-module isomorphism. Then we can define an even non-degenerate invariant bilinear form $\langle\cdot,\cdot\rangle$ on $W$ such that $\langle w',w\rangle =\langle f(w'),w\rangle$ for $w',w\in W$. Such a bilinear form is also symmetric:
\begin{lemma}\label{lem:bilinear-form-symmetric}
    If $\langle\cdot,\cdot\rangle$ is an even invariant bilinear form on a simple $\cS(c,0)$-module $W$, then $\langle\cdot,\cdot\rangle$ is symmetric.
\end{lemma}
\begin{proof}
Let $W\cong \Pi^i(L^\NS(c,h))$ for some $h\in\CC$ and $i=0,1$. We will prove that the restriction of $\langle\cdot,\cdot\rangle$ to the conformal weight space $W_{[h+n]}$ is symmetric by induction on $n\in\frac{1}{2}\ZZ_{\geq 0}$. The base case $n=0$ holds because $W_{[h]}$ is one dimensional.

For $n>0$, $W_{[h+n]}$ is spanned by vectors of the form $L_{-m}w'$ for $m\in\ZZ_{\geq 1}$, $w'\in W_{[h+n-m]}$ and $G_{-m}w'$ for $m\in\ZZ_{\geq 0}+\frac{1}{2}$, $w'\in W_{[h+n-m]}$. Thus it is enough to show $\langle L_{-m} w',w\rangle =\langle w,L_{m}w'\rangle$ and $\langle G_{-m}w',w\rangle =\langle w,G_{m}w'\rangle$
for all $w\in W_{[h+n]}$. Indeed, the $v=\omega,\tau$ cases of \eqref{eqn:contra} yield
\begin{equation*}
    \langle L_{-m} w', w\rangle =\langle w',L_m w\rangle,\qquad \langle G_{-m} w',w\rangle =-i(-1)^{\vert w'\vert}\langle w', G_m w\rangle,
\end{equation*}
so using the induction hypothesis,
\begin{align*}
    \langle L_{-m}w',w\rangle =\langle w',L_m w\rangle =\langle L_m w,w'\rangle =\langle w,L_{-m}w'\rangle
\end{align*}
as required, and
\begin{equation*}
    \langle G_{-m}w',w\rangle  =-i(-1)^{\vert w'\vert}\langle w',G_m w\rangle =-i(-1)^{\vert w'\vert}\langle G_m w,w'\rangle
    =(-1)^{\vert w\vert+\vert w'\vert+1}\langle w,G_{-m}w'\rangle.
\end{equation*}
This equals $\langle w, G_{-m}w'\rangle$ because $\langle\cdot,\cdot\rangle$ is even.

\end{proof}

\begin{nota}
Recall the notation $c^\NS(t)$ and $h^\NS_{r,s}(t)$ from \eqref{ct} and \eqref{hrs}. For brevity, and when $t$ is understood, we will use $\cS_{r,s}$ to denote the simple $\NS$-module $L(c^\NS(t), h_{r,s}^\NS(t))$.
\end{nota}

Recall from \eqref{cpq} the conformal weights $c_{p,q}^\NS$ for $p,q\in\ZZ_{\geq 2}$ such that $p-q\in 2\ZZ$ and $\gcd\left(\frac{p-q}{2},q\right)=1$. From the embedding diagrams for Verma $\NS$-modules, $\cS(c,0)\neq L^\NS(c,0)$ if and only if $c= c_{p,q}$ for some such $p,q$. Kac and Wang conjectured \cite[Conjecture~3.1]{KW2}, and Adamovi\'{c} proved, that the simple vertex operator superalgebras at these central charges are rational:
\begin{theorem}[\cite{A1}]
    The vertex operator superalgebra $L^\NS(c_{p,q}^\NS,0)$ is rational, and its irreducible modules are given by 
    \[
    \{L(c_{p,q}^\NS, h_{r,s}^\NS) \mid 0<r<p,\; 0<s<q,\; r-s \in 2\ZZ\}.
    \]
\end{theorem} 
 On the other hand, $\cS(c,0)$ (which is simple if and only if $c\neq c_{p,q}$) is never rational or $C_2$-cofinite (see \cite[Corollary 3.1]{KW2}). These non-rational vertex operator superalgebras will be the focus of this paper.

\subsection{Intertwining operators and tensor product modules}

To define the (fusion) tensor product of modules for a vertex operator superalgebra $V$, we need intertwining operators. In general, intertwining operators can be even or odd linear maps, or a sum of even and odd linear maps. But since we only consider even $V$-homomorphisms in this paper, we will also only consider even intertwining operators. The following is the natural superalgebra generalization of \cite[Definition 3.10]{HLZ2}:
\begin{defn}\label{def:intw_op}
Given (weak) $V$-modules $W_1, W_2$, and $W_3$, an \textit{even logarithmic intertwining operator} of type $\binom{W_3}{W_1 \, W_2}$ is an even linear map
\begin{align*}
\calY: W_1 \ot W_2 & \longrightarrow W_3[\log x]\{x\}\\
            w_1\ot w_2 & \longmapsto \calY(w_1, x)w_2=\sum_{h \in \CC} \sum_{k \in \ZZ_{\geq 0}} (w_1)^{\calY}_{h;k}w_2\,x^{-h-1}(\log x)^k 
\end{align*}
satisfying the following conditions:
\begin{enumerate}
    \item For any $w_1\in W_1, w_2\in W_2$ and $h\in \CC$, $(w_1)^{\calY}_{h+n;k}w_2=0$ for all sufficiently positive $n\in\ZZ$, independently of $k$.

    \item The Jacobi identity: For any parity homogeneous $v\in V$ and $w_1\in W_1$ 
\begin{align*}
x_0^{-1}\delta\bigg(\frac{x_1-x_2}{x_0}\bigg)Y_{W_3}(v,x_1)\calY (w_1, x_2) & -(-1)^{|v||w_1|}x_0^{-1}\delta\bigg(\frac{-x_2+x_1}{x_0}\bigg)\calY (w_1, x_2)Y_{W_2}(v,x_1)\\
& = x_2^{-1}\delta\bigg(\frac{x_1-x_0}{x_2}\bigg)\calY (Y_{W_1}(v,x_0)w_1, x_2).
\end{align*}
\item The $L_{-1}$-derivative property: $\calY(L_{-1}w_1,x) = \frac{d}{dx} \calY(w_1,x)$ for $w_1\in W_1$.
\end{enumerate}
\end{defn}

Since we only consider even intertwining operators, and since we always allow logarithmic terms in the formal series expansion of an intertwining operator, from now on, we will abbreviate ``even logarithmic intertwining operator'' as ``intertwining operator.'' 
An intertwining operator $\mathcal{Y}$ of type $\binom{W_3}{W_1\,W_2}$ is \textit{surjective} if $W_3$ is spanned by elements of the form $(w_1)^\calY_{h;k} w_2$ for $w_1\in W_1$, $w_2\in W_2$, $h\in\CC$, and $k\in\ZZ_{\geq 0}$. More generally, we call the submodule of $W_3$ spanned by such elements the image of $\cY$, denoted $\mathrm{Im}\,\cY$.

Extracting the coefficient of $x_0^{-1}x_1^{-n-1}$ in the intertwining operator Jacobi identity yields the (anti)-commutator formula
\begin{align}
v_n\cY(w_1,x) = (-1)^{|v||w_1|}\cY(w_1,x)v_n + \sum_{i\geq 0}\binom{n}{i}x^{n-i}\cY(v_iw_1,x),
\end{align}
for parity-homogeneous $v\in V$, $w_1\in W_1$. When $v = \omega$, this becomes 
\begin{align}\label{eqn:comm1}
L_n\cY(w_1,x) = \cY(w_1,x)L_n + \sum_{i\geq 0}\binom{n+1}{i}x^{n+1-i}\cY(L_{i-1}w_1,x),
\end{align}
and when $V$ is the Neveu-Schwarz vertex operator superalgebra $\cS(c,0)$, the case $v=\tau=G_{-\frac{3}{2}}\vac_{c,0}$ yields
\begin{equation}\label{eqn:comm2}
G_{n-\frac{1}{2}}\cY(w_1,x) = (-1)^{|w_1|}\cY(w_1,x)G_{n-\frac{1}{2}} + \sum_{i\geq 0}\binom{n}{i}x^{n-i}\cY(G_{i-\frac{1}{2}}w_1,x).
\end{equation}
Similarly, the coefficient of $x_0^{-n-1} x_1^{-1}$ in the Jacobi identity yields the associator formula
\begin{align*}
\cY(v_nw_1,x) = \sum_{i\geq 0}(-1)^i\binom{n}{i}(v_{n-i}x^i\cY(w_1,x)-(-1)^{|v||w_1|+n}x^{n-i}\cY(w_1,x)v_i).    
\end{align*}
When $v =\omega$, this becomes
\begin{align}\label{eqn:asso1}
\cY(L_nw_1,x) = \sum_{i\geq 0}(-1)^i\binom{n+1}{i}(L_{n-i}x^i\cY(w_1,x)+(-1)^{n}x^{n+1-i}\cY(w_1,x)L_{i-1}),    
\end{align}
while the case $v = \tau$ gives
\begin{equation}\label{eqn:asso2}
   \cY(G_{n-\frac{1}{2}}w_1,x) = \sum_{i\geq 0}(-1)^i\binom{n}{i}(G_{n-i-\frac{1}{2}}x^i\cY(w_1,x)-(-1)^{|w_1|+n}x^{n-i}\cY(w_1,x)G_{i-\frac{1}{2}}) 
\end{equation}
for $\cS(c,0)$-modules.

We now define the tensor product of $V$-modules, following \cite[Section 4.1]{HLZ3}. Let $\cC$ be an appropriate subcategory of $V$-modules. 
For objects $W_1$ and $W_2$ of $\cC$, their tensor product is, if it exists, an object $W_1\boxtimes W_2$ of $\cC$ equipped with a distinguished intertwining operator $\calY_{W_1,W_2}$ of type $\binom{W_1\boxtimes W_2}{W_1\,W_2}$. The pair $(W_1\boxtimes W_2,\calY_{W_1,W_2})$ is characterized by a universal property: for any object $W_3$ of $\cC$ and intertwining operator $\cY$ of type $\binom{W_3}{W_1\,W_2}$, there is a unique (even) $V$-module homomorphism $f: W_1\boxtimes W_2\rightarrow W_3$ such that $f\circ\calY_{W_1,W_2}=\calY$. 
We call $\calY_{W_1,W_2}$ the tensor product intertwining operator of type $\binom{W_1\boxtimes W_2}{W_1\,W_2}$, and it is surjective (see \cite[Proposition 4.23]{HLZ3}).

\begin{remark}
    If we considered $\cC$ as a supercategory, that is, if we allowed odd $V$-module homomorphisms in $\cC$, then we would also allow odd intertwining operators. But if $\cY$ is an odd intertwining operator of type $\binom{W_3}{W_1\,W_2}$, then, similar to Remark \ref{rem:odd-to-even-morphism}, $P_{W_3}\circ\cY$ is an even intertwining operator of type $\binom{\Pi(W_3)}{W_1\,W_2}$. Thus if the tensor product $(W_1\boxtimes W_2,\cY_{W_1,W_2})$ exists in $\cC$, then it also satisfies the correct universal property for odd intertwining operators. In particular, there is a unique even $V$-module homomorphism $f: W_1\boxtimes W_2\rightarrow \Pi(W_3)$ such that $f\circ\cY_{W_1,W_2}=P_{W_3}\circ\cY$, and then $P_{W_3}\circ f: W_1\boxtimes W_2\rightarrow W_3$ is the unique odd $V$-module homomorphism such that $(P_{W_3}\circ f)\circ\cY_{W_1,W_2}=\cY$. Thus the existence and isomorphism class of the tensor product $W_1\boxtimes W_2$ does not depend on whether we consider $\Pi$-supercategories of $V$-modules or their underlying $\Pi$-categories (using terminology of \cite{BE}).
\end{remark}

Now suppose $\cC$ is closed under tensor products and the parity-reversal functor $\Pi$. For objects $W_1$ and $W_2$ of $\cC$ and $n_1,n_2\in\ZZ$, note that 
\begin{equation}\label{eqn:Pi_ten_prod_reln}
\Pi^{n_1}(W_1)\otimes\Pi^{n_2}(W_2)=\Pi^{n_1+n_2}(W_1\otimes W_2)
\end{equation}
as vector superspaces because for $j=0,1$,
\begin{equation*}
    \left(\Pi^{n_1}(W_1)\otimes\Pi^{n_2}(W_2)\right)^{\bar{j}} =\bigoplus_{i_1+i_2\in j-n_1-n_2+2\ZZ} W_1^{\overline{i_1}}\otimes W_2^{\overline{i_2}} = \Pi^{n_1+n_2}(W_1\otimes W_2)^{\bar{j}}.
\end{equation*}
We show that the same relation holds for $V$-module tensor products:
\begin{prop}\label{prop:Pi-and-fusion}
    Let $V$ be a vertex operator superalgebra and $\cC$ a category of $V$-modules that is closed under $\boxtimes$ and $\Pi$. Then for objects $W_1$, $W_2$ of $\cC$ and for $n_1,n_2\in\ZZ$, there is a unique $V$-module isomorphism
    \begin{equation*}
        F: \Pi^{n_1}(W_1)\boxtimes\Pi^{n_2}(W_2)\longrightarrow\Pi^{n_1+n_2}(W_1\boxtimes W_2)
    \end{equation*}
    such that $F\circ\calY_{\Pi^{n_1}(W_1),\Pi^{n_2}(W_2)} =\cY_{W_1,W_2}\circ(\id_{W_1}\otimes P_{W_2}^{n_1})$, where $P_{W_2}$ is the parity involution of the vector superspace $W_2$.
\end{prop}
\begin{proof}
    Set $\widetilde{\calY}=\cY_{W_1,W_2}\circ(\id_{W_1}\otimes P_{W_2}^{n_1})$; by \eqref{eqn:Pi_ten_prod_reln}, we can view $\widetilde{\calY}$ as an even linear map
    \begin{equation*}
        \Pi^{n_1}(W_1)\otimes\Pi^{n_2}(W_2)\longrightarrow\Pi^{n_1+n_2}(W_1\boxtimes W_2)[\log x]\lbrace x\rbrace.
    \end{equation*}
    We claim that $\widetilde{\cY}$ is an intertwining operator of type $\binom{\Pi^{n_1+n_2}(W_1\boxtimes W_2)}{\Pi^{n_1}(W_1)\,\Pi^{n_2}(W_2)}$. Indeed, condition (1) of Definition \ref{def:intw_op} is immediate and the $L_{-1}$-derivative property (3) follows because $Y_{\Pi^{n_1}(W_1)}(\omega,x)=Y_{W_1}(\omega,x)$. Then the Jacobi identity for $\cY_{W_1,W_2}$ composed with $P_{W_2}^{n_1}$ yields an equality involving the three terms
    \begin{align*}
        Y_{W_1\boxtimes W_2}(v,x_1)\cY_{W_1,W_2}(w_1,x_2)P_{W_2}^{n_1} & = Y_{\Pi^{n_1+n_2}(W_1\boxtimes W_2)}(v,x_1)\widetilde{\cY}(w_1,x_2),\\
        (-1)^{\vert v\vert\vert w_1\vert}\cY_{W_1,W_2}(w_1,x_2)Y_{W_2}(v,x_1)P_{W_2}^{n_1} & =(-1)^{\vert v\vert(\vert w_1\vert+n_1)}\widetilde{\cY}(w_1,x_2)Y_{\Pi^{n_2}(W_2)}(v,x_1),\\
       \cY_{W_1,W_2}(Y_{W_1}(v,x_0)w_1,x_2)P_{W_2}^{n_1} &=\widetilde{\cY}(Y_{\Pi^{n_1}(W_1)}(v,x_0)w_1,x_2).
    \end{align*}
    This yields the Jacobi identity for $\widetilde{\cY}$ because the parity of $w_1$ changes to $\vert w_1\vert +n_1$ (mod $2$) in $\Pi^{n_1}(W_1)$. Thus $\widetilde{\cY}$ is an intertwining operator, and the existence and uniqueness of the $V$-module homomorphism $F$ follows from the universal property of tensor products.

    To show that $F$ is an isomorphism, we can replace $W_1$, $W_2$ with $\Pi^{n_1}(W_1)$, $\Pi^{n_2}(W_2)$ in the above argument to get a unique $V$-module homomorphism 
    \begin{equation*}
        G: W_1\boxtimes W_2\longrightarrow\Pi^{n_1+n_2}\left(\Pi^{n_1}(W_1)\boxtimes\Pi^{n_2}(W_2)\right)
    \end{equation*}
    such that $G\circ\cY_{W_1,W_2} =\cY_{\Pi^{n_1}(W_1),\Pi^{n_2}(W_2)}\circ(\id_{\Pi^{n_1}(W_1)}\otimes P_{\Pi^{n_2}(W_2)}^{n_1})$. We can equally well view $G$ as a $V$-module homomorphism from $\Pi^{n_1+n_2}(W_1\boxtimes W_2)$ to $\Pi^{n_1}(W_1)\boxtimes\Pi^{n_2}(W_2)$, and then the definitions yield
    \begin{align*}
        F\circ G\circ\cY_{W_1,W_2} &=F\circ \cY_{\Pi^{n_1}(W_1),\Pi^{n_2}(W_2)}\circ(\id_{\Pi^{n_1}(W_1)}\otimes P_{\Pi^{n_2}(W_2)}^{n_1})\nonumber\\
        &= \cY_{W_1,W_2}\circ(\id_{W_1}\otimes P_{W_2}^{n_1}P_{\Pi^{n_2}(W_2)}^{n_1})\nonumber\\
        &=(-1)^{n_1n_2}\cY_{W_1,W_2}
    \end{align*}
since $P_{\Pi(W_2)}=-P_{W_2}$, and similarly
    \begin{equation*}
        G\circ F\circ\cY_{\Pi^{n_1}(W_1),\Pi^{n_2}(W_2)} =(-1)^{n_1n_2}\cY_{\Pi^{n_1}(W_1),\Pi^{n_2}(W_2)}.
    \end{equation*}
    Because $\cY_{W_1,W_2}$ and $\cY_{\Pi^{n_1}(W_1),\Pi^{n_2}(W_2)}$ are surjective, it follows that $F$ is an isomorphism with inverse $(-1)^{n_1n_2}G$.
    
\end{proof}

\subsection{Vertex tensor categories and \texorpdfstring{$C_1$}{C1}-cofinite modules}

Let $V$ be a vertex operator superalgebra and let $\cC$ be a category of $V$-modules which contains $V$ and is closed under $\boxtimes$ and $\Pi$. After adding sign factors appropriately, the work for vertex operator algebras in \cite{HLZ1}-\cite{HLZ8} shows that under suitable conditions, $\boxtimes$ gives $\cC$ the structure of a braided tensor category (with a parity-reversal functor). See also \cite[Section 3.3]{CKM-exts} for an explicit description of the braided tensor category structure in the superalgebra setting.

For $\cC$ to be a braided tensor category, we need a tensor product of $V$-module homomorphisms, left and right unit isomorphisms, braiding isomorphisms, and associativity isomorphisms. The first three are automatic given that $\cC$ is closed under $\boxtimes$:
\begin{itemize}
    \item For (even) $V$-module homomorphisms $f_1: W_1\rightarrow X_1$ and $f_2: W_2\rightarrow X_2$ in $\cC$,
    \begin{equation*}
        f_1\boxtimes f_2: W_1\boxtimes W_2\longrightarrow X_1\boxtimes X_2
    \end{equation*}
    is the unique $V$-module homomorphism (guaranteed by the universal property of $(W_1\boxtimes W_2,\cY_{W_1,W_2})$ such that
    \begin{equation*}
        (f_1\boxtimes f_2)\circ\cY_{W_1,W_2} =\cY_{X_1,X_2}\circ(f_1\otimes f_2).
    \end{equation*}
    Note that $f_1\boxtimes f_2$ is even because $f_1$ and $f_2$ are even, and because $\cY_{W_1,W_2}$ and $\cY_{X_1,X_2}$ are even and surjective.

    \item For an object $W$ in $\cC$, the left and right unit isomorphisms $l_W: V\boxtimes W\rightarrow W$ and $r_W: W\boxtimes V\rightarrow W$ are the unique $V$-module homomorphisms such that
    \begin{align}
        l_W(\cY_{V,W}(v,x)w) & = Y_W(v,x)w \label{eqn:left_unit}\\
        r_W(\cY_{W,V}(w,x)v) & = (-1)^{\vert v\vert\vert w\vert}e^{xL_{-1}} Y_W(v,-x)w\label{eqn:right_unit}
    \end{align}
    for parity-homogeneous $v\in V$, $w\in W$.

    \item For objects $W_1$, $W_2$ in $\cC$, the braiding isomorphism $\cR_{W_1,W_2}: W_1\boxtimes W_2\rightarrow W_2\boxtimes W_1$ is the unique $V$-module homomorphism such that
    \begin{equation}\label{eqn:braiding}
\cR_{W_1,W_2}(\cY_{W_1,W_2}(w_1,x)w_2) = (-1)^{|w_1||w_2|}e^{xL_{-1}}\cY_{W_2,W_1}(w_2, e^{\pi i} x)w_1
    \end{equation}
    for parity-homogeneous $w_1\in W_1$, $w_2\in W_2$. Here, $e^{\pi i}x$ means that we substitute $x^h\mapsto e^{\pi i h}x^h$ and $\log x\mapsto \log x+\pi i$ in the intertwining operator $\cY_{W_2,W_1}$.
    
\end{itemize}
To get associativity isomorphisms, we need to impose further conditions on $\cC$.

To describe the associativity isomorphisms in $\cC$, if they exist, we introduce some notation. If $W=\bigoplus_{h\in\CC} W_{[h]}$ is a graded vector space, then its \textit{algebraic completion} is $\overline{W}=\prod_{h\in\CC} W_{[h]}$, and if $f: W\rightarrow X$ is a grading-preserving linear map between graded vector spaces, then there is a natural extension $\overline{f}:\overline{W}\rightarrow\overline{X}$ between algebraic completions.
Now if $\cY$ is a $V$-module intertwining operator of type $\binom{W_3}{W_1\,W_2}$, then we can replace the formal variable $x$ in $\cY$ with a complex number $z\in\CC^\times$ to obtain a \textit{$P(z)$-intertwining map} $W_1\otimes W_2\rightarrow\overline{W}_3$, in the terminology of \cite{HLZ3}. In particular, for $r\in\RR_+$, $\cY(\cdot,e^{\ln r})$ denotes the $P(r)$-intertwining map obtained by substituting $x^h\mapsto r^h$ and $\log x\mapsto\ln r$ in $\cY$.

Now following \cite[Section 12.2]{HLZ8} and \cite[Section 3.3.5]{CKM-exts}, the associativity isomorphisms, if they exist, can be described as follows. For $V$-modules $W_1$, $W_2$, and $W_3$ in $\cC$, the associativity isomorphism
\[
\cA_{W_1, W_2, W_3}: W_1\boxtimes (W_2 \boxtimes W_3) \rightarrow (W_1 \boxtimes W_2) \boxtimes W_3
\]
is characterized by the equality
\begin{align}\label{eqn:assoc}
   \langle w', \overline{\cA_{W_1, W_2, W_3}} & (\cY_{W_1, W_2\boxtimes W_3}(w_1, e^{\ln\; r_1})\cY_{W_2, W_3}(w_2, e^{\ln\; r_2})w_3)\rangle \nonumber\\
& = \langle w', \cY_{W_1\boxtimes W_2, W_3}(\cY_{W_1, W_2}(w_1, e^{\ln(r_1-r_2)})w_2, e^{\ln\; r_2})w_3\rangle
\end{align}
for $w_1 \in W_1$, $w_2 \in W_2$, $w_3 \in W_3$, $w' \in ((W_1 \boxtimes W_2) \boxtimes W_3)'$, and $r_1, r_2 \in \RR$ such that $r_1 > r_2 > r_1-r_2 >0$. By \cite[Proposition 3.32]{CKM-exts}, $\cA_{W_1,W_2,W_3}$ does not depend on the choice of $r_1$ and $r_2$.

For $\cA_{W_1,W_2,W_3}$ to exist, the product of two intertwining operators, as on the left side of \eqref{eqn:assoc} must converge absolutely to an element of $\overline{W_1\boxtimes(W_2\boxtimes W_3)}$, and the left side of \eqref{eqn:assoc} should analytically extend to a suitable multivalued analytic function on the region $\vert z_2\vert>\vert z_1-z_2\vert >0$. Similarly, the iterate of two intertwining operators, as on the right side of \eqref{eqn:assoc}, should analytically extend to a suitable function on the region $\vert z_1\vert >\vert z_2\vert >0$.
 This is called convergence and extension property for products and iterates of intertwining operators; see \cite[Section 11.1]{HLZ7} for the precise statement. 
 
 Now, \cite[Theorem 10.3]{HLZ6}, the proof of \cite[Theorem 11.4]{HLZ7}, and \cite[Theorem 3.1]{H1} show that in addition to the convergence and extension property,
 one more condition is sufficient for existence of the associativity isomorphisms in $\cC$. Namely, for objects $W_1$, $W_2$ of $\cC$ and $z \in \CC^{\times}$, \cite[Theorem 5.4]{HLZ4} shows that there is a weak $V$-module $\mathrm{COMP}_{P(z)}((W_1\otimes W_2)^*)$ which is realized as the subspace of linear functionals in $(W_1\otimes W_2)^*$ that satisfy the \textit{$P(z)$-compatibility condition} of \cite{HLZ4}. Then the additional sufficient condition is that any singly-generated lower-bounded generalized $V$-module $W\subset \mathrm{COMP}_{P(z)}((W_1\otimes W_2)^*)$ should be an object of $\cC$. That is, the associativity isomorphisms in $\cC$ exist if this condition and the convergence and extension property hold. Further, by \cite[Theorem 12.15]{HLZ8}, the triangle, pentagon, and hexagon axioms of a braided tensor category hold in $\cC$ if a suitable convergence condition for the product of three intertwining operators also holds; see \cite[Assumption 12.2]{HLZ8} for the precise statement.

 To summarize, the category $\cC$ of $V$-modules is a vertex tensor category in the sense of \cite{HLZ8}, and also a braided tensor category, if it is abelian, contains $V$, and is closed under $\boxtimes$ and $\Pi$, and if the convergence and extension condition for products and iterates of two intertwining operators holds, along with the convergence condition for products of three intertwining operators and the additional condition on lower-bounded submodules of $\mathrm{COMP}_{P(z)}((W_1\otimes W_2)^*)$ described above. A natural candidate for such a category $\cC$ is the category of $C_1$-cofinite $V$-modules:
 \begin{defn} \label{cc1}
If $W$ is a weak module of a vertex operator superalgebra $V=\bigoplus_{n\in \frac{1}{2}\mathbb{Z}}V_{(n)}$, define
\begin{equation*}
    C_1(W) = \mathrm{span}\{v_{-1}w\mid v \in V_+,\,w\in W\},
\end{equation*}
where $V_+ = \bigoplus_{n\in \frac{1}{2}\mathbb{Z}_{>0}}V_{(n)}$.
Then $W$ is \textit{$C_1$-cofinite} if $\dim W/C_1(W) < \infty$.
\end{defn}

The definition of $C_1$-cofinite modules essentially goes back to Nahm \cite{Na}, who called such modules ``quasi-rational.'' Let $\cC_1(V)$ denote the category of $C_1$-cofinite grading-restricted generalized $V$-modules; as in \cite[Proposition 2.2]{CMY-completions}, this is identical to the category of $C_1$-cofinite $\frac{1}{2}\NN$-gradable weak $V$-modules. The following theorem was essentially proved by Nahm \cite{Na} in physics language and was proved in the setting of vertex operator algebras by Miyamoto \cite{Mi}; the proof also holds for vertex operator superalgebras as long as suitable sign factors are added whenever the Jacobi identity is used:

\begin{theorem}[\cite{Mi} Key Theorem]
If $W_1$, $W_2$ are objects of $\cC_1(V)$ for a vertex operator superalgebra $V$, and if $\cY$ is a surjective intertwining operator of type $\binom{X}{W_1
\, W_2}$ for some $\frac{1}{2}\NN$-gradable weak $V$-module $X$, then $X$ is $C_1$-cofinite. Specifically,
\begin{equation*}
    \dim (X/C_1(X)) \leq \dim(W_1/C_1(W_1))\dim(W_2/C_1(W_2)).
\end{equation*}

\end{theorem}

We can strengthen this result by removing the $\frac{1}{2}\NN$-gradability assumption on $X$:
\begin{prop}[\cite{CMY-completions} Lemma 2.13, Corollary 2.14]\label{prop:Miy-Key-Thm}
If $W_1$, $W_2$ are objects of $\cC_1(V)$ and $\mathcal{Y}$ is a surjective intertwining operator of type $\binom{X}{W_1\, W_2}$ for some generalized $V$-module $X$, then $X$ is an object of $\cC_1(V)$.
\end{prop}

As an application of these results, we have the following crucial theorem:
\begin{prop}[\cite{Mi} Main Theorem]\label{prop:vosa-miyamoto}
   If $W_1$, $W_2$ are objects of $\cC_1(V)$ for a vertex operator superalgebra $V$, then there is a tensor product $(W_1\boxtimes W_2,\cY_{W_1,W_2})$ of $W_1$ and $W_2$ in the category of generalized $V$-modules, and $W_1\boxtimes W_2$ is $C_1$-cofinite. In particular, $\cC_1(V)$ is closed under tensor products. 

\end{prop}

It is easy to see that $\cC_1(V)$ contains $V$ and is closed under $\Pi$. As for the convergence and extension property for products and iterates of intertwining operators, \cite[Theorem 1.6]{H2} and \cite[Theorem 11.6]{HLZ7} show that compositions of intertwining operators such as
\begin{align*}
    \langle w_0', \cY_1(w_{0}, z_1)\cdots \cY_n(w_{n}, z_n)w_{n+1}\rangle,\nonumber\\
    \langle w_0', \cY^1(\cY^2(w_1,z_1-z_2)w_2,z_2)w_3\rangle,
\end{align*}
where the vectors $w_0', w_1, w_2,\ldots, w_n$ are elements of $C_1$-cofinite grading-restricted generalized modules for a vertex operator algebra, satisfy systems of partial differential equations with regular singular points at $z_i =0$, $z_i =\infty$, $z_i=z_j$ for some $i,j=1,2,\ldots, n$, $i\neq j$. After checking Proposition 1.1, Theorem 1.4, Proposition~2.1, and Theorem~2.3 of \cite{H2}, one sees that their proofs still hold for $C_1$-cofinite modules of a vertex operator superalgebra $V$ after adding suitable sign factors whenever the Jacobi identity is used. It then follows from the theory of regular singular point differential equations that compositions of intertwining operators in $\cC_1(V)$ satisfy the necessary convergence and extension properties.

Finally, we need to address whether $\cC_1(V)$ satisfies the additional condition on submodules of $\mathrm{COMP}_{P(z)}((W_1\otimes W_2)^*)$. The following result is proved in \cite[Theorem~4.2.5]{CJORY}, \cite[Theorem~3.6]{CY}, and also \cite[Theorem~2.3]{Mc2}, in the setting of vertex operator algebras, but the result also holds for superalgebras:
\begin{prop}\label{prop:vosa-huang}
Assume that $\cC_1(V)$ is closed under contragredient modules. Then if $W_1$, $W_2$ are objects of $\cC_1(V)$, any singly-generated lower-bounded generalized $V$-submodule $W\subset\mathrm{COMP}_{P(z)}((W_1\otimes W_2)^*)$ is $C_1$-cofinite.
 
\end{prop}
\begin{proof}
    The proof is a superalgebra generalization of \cite[Theorem~2.3]{Mc2}.
    The inclusion $W \hookrightarrow (W_1\otimes W_2)^*$ induces an even linear map $I: W_1\otimes W_2 \rightarrow W^* = \overline{W'}$ such that
    \[
    \langle I(w_1\otimes w_2), f\rangle = (-1)^{|f|(|w_1|+|w_2|)}f(w_1\otimes w_2)
    \]
    for linear functionals $f \in W$. Since $W$ is lower bounded, $W'$ is also a lower-bounded generalized module, and then the $P(z)$-compatibility condition of \cite{HLZ4} (or rather, its superalgebra generalization), implies that $I$ is a $P(z)$-intertwining map. Thus by \cite[Proposition 4.8]{HLZ3}, $I = \cY(\cdot, z)$ for some intertwining operator $\cY$ of type $\binom{W'}{W_1\; W_2}$.

    By Proposition \ref{prop:Miy-Key-Thm}, Im $\cY$ is an object in $\cC_1(V)$, and thus $({\rm Im}\; \cY)'$ is also $C_1$-cofinite by assumption. Thus we have a $V$-module homomorphism
    \[
    \delta: W \hookrightarrow W'' \twoheadrightarrow ({\rm Im}\; \cY)'
    \]
    of $W$ to a grading-restricted generalized $V$-module, characterized by
    \[
   \langle \delta(f), I(w_1\otimes w_2)\rangle = (-1)^{|f|(|w_1|+|w_2|)}\langle I(w_1\otimes w_2),f \rangle = f(w_1\otimes w_2)
    \]
    for $f \in W$, $w_1 \in W_1$, and $w_2 \in W_2$. It is clear from the above that $\delta$ is injective, so $W$ is a grading-restricted generalized $V$-module. It then follows that the map $W \rightarrow W''$ is an isomorphism, so that $\delta$ is also surjective. Hence $W \cong ({\rm Im}\; \cY)'$ is an object of $\cC_1(V)$.
    
\end{proof}

Note that $\cC_1(V)$ is an abelian category if it is closed under contragredients. Indeed, it is easy to see that $\cC_1(V)$ is closed under quotients and finite direct sums, and closure under quotients and contragredients implies closure under submodules. Thus from the above results and discussion, we obtain:

\begin{theorem}\label{thm:vosa-btc}
Let $\cC_1(V)$ be the category of $C_1$-cofinite grading-restricted generalized modules for a vertex operator superalgebra $V$. If $\cC_1(V)$ is closed under contragredients, then $\cC_1(V)$ admits the braided tensor category structure of \cite{HLZ1}-\cite{HLZ8}. In particular, $\cC_1(V)$ is a braided tensor category if the following two conditions hold:
\begin{enumerate}
    \item The contragredient $W'$ of any simple $C_1$-cofinite $V$-module $W$ is $C_1$-cofinite.
    \item $\cC_1(V)$ is equal to the category of finite-length $V$-modules whose composition factors are $C_1$-cofinite.
\end{enumerate}
\end{theorem}

\begin{remark}
    In the next section, we will use Theorem \ref{thm:vosa-btc} to show that the category of $C_1$-cofinite $\cS(c,0)$-modules is a braided tensor category. However, in future examples, it may be difficult to determine the entire category of $C_1$-cofinite modules. Thus it is useful to generalize Theorem \ref{thm:vosa-btc} to subcategories $\cC\subset\cC_1(V)$ that satisfy suitable conditions. Specifically, let $\cC$ be a full subcategory of the category $\cC_1(V)$ of $C_1$-cofinite grading-restricted generalized $V$-modules which contains $V$ and is closed under finite direct sums, quotient modules, contragredients, and tensor products. Then $\cC$ is abelian and admits the braided tensor category structure of \cite{HLZ1}-\cite{HLZ8}. 
    
    Indeed, $\cC$ is abelian because closure under quotients and contragredients implies that $\cC$ contains $0$ and is closed under submodules. Also, because $\cC$ contains $V$ and is closed under tensor products, $\cC$ has tensor products of morphisms, unit isomorphisms, and braiding isomorphisms. For the associativity and coherence properties, the necessary convergence and extension properties for compositions of intertwining operators in $\cC$ follow from \cite{H2, HLZ7} as previously because all objects of $\cC$, as well as their contragredients, are $C_1$-cofinite. Finally, since $\cC$ is closed under contragredients, the proof of Proposition \ref{prop:vosa-huang} goes through with $\cC$ replacing $\cC_1(V)$, provided that $\mathrm{Im}\,\cY$ is an object of $\cC$ for any intertwining operator $\cY$ of type $\binom{X}{W_1\,W_2}$ where $X$ is a generalized $V$-module and $W_1, W_2$ are objects of $\cC$. But this holds because $\mathrm{Im}\,\cY$ is $C_1$-cofinite by Proposition \ref{prop:Miy-Key-Thm} and thus is a homomorphic image of the $C_1$-cofinite tensor product $W_1\boxtimes W_2$; then $\mathrm{Im}\,\cY$ is an object of $\cC$ because $\cC$ is closed under tensor products and quotients. Now $\cC$ admits the braided tensor category structure of \cite{HLZ1}-\cite{HLZ8} as in Theorem \ref{thm:vosa-btc}.
\end{remark}

\section{Tensor categories for the \texorpdfstring{$N=1$}{N=1} super Virasoro algebra}\label{sec:tensor}

In this section, we will apply Theorem \ref{thm:vosa-btc} to the $N=1$ superconformal vertex operator superalgebra $\cS(c,0)$. Since simple $\cS(c,0)$-modules are self-contragredient, it is sufficient to show that the category $\cC_1(\cS(c,0))$ of $C_1$-cofinite $\cS(c,0)$-modules equals the category of finite-length modules with $C_1$-cofinite composition factors. We will then study $\cC_1(\cS(c,0))$ in the special cases $c=\frac{15}{2}-3t-3t^{-1}$ for $t\notin\QQ$ and $t=1$.

\subsection{\texorpdfstring{$C_1$}{C1}-cofinite modules for the \texorpdfstring{$N=1$}{N=1} super Virasoro algebra}

For $V = S(c,0)$ the conformal vector $\omega=L_{-2}\vac_{c,0}$ and the Neveu-Schwarz vector $\tau=G_{-\frac{3}{2}}\vac_{c,0}$ are homogeneous elements (even and odd, respectively) of $\cS(c,0)_+$, with $\omega_{-1}=L_{-2}$ and $\tau_{-1}=G_{-\frac{3}{2}}$.
Recalling Definition \ref{cc1}, we can prove:
\begin{prop}\label{prop:C1W-span-set}
    For any $\cS(c,0)$-module $W$,
\begin{align} \label{c1W}
C_1(W)=\sum_{n \in \mathbb{Z}_{\geq 2}} L_{-n}W+ \sum_{m\in \mathbb{Z}_{\geq 1}}G_{-m-\frac{1}{2}}W.
\end{align} 
\end{prop}
\begin{proof}
Set $\tilC = \sum_{n \in \mathbb{Z}_{\geq 2}} L_{-n}W+ \sum_{m\in \mathbb{Z}_{\geq 1}}G_{-m-\frac{1}{2}}W$, so that we need to show $C_1(W)=\tilC$. By \eqref{operator}, which also applies to the module vertex operator $Y_W$,
\begin{align*}
    Y_W(L_{-n}\vac,x)& =\frac{1}{(n-2)!}\left(\frac{d}{dx}\right)^{n-2}\sum_{i\in\ZZ} L_i\,x^{-i-2} =\sum_{i\in\ZZ} \binom{-i-2}{n-2} L_i\,x^{-i-n},\nonumber\\
    Y_W(G_{-m-\frac{1}{2}},x) & = \frac{1}{(m-1)!}\left(\frac{d}{dx}\right)^{m-1}\sum_{i\in\ZZ} G_{i+\frac{1}{2}}\,x^{-i-2} =\sum_{i\in\ZZ}\binom{-i-2}{m-1} G_{i+\frac{1}{2}}\,x^{-i-m-1}
\end{align*}
for $n\geq 2$ and $m\geq 1$. This implies
\begin{equation}\label{eqn:-j-1_mode}
    (L_{-n}\vac)_{-j-1}= \binom{j+n-2}{n-2} L_{-n-j},\qquad(G_{-m-\frac{1}{2}}\vac)_{-j-1} =\binom{j+m-1}{m-1}G_{-m-j-\frac{1}{2}}
\end{equation}
for $j\in\ZZ$, and the $j=0$ case shows $\tilC\subset C_1(W)$.

    For the converse, the $L_{-1}$-derivative property implies that $v_{-j-1} =\frac{1}{j!}(L_{-1}^j v)_{-1}$ for $v\in\cS(c,0)$ and $j\in\ZZ_{\geq 0}$. Thus $C_1(W)$ is spanned by elements of the form $v_{-j-1}w$ for $j\in\ZZ_{\geq 0}$, $w\in W$, and $v\in\cS(c,0)_+$ (if $j=0)$ or $v\in\cS(c,0)$ (if $j>0$). Thus since $\cS(c,0)_+$ is spanned by vectors of the form 
    $$L_{-i_1}\cdots L_{-i_k} G_{-j_1-\frac{1}{2}}\cdots G_{-j_l-\frac{1}{2}}\vac$$
    for $i_1,\ldots, i_k\geq 2$ and $j_1,\ldots j_l\geq 1$, $C_1(W)$ is spanned by vectors of the form
    \begin{equation*}
        (L_{-n}v)_{-j-1} w,\quad (G_{-m-\frac{1}{2}}v)_{-j-1}w
    \end{equation*}
    for $n\geq 2$, $m\geq 1$, $v\in\cS(c,0)$, $w\in W$, and $j\geq 0$. Thus it is sufficient to show 
    $(L_{-n}v)_{-j-1}w, (G_{-m-\frac{1}{2}}v)_{-j-1}w\in\tilC$ by induction on $\mathrm{wt}\,v$. 
    
    The base case $v=\vac$ is immediate from \eqref{eqn:-j-1_mode}. For the inductive step, the associator formulas \eqref{eqn:asso1} and \eqref{eqn:asso2}  imply
    \begin{align*}
        (L_{-n}v)_{-j-1} w & =\sum_{i\geq 0} (-1)^i\binom{-n+1}{i}\left(L_{-n-i} v_{i-j-1} w +(-1)^n v_{-n-i-j-2}L_{i-1} w\right)\nonumber\\
        (G_{-m-\frac{1}{2}}v)_{-j-1} & =\sum_{i\geq 0} (-1)^i\binom{-m}{i}\left(G_{-m-i-\frac{1}{2}} v_{i-j-1} w +(-1)^{|v|+n} v_{-m-i-j-1}G_{i-\frac{1}{2}}w\right)
    \end{align*}
In both sums, the first term of each summand is clearly in $\tilC$, and the second term of each summand is in $\tilC$ by induction. This shows $C_1(W)\subset\tilC$.

\end{proof}

The following result was proved in \cite{H4} for vertex operator algebras, but its proof is the same for superalgebras:
\begin{lemma}[\cite{H4}] \label{C1H}
Let $V$ be a vertex operator superalgebra and $W$ a generalized $V$-module.
\begin{enumerate}
    \item If $U$ is a generalized $V$-submodule of $W$, then
     $$\frac{W/U}{C_1(W/U)}\cong \frac{W}{C_1(W)+U}.$$
    
    \item If $W$ has finite length and its composition factors are $C_1$-cofinite, then $W$ is $C_1$-cofinite.
\end{enumerate}
\end{lemma}

The next result concens $C_1$-cofinite highest-weight $\cS(c,0)$-modules:
\begin{lemma} \label{lemc1}
\begin{enumerate}
\item A highest-weight $\cS(c,0)$-module is $C_1$-cofinite if and only if it is not isomorphic to a Verma module or a parity reversal of a Verma module. In particular, the irreducible module $L^\NS(c,h)$ is  $C_1$-cofinite if and only if $J^\NS(c,h)\neq 0$, that is, if and only if $h \in H_c$. 
\item Any finite-length $\cS(c,0)$-module whose composition factors are irreducible quotients of reducible Verma modules or their parity reversals is $C_1$-cofinite.
\item Any $C_1$-cofinite highest-weight $\cS(c,0)$-module has finite length with composition factors of the form $\Pi^n(L^\NS(c,h))$ for $h \in H_c$ and $n=0,1$.
\end{enumerate}
\end{lemma}

\begin{proof}
To prove (1), we claim that as vector superspaces,
\begin{align} \label{Vnotc1}
M^\NS(c,h) = \mathrm{Span}_{\CC}\big\{(G_{-\frac{1}{2}})^j\vac_{c,h} \ | \ j\geq 0 \big\} \oplus C_1(M^\NS(c,h)).
\end{align} 
Indeed, $M^\NS(c,h)$ has a PBW-type basis consisting of elements
\begin{equation*}
    L_{-i_1}\cdots L_{-i_k}G_{-j_1-\frac{1}{2}}\cdots G_{-j_l-\frac{1}{2}}\vac_{c,h}
\end{equation*}
for $i_1\geq\cdots\geq i_k\geq 2$ and $j_1\geq\cdots\geq j_l\geq 0$ (recalling that $L_{-1}=G_{-\frac{1}{2}}^2$). By Proposition \ref{prop:C1W-span-set}, the basis elements such that $k\geq 1$ or $j_1\geq 1$ are contained in $C_1(M^\NS(c,h))$, and thus a (possibly non-direct) sum decomposition as in \eqref{Vnotc1} holds. Moreover, if $w$ is any basis element, then the (anti)-commutation relations in $\NS$ show that $L_{-n} w$ for $n\geq 2$ and $G_{-m-\frac{1}{2}}w$ for $m\geq 1$ are linear combinations of basis elements such that $k\geq 1$ or $j_1\geq 1$. Thus if some non-zero linear combination of $G_{-\frac{1}{2}}^j\vac_{c,h}$, $j\geq 0$, were contained in $C_1(M^\NS(c,h)$,  this would imply by Proposition \ref{prop:C1W-span-set} that some non-zero linear combination of PBW-type basis elements in $M^\NS(c,h)$ would be $0$. Since this is impossible, the sum decomposition in \eqref{Vnotc1} is direct.

It is now clear from \eqref{Vnotc1} that Verma $\cS(c,0)$-modules (and their parity reversals) are not  $C_1$-cofinite, and thus any $C_1$-cofinite highest-weight $\cS(c,0)$-module is not isomorphic to a Verma module or parity reversal of a Verma module. Conversely, if $W$ is a highest-weight $\cS(c,0)$-module that is not isomorphic to a Verma module or its parity reversal, then up to a parity reversal, $W\cong M^\NS(c,h)/J$ for some $h \in H_c$ and $J\neq 0$. By Lemma \ref{C1H}(1), 
\[
W/C_1(W) \cong M^\NS(c,h)/(C_1(M^\NS(c,h))+J)
\]
as vector superspaces, and by Theorem \ref{AsTh}, $J$ contains a singular vector of the form \eqref{svector}. Consequently, $G_{-\frac{1}{2}}^{j}\vac_{c,h}\in C_1(M^\NS(c,h))+J$ for some $j\geq 0$, and then $G_{-\frac{1}{2}}^n\vac_{c,h}\in C_1(M^\NS(c,h))+J$ for all $n\geq j$ because the (anti)-commutation relations in $\NS$ together with Proposition \ref{prop:C1W-span-set} show that $G_{-\frac{1}{2}}C_1(M^\NS(c,h))\subset C_1(M^\NS(c,h))$. This implies that $W$ is $C_1$-cofinite, completing the proof of (1).

Part (2) of the lemma now follows from (1) together with Lemma \ref{C1H}(2) since the composition factors in this case are $C_1$-cofinite.

To prove (3), it follows from part (1) that a $C_1$-cofinite highest-weight $\cS(c,0)$-module $W$ must be isomorphic to $ M^\NS(c,h)/J$ or $\Pi(M^\NS(c,h)/J)$ for some $J\neq 0$, so $h\in H_c$ by Proposition \ref{Redu}. Moreover, by Corollary \ref{coro}(1), $J$ is either a Verma submodule or a sum of two Verma submdoules, and then from the embedding diagrams in Proposition \ref{vermadiag}, $W\cong M^\NS(c,h)/J$ or $\Pi(M^\NS(c,h)/J)$ has finite length and its composition factors are irreducible modules $L^\NS(c,h')$ or $\Pi(L^\NS(c,h'))$ for certain $h'\in H_c$.

\end{proof}

\begin{defn}
We define $\ocfin$ to be the category of finite-length grading-restricted generalized $\cS(c,0)$-modules whose irreducible composition factors are not isomorphic to Verma modules or their parity reversals.
\end{defn}
We will need the following technical result to show that any (not necessarily highest-weight) lower-bounded $C_1$-cofinite module  has finite length. The proof is essentially the same as that of \cite[Lemma 3.2.1]{CJORY}, with some adjustments for the super Virasoro case:
\begin{prop} \label{desclarge}
Let $W$ be a lower-bounded $C_1$-cofinite $\cS(c,0)$-module and let $\widetilde{W}\subset W$ be a submodule such that $W/\widetilde{W}$ is generated by a highest-weight vector $w + \widetilde{W}$ for some $w\in W$. Then any element in $\mathcal{U}(\NS^-)w\cap \widetilde{W}$ of sufficiently high conformal weight is in $C_1(\widetilde{W})$.

\end{prop}
\begin{proof}
Let $\pi: W \rightarrow W/\widetilde{W}$ be the obvious projection.
Because $W/\widetilde{W}$ is generated by the highest-weight vector $\pi(w)$, we have $W/\widetilde{W} = \mathcal{U}(\NS^-)\pi(w)$. It follows that
\begin{align} \label{seq}
W=\U(\NS^-)w+\W    
\end{align}
and therefore by \eqref{c1W},
\begin{align} \label{C1Wsum}
C_1(W)=\sum_{n \in \mathbb{Z}_{\geq 2}} L_{-n}\U(\NS^-)w+ \sum_{m\in \mathbb{Z}_{\geq 1}}G_{-m-\frac{1}{2}}\U(\NS^-)w+C_1(\W).
\end{align}
Since $W$ is $C_1$-cofinite, $W/\tilW$ is a $C_1$-cofinite highest-weight $\cS(c,0)$-module, and thus by Lemma \ref{lemc1}(1), $W/\W\cong M^\NS(c,h)/J$ (up to a parity reversal) for some $h\in H_c$ and non-zero submodule $J\subset M^\NS(c,h)$. By Corollary \ref{coro}(1), there are two cases to consider: 
\begin{enumerate}
    \item $J\cong M^\NS(c,h')$ for some $h'\in \mathbb{C}$, or
    \item $J\cong M^\NS(c,h_1')+ M^\NS(c,h_2')$    for some $h_1', h_2'\in \mathbb{C}$.
\end{enumerate}
We will show that an element in $\mathcal{U}(\NS^-)w\cap \widetilde{W}$ of sufficiently high conformal weight is in $C_1(\widetilde{W})$ in either case.

    \textbf{Case 1:} $J\cong M^\NS(c,h')$ for $h'\in \mathbb{C}$, so there is an isomorphism $\phi:  M^\NS(c,h)/M^\NS(c,h') \xrightarrow{\sim} W/\W $. By Theorem \ref{AsTh} there exists $U\in \U(\NS^-)$ of weight $h'-h  \in \frac{1}{2}\mathbb{Z}$ such that $U\vac_{c,h}$ is the unique (up to scale) singular vector in $M^\NS(c,h)$ which generates $M^\NS(c,h').$ Then $U\pi(w)=0$ in $W/\W$ since $\pi(w)=\phi (\vac_{c,h}+M^\NS(c,h'))$, and thus $Uw\in \W$. 
    
    Next, we show that any descendant of $w$ in $\W$ must also be a descendant of $Uw$. Indeed, if $\widetilde{U}w\in\W$ for some $\widetilde{U}\in\U(\NS^-)$, then 
\begin{equation*}
    \widetilde{U}\vac_{c,h}+M^\NS(c,h') =\phi^{-1}(\widetilde{U}\pi(w)) =0,
\end{equation*}
so $\widetilde{U}\vac_{c,h}\in M^\NS(c,h')$. The embedding diagrams in Proposition \ref{vermadiag} together with Theorem \ref{AsTh} imply that $M^\NS(c,h')=\U(\NS^-) U\vac_{c,h}$ as a submodule of $M^\NS(c,h)$, so we can write
\begin{align*}
    \widetilde{U}\vac_{c,h}=T U \vac_{c,h}
    \end{align*}
    for some $T\in \U(\NS^-)$. This implies $\widetilde{U}=TU$ since $M^\NS(c,h)$ has a PBW basis, so $\widetilde{U}w=TUw$,  showing that every descendant of $w$ that is in $\W$ must in fact be a descendant of $Uw$.
That is, 
    \begin{align}\label{eqn:intersection}
        \U(\NS^-)w\cap \W=\U(\NS^-)Uw.
    \end{align}
    After normalizing so that $U$ is of the form \eqref{svector}, we can assume that the coefficient of $(G_{-\frac{1}{2}})^{2(h'-h)}$ in $U$ is $1$.

 On the other hand, since $W$ is $C_1$-cofinite, we know that $(G_{-\frac{1}{2}})^{M}w \in C_1(W)$ for $M$ sufficiently large. Thus by \eqref{C1Wsum}, there exist $U^{(n)}, F^{(m)}\in \U(\NS^-)$ and $w'\in C_1(\W)$ such that 
 \begin{align*} 
     (G_{-\frac{1}{2}})^{M}w=\sum_{n\geq 2} L_{-n} U^{(n)} w+ \sum_{m\geq 1} G_{-m-\frac{1}{2}}F^{(m)}w+w',
 \end{align*}
 so that 
\begin{align*}
w'=(G_{-\frac{1}{2}})^{M}w-\sum_{n\geq 2} L_{-n} U^{(n)} w- \sum_{m\geq 1} G_{-m-\frac{1}{2}}F^{(m)}w \in \U(\NS^-)w\cap \W.
\end{align*}
The previous paragraph shows that there exists $T\in\U(\NS^-)$ of weight $K=\frac{M}{2}-h'+h$ such that 
\begin{equation*}
    TU = (G_{-\frac{1}{2}})^{M}-\sum_{n\geq 2} L_{-n} U^{(n)} - \sum_{m\geq 1} G_{-m-\frac{1}{2}}F^{(m)}.
\end{equation*}
Then the coefficient of $(G_{-\frac{1}{2}})^{2K}$ in $T$ is $1$, so we can write
\begin{align} \label{u2}
T=(G_{-\frac{1}{2}})^{2K}+T'
\end{align}
where $T' \in \sum_{n \in \mathbb{Z}_{\geq 2}} L_{-n}\U(\NS^-)+ \sum_{m\in \mathbb{Z}_{\geq 1}}G_{-m-\frac{1}{2}}\U(\NS^-)$. In particular, $T'Uw\in C_1(\W)$ because $Uw\in \widetilde{W}$. Now using \eqref{u2},
\begin{align*}
(G_{-\frac{1}{2}})^{2K}Uw=TUw-T'Uw=w'-T'Uw \in C_1(\W).  
\end{align*}
By \eqref{eqn:intersection}, this is enough to show that any descendant of $w$ of the form $\widetilde{U}w$ that is in $\W$ will also be in $C_1(\W)$ if the weight of $\widetilde{U}\in \U(\NS^-)$ is large enough.

 \textbf{Case 2:} $J\cong M^\NS(c, h_1')+ M^\NS(c, h_2')$. We modify the proof of Case 1 to this slightly more complicated setting. With an analogous argument we first obtain two elements  $U_1, U_2\in\U(\NS^-)$ (instead of just $U$ as in Case 1) such that  
\begin{align} \label{suf}
\U(\NS^-)w\cap \W=\U(\NS^-)U_1w+\U(\NS^-)U_2w,
\end{align}
with
\begin{align*}
U_i & \in(G_{-\frac{1}{2}})^{2(h_i'-h)}+\sum_{n \in \mathbb{Z}_{\geq 2}} L_{-n}\U(\NS^-)+ \sum_{m\in \mathbb{Z}_{\geq 1}}G_{-m-\frac{1}{2}}\U(\NS^-)
\end{align*}
for $i=1,2$. Thus for all sufficiently large $M$, we have
\begin{align} \label{Mchoice} 
     (G_{-\frac{1}{2}})^{M}w=\sum_{n=2}^{\infty}L_{-n} U^{(n)} w+ \sum_{m=1}^\infty G_{-m-\frac{1}{2}}F^{(m)}w+w'
 \end{align}
as before, where we can now write $w'\in C_1(\W)$ as
\begin{align*}
w'=T_1U_1w+T_2U_2w
\end{align*}
for elements $T_1, T_2 \in \U(\NS^-)$ of weights $K_i=\frac{M}{2}-h_i'+h$ for $i=1,2$.

In this case we can only conclude that there exist $a, b \in \mathbb{C}$ such that $a+b=1$ and  
\begin{align} \label{inc1}
    a(G_{-\frac{1}{2}})^{2K_1}U_1w+b(G_{-\frac{1}{2}})^{2K_2}U_2w   \in C_1(\W).                   
\end{align}
From the embedding diagrams (see Corollary \ref{coro}(4)), $U_1\vac_{c,h}$ and $U_2\vac_{c,h}$ in $M^\NS(c,h)$ have a common descendant $U_3\vac_{c,h}$, and thus there exist $S_1, S_2 \in \U(\NS^-)$ such that 
\begin{align} \label{decu3}
U_3=S_1U_1=S_2U_2.    
\end{align}
By Theorem \ref{svector}, we have 
\begin{align*}
S_i\in (G_{-\frac{1}{2}})^{2(h_3-h_i')}+\sum_{n \in \mathbb{Z}_{\geq 2}} L_{-n}\U(\NS^-)+ \sum_{m\in \mathbb{Z}_{\geq 1}}G_{-m-\frac{1}{2}}\U(\NS^-)   
\end{align*}
for $i=1,2$, where $h_3$ is the weight of $U_3$. Moreover, if we assume $M$ in \eqref{Mchoice} is large enough, then we can multiply $U_3$ by a power of $G_{-\frac{1}{2}}$ if necessary so that $h_3-h_i'=K_i$ for $i=1,2$. This leads to
\begin{equation}\label{last}
(G_{-\frac{1}{2}})^{2K_i}U_iw\in S_iU_iw +C_1(\W)
\end{equation}
for $i=1,2$. Combining \eqref{inc1}, \eqref{decu3}, and \eqref{last}, we then get
\begin{align*}
    S_iU_iw & =(a+b)S_iU_iw = aS_1U_1 w+ bS_2 U_2 w\nonumber\\
    &\in a(G_{-\frac{1}{2}})^{2K_1}U_1w+b(G_{-\frac{1}{2}})^{2K_2}U_2w  + C_1(\W) = C_1(\W)
\end{align*}
for $i=1,2$. This plus \eqref{last} then implies $(G_{-\frac{1}{2}})^{2K_i}U_iw\in C_1(\W)$ for $i=1,2$. This holds for all sufficiently large $K_1$ and $K_2$, so using \eqref{suf} we find that any element in $\mathcal{U}(\NS^-)w\cap \widetilde{W}$ of sufficiently large weight is in $C_1(\widetilde{W})$, proving the proposition.

\end{proof}

We now proceed to show that any $C_1$-cofinite lower-bounded generalized $\cS(c,0)$-module has finite length and that its composition factors must be $C_1$-cofinite.
The next result was proved in \cite{CJORY} for vertex operator algebras; the proof for superalgebras is essentially the same, so we omit the proof.
\begin{lemma} [\cite{CJORY} Lemma 3.1.1]\label{le}
Let $V$ be a vertex operator superalgebra and let $W$ be a generalized $V$-module.
\begin{enumerate}
    \item If $0\neq w\in W$ has $L_0$-eigenvalue $h \in \mathbb{C}$ and $h-\frac{1}{2}n$ is not an $L_0$-eigenvalue for any $n\in \mathbb{Z}_{\geq 1}$, then $w\notin C_1(W)$.
    \item If $W$ is lower bounded and $W=C_1(W)$, then $W=0$.
\end{enumerate}
\end{lemma}

The proof of the following result is essentially the same as that of \cite[Propositions 3.1.2 and 3.1.3]{CJORY}, with some adjustments for the super Virasoro case:
\begin{prop}\label{prop:mainprop}
If $W$ is a $C_1$-cofinite lower-bounded generalized $\cS(c,0)$-module, then:
\begin{enumerate}
    \item There exists a sequence of lower-bounded generalized $\cS(c,0)$-submodules
    $$0=W_0\subset W_1\subset \dots \subset W_{n-1}\subset W_n=W$$ such that $W_i/W_{i-1}$ is a  highest-weight $\cS(c,0)$-module for $i=1, \dots n.$
    \item For any such filtration of $W$, each $W_i$ is $C_1$-cofinite.
\end{enumerate}
\end{prop}
\begin{proof}
We omit the proof of part (1) since it is almost exactly the same as that of \cite[Proposition 3.1.2]{CJORY} (see also the proof of \cite[Theorem 3.7]{CY}).



To prove (2), we will show that if $W_{i+1}$ is $C_1$-cofinite for some fixed $i=1,\ldots, n-1$, then $W_{i}$ is also $C_1$-cofinite. Choose $w_{j}\in W_j$ for $1\leq j\leq n$ such that $w_j+W_{j-1}\in W_j/W_{j-1}$ is a generating highest-weight vector. If $j<i+1$, then $W_j\subset W_{i+1}$, so because $W_{i+1}$ is $C_1$-cofinite, this implies $(G_{-\frac{1}{2}})^{N_j}w_j\in C_1(W_{i+1})$ for $N_j$ large enough. By (\ref{C1Wsum}) this implies
\begin{align} \label{C1i+1}
(G_{-\frac{1}{2}})^{N_j}w_j=\sum_{n\geq 2} L_{-n} U_j^{(n)} w_{i+1}+ \sum_{m\geq 1} G_{-m-\frac{1}{2}}F_j^{(m)}w_{i+1}+w_j'   
\end{align}
for certain $U_j^{(n)}, F_j^{(m)}\in \U(\NS^{-})$ and $w_j'\in C_1(W_j)$. Thus, if $j<i+1$, then 
\begin{align*}
\sum_{n\geq 2} L_{-n} U_j^{(n)} w_{i+1}+ \sum_{m\geq 1} G_{-m-\frac{1}{2}}F_j^{(m)}w_{i+1}=(G_{-\frac{1}{2}})^{N_j}w_j-w_j'\in \U(\NS^-)w_{i+1}\cap W_j.
\end{align*}
Applying Proposition \ref{desclarge} with $W=W_{i+1}, \widetilde{W}=W_i$ and $w=w_i$, we get 
$$\sum_{n\geq 2} L_{-n} U_j^{(n)} w_{i+1}+ \sum_{m\geq 1} G_{-m-\frac{1}{2}}F_j^{(m)}w_{i+1}\in C_1(W_i)$$ 
for sufficiently large $N_j$. This together with \eqref{C1i+1} implies that for all $j<i+1$,
\begin{align} \label{j<i+1}
 (G_{-\frac{1}{2}})^{N_j}w_j\in C_1(W_i)
\end{align}
for sufficiently large $N_j$.
Iterating \eqref{seq} yields 
\begin{align*}
&W_{i}=\sum_{j=1}^i \U(\NS^-)w_j, \textrm{ and }\\ &C_1(W_i)=    \sum_{j=1}^i\sum_{n \in \mathbb{Z}_{\geq 2}} L_{-n} \U(\NS^-)w_j+ \sum_{j=1}^i\sum_{m\in \mathbb{Z}_{\geq 1}} G_{-m-\frac{1}{2}}\U(\NS^-)w_j.
\end{align*}
This together with \eqref{j<i+1} implies that the vectors $(G_{-\frac{1}{2}})^N w_j+C_1(W_i)$ for $1\leq j\leq i$ and  $0\leq N< N_j$ span $W_i/C_1(W_i)$. Thus $W_i$ is $C_1$-cofinite if $W_{i+1}$ is, and this proves part (2) of the proposition since $W_n=W$ is $C_1$-cofinite.

\end{proof} 

Using the above proposition, we now obtain the main theorem of this section:
\begin{theorem}\label{thm:mainthm}
The category $\cC_1(\cS(c,0))$ of $C_1$-cofinite grading-restricted generalized $\cS(c,0)$-modules is the same as the category $\cO_c^{\mathrm{fin}}$ of finite-length grading-restricted generalized $\cS(c,0)$-modules whose composition factors are not isomorphic to Verma modules or their parity reversals.
\end{theorem}
\begin{proof}
Proposition \ref{prop:mainprop} and Lemma \ref{lemc1}(3) show that any module in $\cC_1(\cS(c,0))$ has finite length with composition factors of the form $\Pi^n(L^\NS(c,h))$ for $h \in H_c$ and $n=0,1$. Thus $\cC_1(\cS(c,0))\subset\cO_c^{\mathrm{fin}}$. Conversely, Lemma \ref{C1H}(2) shows that any finite-length module with $C_1$-cofinite composition factors is $C_1$-cofinite, so $\cO^{\mathrm{fin}}_c\subset\cC_1(\cS(c,0))$.

\end{proof}

Since every simple grading-restricted $\cS(c,0)$-module is self-contragredient, and since Theorem \ref{thm:mainthm} shows $\cC_1(\cS(c,0))$ equals the category $\cO^{\mathrm{fin}}_c$ finite-length $\cS(c,0)$-modules whose composition factors are $C_1$-cofinite, Theorem \ref{thm:vosa-btc} implies that $\cO_c^{\mathrm{fin}}$ is a braided tensor category:

\begin{cor}\label{cor:tensor}
    For any central charge $c\in\CC$, the category $\cO_c^{\mathrm{fin}}$ of finite-length grading-restricted generalized $V$-modules with $C_1$-cofinite composition factors admits the braided tensor category structure of \cite{HLZ1}-\cite{HLZ8}.
\end{cor}

In the following subsections, we discuss some of the detailed structure of $\cO_c^{\mathrm{fin}}$ for central charges at which semisimple tensor categories occur.

\subsection{Tensor structure at central charge \texorpdfstring{$c^\NS(t)$}{cNS(t)} for irrational \texorpdfstring{$t$}{t}}
In this subsection we consider the tensor category $\ocfin$ when
\begin{align} \label{genc}
c=c^\NS(t)=\frac{15}{2}-3(t+t^{-1}), \quad t\in \mathbb{C}\setminus \mathbb{Q}.    
\end{align}
For these central charges, the Verma module embedding diagrams are particularly simple (see Proposition \ref{vermadiag}(3)). 

We start by proving that $\ocfin$ for $c=c^\NS(t)$, $t\notin\QQ$,
is semisimple, similar to \cite[Theorem 5.1.2]{CJORY}. First note that the simple objects in $\ocfin$ are the irreducible modules $\cS_{r,s}:=L^\NS(c, h^\NS_{r,s})$ with $r,s\in \mathbb{Z}_{\geq 1}$ such that $r-s\in 2\mathbb{Z}$, together with their parity reversals. Since $t\notin\QQ$, $h^\NS_{r,s} = h^\NS_{r',s'}$ for $r,s,r',s'\in\ZZ_{\geq 1}$ if and only if $r=r'$ and $s=s'$.
\begin{theorem}\label{thm:t-irrational-simple}
If $c=c^\NS(t)$ for $t\notin\QQ$, then $\ocfin$ is semisimple with distinct simple objects $\Pi^n(\cS_{r,s})$ for $n=0,1$ and $r,s\in \mathbb{Z}_+$ such that $r-s\in 2\mathbb{Z}$.
\end{theorem}
\begin{proof}
    Since every module in $\ocfin$ has finite length, it is enough to show that any short exact sequence
    \begin{align*}
    0\longrightarrow \Pi^{n_1}(L^\NS(c,h_1))\longrightarrow M \longrightarrow \Pi^{n_2}(L^\NS(c,h_2))\longrightarrow 0
\end{align*}
such that $M$ is a module in $\ocfin$, $n_1,n_2\in\ZZ$, and $h_1,h_2\in H_c$ splits. Because operators in $\NS$ change conformal weights by half-integers, any such exact sequence splits if $h_1-h_2\notin\frac{1}{2}\ZZ$. Moreover, by \cite[Lemma 6.2.2]{GK}, simple modules in $\ocfin$ do not have non-split self-extensions, so the exact sequence splits when $h_1=h_2$. 

If $h_1-h_2\in\frac{1}{2}\ZZ_{\geq 1}$, then there is a highest-weight vector $v_{h_2}\in M$ of conformal weight $h_2$. If $N$ is the submodule of $M$ generated by $v_{h_2}$, then $N$ is a quotient of $\Pi^{n_2}(M^\NS(c, h_2))$. As $t$ is irrational, the embedding diagram in Proposition \ref{vermadiag}(3) shows that $N$ is isomorphic to either $\Pi^{n_2}(M^\NS(c, h_2))$ or $\Pi^{n_2}(L^\NS(c,h_2))$. Since Verma modules are not objects of $\ocfin$, we get $N\cong \Pi^{n_2}(L^\NS(c,h_2))$ and the sequence splits.

If $h_2-h_1\in\frac{1}{2}\ZZ_{\geq 1}$, we take contragredient duals to get a short exact sequence
\begin{align*}
    0\longrightarrow \Pi^{n_2}(L^\NS(c,h_2))\longrightarrow M' \longrightarrow \Pi^{n_1}(L^\NS(c,h_1))\longrightarrow 0.
\end{align*}
Then by the previous paragraph, $M'\cong \Pi^{n_1}(L^\NS(c,h_1))\oplus \Pi^{n_2}(L^\NS(c,h_2))$, and thus $M\cong \Pi^{n_1}(L^\NS(c,h_1))\oplus \Pi^{n_2}(L^\NS(c,h_2))$ as well.

\end{proof}

Our next goal is to determine the fusion rules for the simple modules in the semisimple tensor category $\ocfin$ for $c=c^\NS(t)$, $t\notin\QQ$. Actually, these fusion rules were already determined in \cite[Example 7.11]{CMY-completions} for the $\Pi$-supercategory version of $\ocfin$ (where we allow odd as well as even morphisms), except that it was not established in \cite{CMY-completions} that the results apply to every central charge $c^\NS(t)$, $t\notin\QQ$. To understand why, we next recall the methods of \cite[Example 7.11]{CMY-completions}.

For $\ell \in \CC$, let $L(c_\ell,0)$ be the simple Virasoro vertex operator algebra at central charge 
\begin{align*}
c_{\ell}:=13-6(\ell+\ell^{-1}).
\end{align*}
Its $C_1$-cofinite simple modules $L(c_\ell,h_{r,s}(\ell))$ for $r,s\in\ZZ_{\geq 1}$ have lowest conformal weights
\begin{equation*}
 h_{r,s}(\ell) =\frac{r^2-1}{4} \ell-\frac{rs-1}{2}+\frac{s^2-1}{4}\ell^{-1}.
\end{equation*}
By \cite[Theorem~4.2.6]{CJORY}, the category of finite-length $L(c_\ell,0)$-modules with composition factors $L(c_\ell,h_{r,s}(\ell))$ for $r,s\in\ZZ_{\geq 1}$ is a braided tensor category.

For all $t\neq 0,-1$, let $a=\frac{1}{2}(t+1)$ and $b=\frac{1}{2}(t^{-1}+1)$. The Virasoro vertex operator algebras $L(c_a, 0)$ and $L(c_b,0)$ form a commuting pair of subalgebras in $\cS(c^\NS(t), 0) \otimes \cF(1)$, where $\cF(1)$ is the rank $1$ free fermion vertex superalgebra of central charge $\frac{1}{2}$. See Appendix \ref{N1FERM} for explicit expressions (involving $\sqrt{t}$) for the conformal vectors of $L(c_a,0)$ and $L(c_b,0)$ in $\cS(c^{\NS}(t),0)\otimes\cF(1)$. We can replace $t$ with a formal variable $x$ and view $\cS(c^{\NS}(x),0)$, $\cF(1)$, $L(c_a,0)$, and $L(c_b,0)$ as vertex operator (super)algebras over the field $\CC(x^{1/2})$; the explicit expressions in Appendix \ref{N1FERM} show that we still have an embedding
\begin{equation*}
    L(c_a,0)\otimes L(c_b,0)\hookrightarrow\cS(c^{\NS}(x),0)\otimes\cF(1).
\end{equation*}
Moreover, it follows from Corollary 2.6 and Theorem 2.10  of \cite{CGL} (see also the proof of \cite[Lemma 2.11]{CGL}) that
\begin{align}
\label{dec1}  \cS(c^\NS(x), 0) \otimes \cF(1) &\cong \bigoplus_{n=1}^\infty L(c_{a}, h_{1, n}(a))\otimes L(c_b, h_{1, n}(b))
\end{align}
as $L(c_a,0)\otimes L(c_b,0)$-modules over $\CC(x^{1/2})$. This means that the decomposition still holds after specializing $x\mapsto t$ for all but countably many $t\in\CC$. The methods of \cite[Example 7.11]{CMY-completions} yield rigidity and fusion rules of $\cO^{\mathrm{fin}}_{c^{\NS}(t)}$ for all irrational $t$ such that \eqref{dec1} holds, so we just need to prove this decomposition for all $t\notin\QQ$:
\begin{prop}\label{prop:decomp}
    For all $t\in\CC\setminus\QQ$,
    \begin{equation}\label{eqn:irrational-dec}
        \cS(c^\NS(t), 0) \otimes \cF(1) \cong \bigoplus_{n=1}^\infty L(c_{a}, h_{1, n}(a))\otimes L(c_b, h_{1, n}(b)) 
    \end{equation}
as $L(c_a,0)\otimes L(c_b,0)$-modules, where $a=\frac{1}{2}(t+1)$ and $b=\frac{1}{2}(t^{-1}+1)$.
\end{prop}
\begin{proof}
    The decomposition \eqref{dec1} implies that for each $n\in\ZZ_{\geq 1}$, $\cS(c^{\NS}(x),0)\otimes\cF(1)$ contains a non-zero vector $v_n$ of conformal weight $h_{1,n}(a)+h_{1,n}(b)=\frac{(n-1)^2}{2}$ which is highest weight for both $L(c_a,0)$ and $L(c_b,0)$. Each $v_n$ is a $\CC(x^{1/2})$-linear combination of basis elements
    \begin{equation}\label{eqn:PBW-basis-for-N1-times-Ferm}
L_{-i_1-2}\cdots L_{-i_l-2} G_{-j_1-\frac{3}{2}}\cdots G_{-j_m-\frac{3}{2}}\vac_\NS\otimes\psi_{-k_1-\frac{1}{2}}\cdots\psi_{-k_n-\frac{1}{2}}\vac_{\cF(1)},
    \end{equation}
    where $i_1\geq\cdots\geq i_l\geq 0$, $j_1>\cdots > j_m\geq 0$, $k_1>\cdots >k_n\geq 0$, and $\psi$ is the strong generator of $\cF(1)$ of conformal weight $\frac{1}{2}$. By clearing denominators, we may assume that the coefficients of each basis element in the linear combination for $v_n$ are polynomials in $x^{1/2}$. Then by dividing out common factors, we may assume that for each $v_n$, these polynomial coefficients have no common roots. In particular, each $v_n$ is in the $\CC[x^{\pm 1/2},(x+1)^{-1}]$-form of $\cS(c^{\NS}(x),0)\otimes\cF(1)$ spanned over $\CC[x^{\pm 1/2},(x+1)^{-1}]$ by the basis elements \eqref{eqn:PBW-basis-for-N1-times-Ferm}. Note from the formulas in Appendix \ref{N1FERM} that $L(c_a,0)\otimes L(c_b,0)$ preserves this $\CC[x^{\pm 1/2},(x+1)^{-1}]$-form.

    Specializing $x^{1/2}\mapsto t^{1/2}$ for $t\in\CC\setminus\lbrace 0,-1\rbrace$ (using either square root of $t$), yields a vertex operator superalgebra homomorphism from the $\CC[x^{\pm 1/2},(x+1)^{-1}]$-form of $\cS(c^{\NS}(x),0)\otimes\cF(1)$ to $\cS(c^{\NS}(t),0)\otimes\cF(1)$. By our choice of normalization for $v_n$, the image of $v_n$ under this homomorphism is non-zero since the elements \eqref{eqn:PBW-basis-for-N1-times-Ferm} also provide a basis for $\cS(c^{\NS}(t),0)\otimes\cF(1)$. Thus for $t\neq 0,-1$, $\cS(c^{\NS}(t),0)\otimes\cF(1)$ contains a highest-weight $L(c_a,0)\otimes L(c_b,0)$-vector of conformal weights $(h_{1,n}(a),h_{1,n}(b))$ for each $n\in\ZZ_{\geq 1}$; let $M_n$ be the $L(c_a,0)\otimes L(c_b,0)$-submodule of $\cS(c^{\NS}(t),0)\otimes\cF(1)$ generated by this highest-weight vector.

    Each $M_n$ is a quotient of the tensor product of two Virasoro Verma modules, and thus by the embedding diagrams for Virasoro Verma modules (see for example \cite[Section 5.3.1]{IK-Virasoro}), the composition factors of $M_n$ are contained among
    \begin{align*}
       & L(c_a,h_{1,n}(a))\otimes L(c_a,h_{1,n}(b)),\quad L(c_a,h_{1,n}(a)+n)\otimes L(c_a,h_{1,n}(b)),\nonumber\\
       & L(c_a,h_{1,n}(a))\otimes L(c_a,h_{1,n}(b)+n), \quad L(c_a,h_{1,n}(a)+n)\otimes L(c_a,h_{1,n}(b)+n).
    \end{align*}
    For $t\notin\QQ$, the composition factors of $M_m$ and $M_n$ are disjoint unless $m=n$, and this implies $\cS(c^{\NS}(t),0)\otimes\cF(1)$ contains $\bigoplus_{n=1}^\infty M_n$ as a submodule. Indeed, otherwise there would be a subset $\lbrace n_i\rbrace_{i=0}^I\subset\ZZ_{\geq 0}$ such that $M_{n_0}\cap\sum_{i=1}^I M_{n_i}\neq 0$, and this would imply $M_{n_0}$ shares a composition factor with some $n_i$ for $i\geq 1$. Thus for $t\notin\QQ$, $\cS(c^{\NS}(t),0)\otimes\cF(1)$ contains $\bigoplus_{n=1}^\infty L(c_a, h_{1,n}(a))\otimes L(c_b, h_{1,n}(b))$ as an $L(c_a,0)\otimes L(c_b,0)$-module subquotient.

    To complete the proof, it now suffices to show that both sides of \eqref{eqn:irrational-dec} have the same character. The basis \eqref{eqn:PBW-basis-for-N1-times-Ferm} implies that
    \begin{equation*}
        \mathrm{ch}[\cS(c^{\NS}(t),0)\otimes\cF(1)] = q^{-(8-3t-3t^{-1})/24}\prod_{m=1}^\infty \frac{(1+q^{m+1/2})(1+q^{m-1/2})}{1-q^{m+1}},
    \end{equation*}
    while the embedding diagrams for Virasoro Verma modules for $a,b\notin\QQ$ imply
    \begin{equation*}
        \mathrm{ch}\left[\bigoplus_{n=1}^\infty L(c_{a}, h_{1, n}(a))\otimes L(c_b, h_{1, n}(b))\right] = q^{-(8-3t-3t^{-1})/24}\sum_{n=1}^\infty \frac{q^{(n-1)^2/2}(1-q^n)^2}{\prod_{m=1}^\infty (1-q^m)^2}.
    \end{equation*}
    Multiplying both expressions by $q^{(8-3t-3t^{-1})/24}(1+q^{1/2})\prod_{m=1}^\infty (1-q^{m+1})(1-q^m)$, we see it is enough to prove
    \begin{equation}\label{eqn:N1-ferm-char-calc}
        \prod_{m=1}^\infty (1-q^m)(1+q^{m-1/2})^2 =\sum_{n=1}^\infty \frac{q^{(n-1)^2/2}(1-q^n)^2}{1-q^{1/2}}.
    \end{equation}
    The right side is
    \begin{align*}
        \sum_{n=0}^\infty \frac{q^{n^2/2}(1-q^{n+1})^2}{1-q^{1/2}} & =\sum_{n=0}^\infty \frac{(q^{n^2/2}-q^{(n+1)^2/2+1/2})(1-q^{n+1})}{1-q^{1/2}}\nonumber\\
        &= \sum_{n=0}^\infty \frac{q^{n^2/2}(1-q^{n+1})}{1-q^{1/2}} -\sum_{n=1}^\infty \frac{q^{n^2/2} q^{1/2}(1-q^n)}{1-q^{1/2}}\nonumber\\
        & = \frac{1-q}{1-q^{1/2}} +\sum_{n=1}^\infty \frac{q^{n^2/2}(1-q^{1/2}+q^{n+1/2}-q^{n+1})}{1-q^{1/2}}\nonumber\\
        & = 1+q^{1/2}+\sum_{n=1}^\infty q^{n^2/2}(1+q^{n+1/2})\nonumber\\
        & = \sum_{n=0}^\infty (q^{n^2/2}+q^{(n+1)^2/2}) =\sum_{n=-\infty}^\infty q^{n^2/2}.
    \end{align*}
    Now \eqref{eqn:N1-ferm-char-calc} follows from the Jacobi triple product identity, completing the proof.
    
\end{proof}

Now that we have the decomposition \eqref{dec1} for all irrational $t$, the same arguments as in \cite[Example~7.11]{CMY-completions} prove rigidity and determine fusion rules for the category $\ocfin$. To explain in more detail, let $\cC_{c_a}^{Vir}$ and $\cC_{c_b}^{Vir}$ be the $C_1$-cofinite module categories for the Virasoro vertex operator algebras $L(c_a,0)$ and $L(c_b,0)$. Then the category of $C_1$-cofinite $L(c_a,0)\otimes L(c_b,0)$-modules is braided tensor equivalent to the Deligne tensor product $\cC_{c_a}^{Vir} \boxtimes \cC_{c_{b}}^{Vir}$ \cite[Theorem 1.2]{McR-Deligne}, and there is a monoidal induction functor
\[
F: \cC_{c_a}^{Vir} \boxtimes \cC_{c_{b}}^{Vir} \longrightarrow {\rm Rep}\; S(c^\NS(t),0) \otimes \cF(1).
\]
Here Rep $S(c^\NS(t),0) \otimes \cF(1)$ is the category of not necessarily local $S(c^\NS(t),0) \otimes \cF(1)$-modules which,  as $L(c_a,0)\otimes L(c_b,0)$-modules, are objects of the direct limit completion of $\cC_{c_a}^{Vir} \boxtimes \cC_{c_{b}}^{Vir}$ (see \cite{CKM-exts, CMY-completions} for the definitions of these categories).

  By \cite[Proposition~4.4]{CKM-exts}, the module $F(L(c_{a}, h_{r, 1}(a))\otimes L(c_b, h_{s, 1}(b)))$ is simple for $r,s\in\ZZ_{\geq 1}$, and as in \cite[Example~7.11]{CMY-completions}, this simple module is local if and only if $r+s\in 2\ZZ$. (If $r+s\in 2\ZZ+1$, then it is a Ramond sector parity-twisted $S(c^\NS(t),0) \otimes \cF(1)$-module.) Also,
\begin{equation}\label{eqn:induction-decomp}
    F(L(c_{a}, h_{r, 1}(a))\otimes L(c_b, h_{s, 1}(b))) \cong \bigoplus_{n=1}^\infty L(c_{a}, h_{r, n}(a))\otimes L(c_b, h_{s, n}(b))
\end{equation}
as $L(c_a,0)\otimes L(c_b,0)$-modules, with lowest conformal weight $h_{r,s}^{\NS}(t)$ occurring in the $n=\frac{r+s}{2}$ summand. In this decomposition, the $n$th summand has the parity of $n-1$, and thus
\begin{equation}\label{eqn:identify-induction}
 F(L(c_{a}, h_{r, 1}(a))\otimes L(c_b, h_{s, 1}(b))) \cong \Pi^{(r+s-2)/2}(\cS_{r,s})\otimes \cF(1) .
\end{equation}
Let  $(\cC_{c_a}^{Vir} \boxtimes \cC_{c_{b}}^{Vir})_0\subset\cC_{c_a}^{Vir} \boxtimes \cC_{c_{b}}^{Vir}$ be the semisimple tensor subcategory with simple objects $L(c_{a}, h_{r, 1}(a))\otimes L(c_b, h_{s, 1}(b))$ such that $r+s\in 2\ZZ$. Then \eqref{eqn:identify-induction} implies that the essential image of the restriction of $F$ to $(\cC_{c_a}^{Vir} \boxtimes \cC_{c_{b}}^{Vir})_0$ is contained in the full subcategory $\widetilde{\mathcal{O}}\subset {\rm Rep}\; S(c^\NS(t),0) \otimes \cF(1)$ consisting of all modules isomorphic to $W\otimes\cF(1)$ for some $W$ in $\cO^{\mathrm{fin}}_{c^{\NS}(t)}$. Note that the functor $\cO^{\mathrm{fin}}_{c^{\NS}(t)}\rightarrow\widetilde{\cO}$ given on objects by $W\mapsto W\otimes\cF(1)$ and on morphisms by $f\mapsto f\otimes\id_{\cF(1)}$ is a braided tensor equivalence by properties of the Deligne product braided tensor category.

As in \cite[Example 7.11]{CMY-completions}, we now have a fully faithful braided tensor functor
\begin{equation}\label{eqn:Vir-Vir-0-embedding}
(\cC_{c_a}^{Vir} \boxtimes \cC_{c_{b}}^{Vir})_0  \xrightarrow{F} \widetilde{\cO} \xrightarrow{\sim} \cO_{c^\NS(t)}^{\rm fin}
\end{equation}
which by \eqref{eqn:induction-decomp} and \eqref{eqn:identify-induction} maps
\begin{equation}\label{eqn:embedding-on-objects}
    L(c_{a}, h_{r, 1}(a))\otimes L(c_b, h_{s, 1}(b)) \longmapsto \Pi^{(r+s-2)/2}(\cS_{r,s}).
\end{equation}
Using the rigidity and fusion rules of $(\cC_{c_a}^{Vir} \boxtimes \cC_{c_{b}}^{Vir})_0$, we then obtain the following theorem as in \cite[Example 7.11]{CMY-completions}, but here we present the $\Pi$-category rather than $\Pi$-supercategory version of the result:

\begin{theorem}\label{thm:t-irrational-properties}
    Let $c=c^{\NS}(t)$ for $t\notin \QQ$. Then:
    \begin{itemize}
        \item[(1)] The semisimple tensor subcategory of $\cO_{c}^{\rm fin}$ with simple objects $\Pi^{(r+s-2)/2}(\cS_{r,s})$ for $r,s\in\ZZ_{\geq 1}$ such that $r+s\in 2\ZZ$ is braided tensor equivalent to $(\cC_{c_a}^{Vir} \boxtimes \cC_{c_{b}}^{Vir})_0$.
        \item[(2)] The category $\cO_{c}^{\rm fin}$ is rigid and slightly degenerate, that is, its M\"{u}ger center is semisimple with simple objects $\cS_{1,1}$ and $\Pi(\cS_{1,1})$.
        \item[(3)] For $r,s, r',s' \in \ZZ_{\geq 1}$ such that $r+s, r'+s' \in 2\ZZ$,
        \begin{equation}\label{eqn:irr-fus-rules}
        \cS_{r,s}\boxtimes \cS_{r',s'} \cong \bigoplus_{\substack{r'' = |r-r'|+1\\ r+r'+r''\, {\rm odd}}}^{r+r'-1}\bigoplus_{\substack{s'' = |s-s'|+1\\ s+s'+s''\, {\rm odd}}}^{s+s'-1} \Pi^{(r+r'-r''+s+s'-s''-2)/2}(\cS_{r'',s''}).
        \end{equation}
        \item[(4)] For $r,s\in\ZZ_{\geq 1}$ such that $r+s\in 2\ZZ$,
        \begin{equation}\label{branch} 
          \cS_{r,s} \otimes \cF(1) \cong \bigoplus_{n=1}^\infty L(c_{a}, h_{r, n}(a))\otimes L(c_b, h_{s, n}(b)).   
         \end{equation}
         as an $L(c_a,0)\otimes L(c_b,0)$-module.
        \end{itemize}
    \end{theorem}
    \begin{proof}
(1) follows from \eqref{eqn:Vir-Vir-0-embedding} and \eqref{eqn:embedding-on-objects}, and then (1) implies $\cO^{\mathrm{fin}}_c$ is rigid since $(\cC_{c_a}^{Vir} \boxtimes \cC_{c_{b}}^{Vir})_0$ is rigid \cite[Section 5]{CJORY} and parity reversals of rigid modules are rigid.  For (3), the monoidal embedding $(\cC_{c_a}^{Vir} \boxtimes \cC_{c_{b}}^{Vir})_0\hookrightarrow\cO_c^{\mathrm{fin}}$ and the fusion rules from \cite{CJORY} imply
        \begin{equation*}
            \Pi^{(r+s-2)/2}(\cS_{r,s})\boxtimes\Pi^{(r'+s'-2)/2}(\cS_{r',s'}) \cong \bigoplus_{\substack{r'' = |r-r'|+1\\ r+r'+r''\, {\rm odd}}}^{r+r'-1}\bigoplus_{\substack{s'' = |s-s'|+1\\ s+s'+s''\, {\rm odd}}}^{s+s'-1} \Pi^{(r''+s''-2)/2}(\cS_{r'',s''}).
        \end{equation*}
        Then Proposition \ref{prop:Pi-and-fusion} yields \eqref{eqn:irr-fus-rules}.
        
        For the slight degeneracy statement in (2), suppose $\Pi^n(\cS_{r,s})$ is in the M\"{u}ger center of $\cO^{\mathrm{fin}}_c$. Then in particular the double braiding
        \begin{equation*}
            \cR^2: \Pi^n(\cS_{r,s})\boxtimes\Pi^n(\cS_{2,2}) \xrightarrow{\sim} \Pi^n(\cS_{2,2})\boxtimes\Pi^n(\cS_{r,s}) \xrightarrow{\sim} \Pi^n(\cS_{r,s})\boxtimes\Pi^n(\cS_{2,2})
        \end{equation*}
        is the identity, and by \eqref{eqn:irr-fus-rules} and Proposition \ref{prop:Pi-and-fusion}, 
        \begin{align*}
        \Pi^n(\cS_{r,s})\boxtimes\Pi^n(\cS_{2,2}) &\cong \cS_{r+1,s+1}\oplus\Pi(\cS_{r+1,s-1})\oplus\Pi(\cS_{r-1,s+1})\oplus\cS_{r-1,s-1}\\
        & = \bigoplus_{\delta,\varepsilon\in\lbrace\pm 1\rbrace} \Pi^{(1-\delta\varepsilon)/2}(\cS_{r+\delta,s+\varepsilon}),
        \end{align*}
with the convention $\cS_{r',s'}=0$ if $r'=0$ or $s'=0$. By \eqref{eqn:braiding} and the $L_0$-conjugation property for intertwining operators,
\begin{align*}
    \cR^2(\left(\cY_{\Pi^n(\cS_{r,s}),\Pi^n(\cS_{2,2})}(w_{r,s},x)w_{2,2}\right) & = \cY_{\Pi^n(\cS_{r,s}),\Pi^n(\cS_{2,2})}(w_{r,s},e^{2\pi i}x)w_{2,2}\nonumber\\
    &= e^{2\pi i(L_0-\mathrm{wt}\,w_{r,s}-\mathrm{wt}\,w_{2,2})}\cY_{\Pi^n(\cS_{r,s}),\Pi^n(\cS_{2,2})}(w_{r,s},x)w_{2,2}
\end{align*}
for homogeneous $w_{r,s}\in\Pi(\cS_{r,s}), w_{2,2}\in\Pi(\cS_{2,2})$. So because the tensor product intertwining operator $\cY_{\Pi^n(\cS_{r,s}),\Pi^n(\cS_{2,2})}$ is even, this means
\begin{equation*}
    \cR^2\vert_{\Pi^{(1-\delta\varepsilon)/2}(\cS_{r+\delta,s+\varepsilon})} = e^{2\pi i(h_{r+\delta,s+\varepsilon}^\NS+(1-\delta\varepsilon)/4-h_{r,s}^\NS-h_{2,2}^\NS)} =e^{\pi i\left((r\delta-1)t-(r\varepsilon+ s\delta-2)+(s\varepsilon-1)t^{-1}\right)/2}.
\end{equation*}
As in \cite[Example 7.11]{CMY-completions}, when $r=1$ and $\delta=1$, this is not the identity unless $s=1$. If $r>1$, then $\cR^2\vert_{\cS_{r+1,s+1}}$ could be the identity, but then $\cR^2\vert_{\Pi(\cS_{r-1,s+1})}$ will not be. Thus $\Pi^n(\cS_{r,s})$ is not in the M\"{u}ger center of $\cO^{\mathrm{fin}}_c$ if $(r,s)\neq(1,1)$. Conversely, $\Pi^n(\cS_{1,1})$ is in the M\"{u}ger center for $n=0,1$ because the double braiding on $\Pi^m(\cS_{r,s})\boxtimes\Pi^n(\cS_{1,1})\cong\Pi^{m+n}(\cS_{r,s})$ is given by
\begin{equation*}
    e^{2\pi i(h_{r,s}^{\NS}+(m+n)/2-h^{\NS}_{r,s}-m/2-h^\NS_{1,1}-n/2)} =1
\end{equation*}
for all $r,s\in\ZZ_{\geq 1}$ and $m=0,1$.

        Finally, (4) follows from \eqref{eqn:induction-decomp} and \eqref{eqn:identify-induction} since $\cS_{r,s}\otimes\cF(1)$ has the same $L(c_a,0)\otimes L(c_b,0)$-module decomposition as $\Pi(\cS_{r,s})\otimes\cF(1)=\Pi(\cS_{r,s}\otimes\cF(1))$.
        
    \end{proof}

    For later use, we identify the tensor product of highest-weight vectors in $L(c_a, h_{2,n}(a))\otimes L(c_b, h_{2,n}(b))$ for $n=1,2$ with vectors in $\cS_{2,2} \otimes \mathcal{F}(1)$. First, from \eqref{branch}, 
\begin{align} \label{modex2}
\cS_{2,2} \otimes \mathcal{F}(1)\cong \bigoplus_{n=1}^{\infty} L(c_a, h_{2,n}(a))\otimes L(c_b, h_{2,n}(b))
\end{align}
as an $L(c_a, 0)\otimes L(c_b,0)$-module. Since 
\begin{align*}
h_{2,1}(a)+h_{2,1}(b)=h^{\NS}_{2,2}(t)+\frac{1}{2},
\end{align*}
the image of the tensor product of the highest-weight vectors $u_{2,1}^a\in L(c_a, h_{2,1}(a))$ and $u_{2,1}^b\in L(c_b, h_{2,1}(b))$ under the isomorphism \eqref{modex2}) is a linear combination of $G_{-\frac{1}{2}} v_{2,2} \otimes \vac_{\mathcal{F}(1)}$ and $v_{2,2}\otimes \psi_{-\frac{1}{2}}\vac_{\mathcal{F}(1)}$, where $\psi$ is the strong generator of $\cF(1)$. Specifically, we show in Appendix \ref{N1FERM} that we can take
\begin{align}\label{u21timesu21}
u_{2,1}^a\otimes u_{2,1}^b=\frac{2e^{\pi i/4}\sqrt{t}}{t-1} G_{-\frac{1}{2}}v_{2,2}\otimes\vac_{\cF(1)} +e^{-\pi i/4}\, v_{2,2}\otimes\psi_{-\frac{1}{2}}\vac_{\cF(1)} .  
\end{align}
On the other hand, we identify $u_{2,2}^a\otimes u_{2,2}^b = v_{2,2}\otimes \vac_{\cF(1)}$ since $h_{2,2}(a)+h_{2,2}(b)=h^{\NS}_{2,2}(t)$.

\subsection{Tensor structure at central charge \texorpdfstring{$\frac{3}{2}$}{3/2}}

Now we consider the module category of the ``degenerate minimal model'' vertex operator superalgebra $\cS(c^\NS(1),0)=\cS(\frac{3}{2},0)$ studied by Milas in \cite{M}. For central charge $\frac{3}{2}$, reducible Verma modules have lowest conformal weights
\begin{align*}
   H_{3/2} & =\lbrace h_{r,s}^\NS(1)\mid r,s\in\ZZ_{\geq 1},\,r-s\in2\ZZ\rbrace\\
   &=\left\lbrace \frac{(r-s)^2}{8}\mid r,s\in\ZZ_{\geq 1},\,r-s\in 2\ZZ\right\rbrace =\lbrace h_{2n+1,1}^\NS(1)\mid n\in\ZZ_{\geq 0}\rbrace.
\end{align*}
Thus the simple modules in $ \mathcal{O}_{3/2}^{\mathrm{fin}}$ are $\lbrace\cS_{2n+1,1}\mid n\in\ZZ_{\geq 0}\rbrace$. 
The embedding diagrams for the corresponding Verma modules are given in Proposition \ref{vermadiag}(1)(d), and they show that 
\begin{align*}
\cS_{2n+1,1}\cong M^\NS_{2n+1,1}/M^\NS_{2n+3,1}
\end{align*}
for $n\in\ZZ_{\geq 0}$. Although $\cO^{\mathrm{fin}}_{3/2}$ is not semisimple, we will show that it has a semisimple tensor subcategory that contains all simple objects.

To do so, we use the fact that $\cS(\frac{3}{2}, 0)$ is the $SO(3)$-fixed point subalgebra of the vertex operator superalgebra $\cF(3)$ of three free fermions. Note that the full automorphism group of $\cF(3)$ is $O(3)$, generated by $SO(3)$ together with the parity automorphism. The decomposition of $\cF(3)$ as an $SO(3) \otimes \cS(\frac{3}{2}, 0)$-module is easily computed:
\begin{lemma}\label{lem:F_decomposition}
    As an $SO(3) \otimes \cS(\frac{3}{2}, 0)$-module, 
    \[
    \cF(3) \cong \bigoplus_{n=0 }^\infty \rho_{2n+1} \otimes \Pi^n\left(\cS_{2n+1,1}\right),
    \]
    where $\rho_{2n+1}$ denotes the $(2n+1)$-dimensional simple $SO(3)$-module. 
\end{lemma}
\begin{proof}
As $\cF(3)$ is a continuous $SO(3)$-module with finite-dimensional conformal weight spaces,
\[
\cF(3) \cong \bigoplus_{n=0  }^\infty \rho_{2n+1} \otimes L_{2n}
\]
as an $SO(3)$-module, where by \cite[Theorem 1.1]{KR}, the multiplicity spaces $L_{2n}$ are simple $\cS(\frac{3}{2}, 0)$-modules.
For $m\in\frac{1}{2}\ZZ$ and $n\in\ZZ$, let $F_{m, n}$ be the subspace of $\cF(3)$ of conformal weight $m$ and $\mathfrak{so}(3)$-weight $n$. Since the three fermion strong generators of $\cF(3)$ have $\mathfrak{so}(3)$-weights $\pm 2$ and $0$, the graded dimension of $\cF(3)$ has the form
\begin{equation*}
    \mathrm{ch}[\cF(3)](z, q)=  \sum_{m\in\frac{1}{2}\ZZ,\,n \in 2\mathbb Z} \dim(F_{m, n}) z^n q^m
\end{equation*}
 Let 
 \begin{equation*}
 F_{2n}(q) = \sum\limits_{ m \in \frac{1}{2} \mathbb Z} \text{dim}(F_{m,2n})  q^m
 \end{equation*}
 be the coefficient of $z^{2n}$, so that $\text{ch}[L_{2n}]= F_{2n}(q) - F_{2n+2}(q)$.

Now, $\text{ch}[\cF(3)](z, q)$ has the nice product form
\[
\text{ch}[\cF(3)](z, q)= \prod_{m=1}^\infty (1 + z^2 q^{m-1/2} )(1 +  q^{m-1/2} )(1 + z^{-2} q^{m-1/2} ).
\]
Using the Jacobi triple product identity
\[
\sum_{n \in \mathbb Z} z^{2n} q^{n^2/2} = \prod_{m=1}^\infty (1 + z^2 q^{m-1/2} )(1 -  q^{m} )(1 + z^{-2} q^{m-1/2} ),
\]
this simplifies to
\[
\text{ch}[\cF(3)](z, q)=  \prod_{m=1}^\infty \frac{1+q^{m-1/2}}{1-q^m} \sum_{n \in \mathbb Z } z^{2n} q^{n^2/2},  
\]
and so
\begin{equation*} 
\begin{split}
\text{ch}[L_{2n}] &=(q^{n^2/2} -   q^{(n+1)^2/2}) \prod_{m=1}^\infty \frac{1+q^{m-1/2}}{1-q^m}\\
&= q^{n^2/2} (1  -   q^{n+1/2}) \prod_{m=1}^\infty \frac{1+q^{m-1/2}}{1-q^m}
= q^{n^2/2}\cdot \frac{1  -   q^{2n+1}}{1 + q^{n+1/2}} \prod_{m=1}^\infty \frac{1+q^{m-1/2}}{1-q^m}.
\end{split}
\end{equation*} 
Since $h_{2n+1,1}=\frac{n^2}{2}$, the lowest conformal weights of $L_{2n}$ and $\cS_{2n+1,1}$ coincide; moreover, the lowest conformal weight space of $L_{2n}$ has the same parity as $n$ since the even part of $\cF(3)$ is precisely its $\ZZ$-graded part. Thus because $L_{2n}$ is simple, we get $L_{2n}\cong \Pi^n\left(\cS_{2n+1,1}\right)$. 

\end{proof}
We record the character formula for $\cS_{2n+1,1}$ derived in the preceding proof as a corollary:
\begin{cor}
For $n\in\ZZ_{\geq 0}$, the $\cS(\frac{3}{2},0)$-module $\cS_{2n+1,1}$ has graded dimension
    \[ \mathrm{ch}\left[\cS_{2n+1,1}\right] = q^{n^2/2}\cdot \frac{1  -   q^{2n+1}}{1 + q^{n+1/2}} \prod_{m=1}^\infty \frac{1+q^{m-1/2}}{1-q^m}. \]
\end{cor}

We will now use the $SO(3)$-action on $\cF(3)$ and the decomposition of Lemma \ref{lem:F_decomposition} to show that the symmetric tensor category $\mathrm{Rep}\,SO(3)$ embeds into $\mathcal{O}^{\mathrm{fin}}_{3/2}$. To do so, we first recall the main result of \cite{Mc}. Let $V$ be a simple vertex operator superalgebra, $G$ a compact Lie group that acts continuously on $V$ by automorphisms and contains the parity automorphism of $V$, and $\mathcal{C}$ a braided tensor category of modules for the fixed-point vertex operator algebra $V^G$ that contains all simple modules which appear in the decomposition of $V$ as a $V^G$-module. Then \cite[Corollary 4.8]{Mc} says that there is a braided tensor embedding $\Phi: (\mathrm{Rep}\,G)^\Omega\rightarrow\cC$ where $(\mathrm{Rep}\,G)^\Omega$ agrees with the category $\mathrm{Rep}\,G$ of finite-dimensional continuous $G$-modules as a tensor category but has modified braiding isomorphisms (related to the fact that $V^G$ is a vertex operator algebra but $V$ is a superalgebra).

Our setting is slightly different from that of \cite{Mc}, since $SO(3)$ does not contain the parity automorphism of $\cF(3)$, and thus the fixed-point subalgebra $\cF(3)^{SO(3)}=\cS(\frac{3}{2},0)$ is a vertex operator superalgebra, not a vertex operator algebra. By slightly modifying the arguments in \cite{Mc} to this setting, we obtain a braided tensor embedding with no need to modify the braiding isomorphisms of $\mathrm{Rep}\,SO(3)$:
\begin{theorem}\label{thm:t=1_tensor_embedding}
    There is a fully faithful braided tensor functor $\Phi:\mathrm{Rep}\,SO(3)\rightarrow\mathcal{O}^{\mathrm{fin}}_{3/2}$ such that $\Phi(\rho_{2n+1})\cong \Pi^n\left(\cS_{2n+1,1}\right)$ for all $n\geq 0$.
\end{theorem}
\begin{proof}
As in \cite[Section 3.2]{Mc}, the functor $\Phi$ is defined on objects by $\Phi(M)=(M\otimes \cF(3))^{SO(3)}$ for $M$ a finite-dimensional simple continuous $SO(3)$-module. Here, we are taking $SO(3)$-fixed points with respect to the tensor product action on $M\otimes \cF(3)$, and the parity decomposition of $\Phi(M)$ is determined by the parity decomposition of $\cF(3)$ (with $M$ taken to be purely even). Then for a morphism $f: M_1\rightarrow M_2$ in $\mathrm{Rep}\,SO(3)$, we define $\Phi(f)=(f\otimes\id_{\cF(3)})\vert_{(M_1\otimes \cF(3))^{SO(3)}}$; note that the image of $\Phi(f)$ is indeed contained in $(M_2\otimes \cF(3))^{SO(3)}$, and that $\Phi(f)$ thus defined is an even $\cS(\frac{3}{2},0)$-module homomorphism. As in \cite[Proposition 3.5]{Mc}, it is easy to show that $\Phi(\rho_{2n+1})\cong\Pi^n\left(\cS_{2n+1,1}\right)$, and it then follows that $\Phi$ is fully faithful since it takes the distinct simple objects of the semisimple category $\mathrm{Rep}\,SO(3)$ to distinct simple objects of $\mathcal{O}^{\mathrm{fin}}_{3/2}$.

To show that $\Phi$ is a braided tensor functor, let $M_1$ and $M_2$ be objects of $\mathrm{Rep}\,SO(3)$. Then we can define an intertwining operator $\mathcal{Y}$ of type $\binom{M_1\otimes M_2\otimes \cF(3)}{M_1\otimes \cF(3)\,\,M_2\otimes \cF(3)}$ by
\begin{equation*}
\cY(m_1\otimes v_1,x)(m_2\otimes v_2) = m_1\otimes m_2\otimes Y(v_1,x)v_2
    \end{equation*}
    for $m_1\in M_1$, $m_2\in M_2$, and $v_1,v_2\in \cF(3)$. Since $SO(3)$ acts on $\cF(3)$ by vertex operator superalgebra automorphisms, $\cY$ restricts to an intertwining operator of type $\binom{\Phi(M_1\otimes M_2)}{\Phi(M_1)\,\Phi(M_2)}$, and then this intertwining operator induces an even $\cS(\frac{3}{2},0)$-module homomorphism
    \begin{equation*}
J_{M_1,M_2}: \Phi(M_1)\boxtimes\Phi(M_2)\rightarrow\Phi(M_1\otimes M_2)
    \end{equation*}
by the universal property of tensor products in $\mathcal{O}^{\mathrm{fin}}_{3/2}$. The homomorphisms $J_{M_1,M_2}$ determine a natural transformation, and as in \cite[Theorem 4.5]{Mc}, this natural transformation is compatible with the unit, associativity, and braiding isomorphisms in $\mathrm{Rep}\,SO(3)$ and $\mathcal{O}^{\mathrm{fin}}_{3/2}$ (for compatibility of $J_{M_1,M_2}$ with the right unit and braiding isomorphisms, some sign factors need to be added in the calculations in \cite{Mc}). Finally, each $J_{M_1,M_2}$ is an isomorphism exactly as in Proposition 4.3 and Theorem 4.7 of \cite{Mc}. Thus $J$ is a braided tensor functor.

\end{proof}

From the preceding theorem, we recover the fusion rules obtained in \cite[Theorem 10.2]{M}, which shows that the semisimple subcategory of $\mathcal{O}^{\mathrm{fin}}_{3/2}$ containing all simple modules is a tensor subcategory with the same Grothendieck ring as $\mathrm{Rep}\,\mathfrak{osp}(1|2)$. We can also show that $\mathcal{O}^{\mathrm{fin}}_{3/2}$ is rigid:
\begin{theorem}\label{thm:t=1-properties}
The tensor category $\mathcal{O}^{\mathrm{fin}}_{3/2}$ is rigid, and tensor products of simple objects in $\mathcal{O}^{\mathrm{fin}}_{3/2}$ are given by 
  \begin{equation*}
     \cS_{2n+1,1}\boxtimes \cS_{2n'+1,1}\cong\bigoplus_{n''=\vert n-n'\vert}^{n+n'} \Pi^{n+n'+n''}\left(\cS_{2n''+1,1}\right)
  \end{equation*}
 for $n,n'\in\ZZ_{\geq 0}$.
\end{theorem}
\begin{proof}
Since $\mathrm{Rep}\,SO(3)$ is rigid, Theorem \ref{thm:t=1_tensor_embedding} implies that $\Pi^n(\cS_{2n+1,1})$ is rigid for all $n\in\ZZ_{\geq 0}$. Then all simple objects of $\mathcal{O}^{\mathrm{fin}}_{3/2}$ are rigid since parity reversals of rigid modules are rigid. Finally, $\mathcal{O}^{\mathrm{fin}}_{3/2}$ is rigid by the straightforward superalgebra generalization of \cite[Theorem 4.4.1]{CMY-singlet}. Moreover, from Theorem \ref{thm:t=1_tensor_embedding} and the fusion rules in $\mathrm{Rep}\,SO(3)$, we get
    \begin{equation*}
\Pi^n\left(\cS_{2n+1,1}\right)\boxtimes \Pi^{n'}\left(\cS_{2n'+1,1}\right)\cong\bigoplus_{n''=\vert n-n'\vert}^{n+n'} \Pi^{n''}\left(\cS_{2n''+1,1}\right).
        \end{equation*}
The formula for $\cS_{2n+1,1}\boxtimes\cS_{2n'+1,1}$ then follows from Proposition \ref{prop:Pi-and-fusion}.

\end{proof}

\begin{remark}
    The $N=1$ super Virasoro algebra admits a coset realization, namely
    \[
    \cS(c^{\NS}(t), 0) \cong \text{Com}\left( V^{k+2}(\mathfrak{so}_3), V^k(\mathfrak{so}_3) \otimes \cF(3) \right)
    \]
    for generic $t$, where the level $k$ is related to $t$ via $k = -4\psi-2$ and 
    \[
    \frac{1}{\psi} + \frac{2}{t} = 2.
    \]
    This is a special case of the orthosymplectic triality \cite{CL}. In the large level limit the parameter $t$ goes to $1$ and the coset in the limit becomes the $SO(3)$-orbifold of $\cF(3)$ \cite{CL2}. If $k$ is an admissible level, that is, $k = -2 + \frac{u}{v}$ with $u,v$ co-prime positive integers and $u > 1$, then this coset realization still holds if one replaces the universal affine VOAs and the universal $N=1$ super Virasoro algebra by their simple quotients. One then obtains the well-known rational super Virasoro minimal models \cite{A1}. 
    Using this coset realization to study $\cS(c^{\NS}(t),0)$ for other rational values of $t$ might be possible if one understands the branching rules. This would be of particular interest for $t = \frac{1}{n}$  and $n$ an integer. 
\end{remark}

.

\section{Rigidity of \texorpdfstring{$\cS_{2,2}$}{S22}}\label{sec:rigidity}

In this section, we show that the simple $\cS(c(t),0)$-module $\cS_{2,2}$ in $\cO^{\mathrm{fin}}_{c(t)}$ is rigid and self-dual for all $t\in\CC^\times$ except for $t^{\pm1}\in\ZZ_{\leq -1}$ and $t=\frac{p}{q}$ for $p,q\in\ZZ_{\geq 2}$ such that $p-q\in 2\ZZ$ and $\gcd(\frac{p-q}{2},q)=1$ (recall that $\cS(c(t),0)$ is not simple in the latter case). We begin by using the Zhu algebra of $\cS(c(t),0)$ and its (bi)modules to determine the possible lowest conformal weight spaces in fusion tensor products of the form $\cS_{2,2}\boxtimes\cS_{r,s}$.

\subsection{Zhu bimodules and intertwining operators involving \texorpdfstring{$\cS_{2,2}$}{S22}}

The Zhu algebra of a vertex operator algebra $V$ was defined in \cite{Z} as a tool to study $V$-modules. In \cite[Section 1.2]{KW2}, the Zhu algebra was generalized to $\frac{1}{2}\ZZ$-graded vertex operator superalgebras $V$ such that $V^{\bar{i}}=\bigoplus_{n\in\frac{i}{2}+\ZZ} V_{(n)}$ for $i=0,1$. In particular, $A(V)$ is a unital associative algebra structure on the vector space $V/O(V)$, where $O(V)$ is a certain $\ZZ/2\ZZ$-graded subspace which contains $V^{\bar{1}}$. In fact, $A(V)$ is a quotient of the Zhu algebra $A(V^{\bar{0}})$ of the vertex operator algebra $V^{\bar{0}}$. For an element $v\in V$, we use $[v]$ to denote its image $v+O(V)\in A(V)$.

If $W=\bigoplus_{n\in\frac{1}{2}\NN} W(n)$ is a $\frac{1}{2}\NN$-gradable weak $V$-module, then $W(0)$ is an $A(V)$-module such that $[v]\cdot w=o(v)w$ for $v\in V$, $w\in W(0)$, where $o(v)=\mathrm{Res}_x\,x^{-1}Y_W(x^{L_0} v,x)$ is the conformal weight preserving coefficient of the vertex operator for $v$ \cite[Theorem 1.2]{KW2}. Also, it is shown in \cite[Section 1.3]{KW2} that there is an $A(V)$-bimodule $A(W) = W/O(W)$ where $O(W)$ is a certain $\ZZ/2\ZZ$-graded subspace of $W$. For $w\in W$, we use $[w]$ to denote its image $w+O(W)\in A(W)$.

Now suppose that for $i=1,2,3$, $W_i$ is a lower-bounded generalized $V$-module whose conformal weights are contained in $h_i+\frac{1}{2}\ZZ_{\geq0}$ for some $h_i\in\CC$ (for example, each $W_i$ could be indecomposable). Then as in Remark \ref{rem:half-N-grading}, each $W_i$ is $\frac{1}{2}\NN$-gradable, with $W_i(n)=(W_i)_{[h_i+n]}$ for $n\in\frac{1}{2}\ZZ_{\geq 0}$. Let $\cY$ be an intertwining operator of type $\binom{W_3}{W_1\,W_2}$. Then for $w_1\in W_1$, $w_2\in W_2$, substituting $x\mapsto 1$ in $\cY(w_1,x)w_2$ using the real-valued branch of logarithm, $\ln 1=0$, yields a well-defined element
\begin{equation*}
    \cY(w_1,1)w_2\in\prod_{n\in\frac{1}{2}\NN} W_3(n).
\end{equation*}
Let $\pi_0: \prod_{n\in\frac{1}{2}\NN} W_3(n)\rightarrow W_3(0)$ denote the projection. Then by \cite[Theorem 1.5(1)]{KW2}, 
\begin{align*}
    \pi(\cY): A(W_1)\otimes_{A(V)} W_2(0) & \longrightarrow W_3(0)\\
    [w_1]\otimes_{A(V)} w_2 & \longmapsto \pi_0(\cY(w_1,1)w_2)
\end{align*}
is a well-defined $A(V)$-module homomorphism. The following easy proposition is the superalgebra version of \cite[Proposition 2.5]{MY1}, which is based on \cite[Proposition 24]{TW}, and has essentially the same proof:
\begin{prop}  \label{surprop}
Let $V$ be a $\frac{1}{2}\ZZ$-graded vertex operator superalgebra such that $V^{\bar{i}}=\bigoplus_{n\in\frac{i}{2}+\ZZ} V_{(n)}$ for $i=0,1$, and for $i=1,2,3$, let $W_i$ be a lower-bounded generalized $V$-module with conformal weights contained in $h_i+\frac{1}{2}\ZZ_{\geq 0}$ for some $h_i\in\CC$. If $\cY$ is a surjective intertwining operator of type $\binom{W_3}{W_1\,W_2}$ and $W_2$ is generated by $W_2(0)$ as a $V$-module, then the $A(V)$-module homomorphism $\pi(\cY): A(W_1)\otimes_{A(V)} W_2(0) \rightarrow W_3(0)$ is surjective.
\end{prop}

We now take $V=\cS(c,0)$ for any $c\in\CC$. By \cite[Lemma 3.1]{KW2}, there is an algebra isomorphism $A(\cS(c,0))\xrightarrow{\sim}\CC[x]$ such that $[\omega]\mapsto x$. Using this identification, it is observed in \cite[Section 7.3]{M} that for $h\in\CC$,
\begin{align} \label{isomVerma}
A(M^\NS(c,h))\cong \CC[x,y]\cdot v_{\bar{0}} \oplus \CC[x,y]\cdot v_{\bar{1}}    
\end{align}
as a superspace, where $v_{\bar{0}}=[\vac_{c,h}]$ and $v_{\bar{1}}=[G_{-\frac{1}{2}}\vac_{c,h}]$. The left and right actions of $A(\cS(c,0))\cong\CC[x]$ on $A(M^\NS(c,h))$ are given by multiplication by $x$ and $y$, respectively:
    \begin{align*}
[\omega]\cdot(f(x,y)v_{\bar{0}}+g(x,y)v_{\bar{1}}) =x f(x,y)v_{\bar{0}}+x g(x,y)v_{\bar{1}}, \\
(f(x,y)v_{\bar{0}}+g(x,y)v_{\bar{1}})\cdot[\omega] =y f(x,y)v_{\bar{0}}+y g(x,y)v_{\bar{1}}.
\end{align*}
Thus $A(M^\NS(c,h))$ is generated by $v_{\bar{0}}=[\vac_{c,h}]$ and $v_{\bar{1}}=[G_{-\frac{1}{2}}\vac_{c,h}]$ as an $A(\cS(c,0))$-bimodule.

We now fix $c=c^{\NS}(t)$ for some $t\in\CC^\times$, and we fix $h=h^\NS_{2,2}(t)$. One can check that the non-trivial singular vector of lowest conformal weight in
 $M^\NS_{2,2} = M^\NS(c, h)$ is
\begin{equation}\label{singularvector}
w_{2,2} = \left(L_{-1}^2 - \frac{4}{3}h^\NS_{2,2}(t)L_{-2} - G_{-\frac{3}{2}}G_{-\frac{1}{2}}\right)v_{2,2}
\end{equation}
up to a non-zero scalar multiple, where $v_{2,2}=\vac_{c,h}$. We set $\mathcal{J}_{2,2}=\langle w_{2,2}\rangle$, the $\cS(c,0)$-submodule generated by $w_{2,2}$, and define $\mathcal{M}_{2,2}=M^\NS_{2,2}/\mathcal{J}_{2,2}$. Note that $\mathcal{M}_{2,2}$ is simple (and isomorphic to $\cS_{2,2}$) if and only if $t\neq\frac{p}{q}$ for $p, q \in\ZZ_{\geq 2}$ such that $p-q\in 2\ZZ$ and $\gcd(\frac{p-q}{2},q)=1$. 
By \cite[Proposition 1.2(1)]{KW2}, $A(\mathcal{M}_{2,2})\cong A(M_{2,2}^\NS)/A(\mathcal{J}_{2,2})$, 
where $A(\mathcal{J}_{2,2})$ denotes the image of $\mathcal{J}_{2,2}$ in $A(M_{2,2}^\NS)$. Since $\cJ_{2,2}$ is a Verma submodule, \eqref{isomVerma} shows that $A(\cJ_{2,2})$ is generated by $[w_{2,2}]$ and $[G_{-\frac{1}{2}}w_{2,2}]$ as an $A(\cS(c,0))$-bimodule, and therefore
\begin{align} \label{ZhuQuoM}
A(\mathcal{M}_{2,2}) \cong  \CC[x,y]/\langle f_{2, 2}(x,y)\rangle\cdot v_{\bar{0}}\oplus \CC[x,y]/\langle g_{2,2}(x,y)\rangle\cdot v_{\bar{1}}
\end{align}
for polynomials $f_{2,2}(x,y)$ and $g_{2,2}(x,y)$ such that $f_{2,2}(x,y)\cdot v_{\bar{0}}=[w_{2,2}]$ and $g_{2,2}(x,y)\cdot v_{\bar{1}}=[G_{-\frac{1}{2}}w_{2,2}]$. We need to find $f_{2,2}(x,y)$ and $g_{2,2}(x,y)$ explicitly.

 First, by the definitions in \cite[Section 1.3]{KW2},
\begin{align} 
&[\omega]\cdot [v]=[(L_0+2L_{-1}+L_{-2})v],\qquad [v]\cdot[\omega]=[(L_{-2}+L_{-1})v] \label{L2}
\end{align}
in $A(M_{2,2}^\NS)$ for any $v\in M_{2,2}^\NS$, which implies that 
\begin{align} \label{Lz}
[L_{-1} v]=[\omega]\cdot[v]-[v]\cdot[\omega]-{\rm wt}(v)[v].
\end{align}
The definitions also imply (see especially the proof sketch of \cite[Theorem 1.4]{KW2}) that $(G_{-n-\frac{3}{2}}+G_{-n-\frac{1}{2}})v\in O(M_{2,2}^\NS)$ for $n\in\ZZ_{\geq 0}$, so that in particular for any $v\in M_{2,2}^\NS$,

\begin{align} \label{Gz}
[G_{-\frac{5}{2}}v]=-[G_{-\frac{3}{2}}v] = [G_{-\frac{1}{2}} v].
\end{align}
Using \eqref{singularvector}, \eqref{L2}, \eqref{Lz}, and \eqref{Gz}, we can calculate $f_{2,2}(x,y)$ to be
\begin{equation*}
f_{2,2}(x,y) = (x-y)^2 - \frac{2}{3}h\left(x+y+\frac{1}{2}h\right).
\end{equation*}
To calculate $g_{2,2}(x,y)$, we use
\begin{equation*}
    G_{-\frac{1}{2}}w_{2,2} = \left(L_{-1}^2 G_{-\frac{1}{2}} + L_{-1}G_{-\frac{3}{2}} - 2\left(\frac{2}{3}h+1\right) L_{-2}G_{-\frac{1}{2}} -\left(\frac{2}{3}h+1\right)G_{-\frac{5}{2}}\right)v_{2,2}
\end{equation*}
and  then \eqref{L2}, \eqref{Lz}, and \eqref{Gz} again to obtain

\begin{equation*}
    g_{2,2}(x,y) = (x-y)^2-\left(\frac{2}{3}h+1\right)\left(x+y+\frac{1}{2}h-\frac{1}{4}\right).
\end{equation*}

We want to use $f_{2,2}(x,y)$ and $g_{2,2}(x,y)$ to obtain information about surjective intertwining operators of type $\binom{W}{\mathcal{M}_{2,2}\,\cS_{r,s}}$ for $W$ an indecomposable module in $\cO_c^{\mathrm{fin}}$ and $r,s\in\ZZ_{\geq 1}$ such that $r-s\in 2\ZZ$. By Proposition \ref{surprop}, the $A(\cS(c,0))\cong\CC[x]$-module $W(0)$ is a homomorphic image of
\begin{equation}\label{eqn:Zhu-bimod-times-mod}
    A(\mathcal{M}_{2,2})\otimes_{A(\cS(c,0))} \cS_{r,s}(0)\cong \CC[x]/\langle f_{2,2}(x,h_{r,s}^\NS)\rangle \oplus \CC[x]/\langle g_{2,2}(x, h_{r,s}^\NS)\rangle,
\end{equation}
where one can check that
\begin{align}
f_{2,2}(x,h_{r,s}^\NS) & = (x-h_{r+1,s+1}^\NS)(x-h_{r-1,s-1}^\NS),\label{f22}\\
g_{2,2}(x,h_{r,s}^\NS)
& = (x-h_{r+1,s-1}^\NS)(x-h_{r-1,s+1}^\NS).\label{g22}
\end{align}
Thus $\dim W(0)^{\bar{i}}\leq 2$ for $i=0,1$, and the (generalized) eigenvalue(s) of $L_0$ on $W(0)^{\bar{0}}$ are $h_{r+1,s+1}^\NS$ and/or $h_{r-1,s-1}^\NS$, while the (generalized) eigenvalue(s) of $L_0$ on $W(0)^{\bar{1}}$ are $h_{r+1,s-1}^\NS$ and/or $h_{r-1,s+1}^\NS$. Using these results, we prove:

\begin{prop}\label{prop:upperbound}
Let $W$ be a module in $\ocfin$ and let $r,s\in \mathbb{Z}_{\geq 1}$ such that $r-s\in 2\ZZ$. If there exists a surjective intertwining operator of type $\binom{W}{\cS_{2,2} \, \cS_{r,s}}$, then the conformal weights of $W^{\bar{0}}$ are contained in 
\begin{equation*}
    \lbrace h^\NS_{r+1,s+1}+  \ZZ_{\geq 0}\rbrace \cup \lbrace h^\NS_{r-1, s-1}+ \ZZ_{\geq 0}\rbrace \cup \left\lbrace h^\NS_{r+1, s-1}+\frac{1}{2}+\ZZ_{\geq 0}\right\rbrace\cup  \left\lbrace h^\NS_{r-1, s+1}+\frac{1}{2}+\ZZ_{\geq 0}\right\rbrace,
\end{equation*}
and the conformal weights of $W^{\bar{1}}$ are contained in 
\begin{equation*}
    \left\lbrace h^\NS_{r+1,s+1}+\frac{1}{2}+  \ZZ_{\geq 0}\right\rbrace \cup \left\lbrace h^\NS_{r-1, s-1}+\frac{1}{2} +\ZZ_{\geq 0}\right\rbrace \cup \lbrace h^\NS_{r+1, s-1}+\ZZ_{\geq 0}\rbrace\cup  \lbrace h^\NS_{r-1, s+1}+\ZZ_{\geq 0}\rbrace.
\end{equation*}
This conclusion holds in particular for $W=\cS_{2,2}\boxtimes \cS_{r,s}$.
\end{prop}
\begin{proof}
    By Proposition \ref{prop:mainprop}(1), $W$ has a filtration
    \begin{equation*}
        0 = W_0 \subset W_1 \subset\cdots \subset W_{n-1}\subset W_n =W
    \end{equation*}
    such that $W_i/W_{i-1}$ is a highest-weight $\cS(c,0)$-module for $i=1,\ldots, n$. Moreover, from the proof of Proposition \ref{prop:mainprop}(1), the highest-weight vector generating $W_i/W_{i-1}$ may be taken from $(W/W_{i-1})(0)$, where the $\frac{1}{2}\NN$-grading on $W/W_{i-1}$ is chosen as in Remark \ref{rem:half-N-grading}. Now, any surjective intertwining operator $\cY$ of type $\binom{W}{\cS_{2,2}\,\cS_{r,s}}$ induces another surjective intertwining operator
    \begin{equation*}
        \cM_{2,2}\otimes\cS_{r,s}\twoheadrightarrow \cS_{2,2}\otimes\cS_{r,s} \xrightarrow{\cY} W[\log x]\lbrace x\rbrace \twoheadrightarrow (W/W_{i-1})[\log x]\lbrace x\rbrace.
    \end{equation*}
    So by Proposition \ref{surprop}, \eqref{eqn:Zhu-bimod-times-mod}, \eqref{f22}, and \eqref{g22}, each $W_i/W_{i-1}$ is a quotient of one of $M_{r+1,s+1}^\NS$, $M_{r-1,s-1}^\NS$, $\Pi(M_{r+1,s-1}^\NS)$, or $\Pi(M_{r-1,s+1}^\NS)$. Thus the even and odd parts of each $W_i/W_{i-1}$ have conformal weights contained in the indicated sets, and then so do the even and odd parts of $W$.
    
\end{proof}

\begin{remark}
    Note that Proposition \ref{prop:upperbound} is consistent with the $(r',s')=(2,2)$ case of \eqref{eqn:irr-fus-rules} when $t\notin\QQ$.
\end{remark}

\subsection{Evaluation and coevaluation candidates}

We continue to fix a central charge $c=c^\NS(t)$ and conformal weight $h=h^\NS_{2,2}(t)$ for some $t\in\CC^\times$, as well as a highest-weight vector $v_{2,2}\in\cS_{2,2}$. We also fix an even non-degenerate invariant bilinear form $\langle\cdot,\cdot\rangle$ on $\cS_{2,2}$ such that $\langle v_{2,2},v_{2,2}\rangle=1$. By Lemma \ref{lem:bilinear-form-symmetric}, $\langle\cdot,\cdot\rangle$ is symmetric.

To show that the $\cS(c,0)$-module $\cS_{2,2}$ is rigid and self-dual, we first need evaluation and coevaluation candidates
\begin{equation*}
    \mathrm{ev}: \cS_{2,2}\boxtimes\cS_{2,2}\longrightarrow\cS_{1,1},\qquad \mathrm{coev}: \cS_{1,1}\longrightarrow\cS_{2,2}\boxtimes\cS_{2,2},
\end{equation*}
and then we need to show that the rigidity composition
\begin{align}\label{eqn:rig-comp}
    \mathfrak{R}: \cS_{2,2}\xrightarrow{r^{-1}} & \cS_{2,2}\boxtimes\cS_{1,1}\xrightarrow{\id\boxtimes\mathrm{coev}} \cS_{2,2}\boxtimes(\cS_{2,2}\boxtimes\cS_{2,2})\nonumber\\
    &\,\xrightarrow{\cA}(\cS_{2,2}\boxtimes\cS_{2,2})\boxtimes\cS_{2,2}\xrightarrow{\mathrm{ev}\boxtimes\id} \cS_{1,1}\boxtimes\cS_{2,2}\xrightarrow{l}\cS_{2,2}
\end{align}
is non-zero. Since $\cS_{2,2}$ is simple, we will then be able to rescale either $\mathrm{ev}$ or $\mathrm{coev}$ so that $\mathfrak{R}$ becomes the identity, and it will then follow (from \cite[Corollary 4.2.2]{CMY3}, for example), that $\cS_{2,2}$ is rigid. 

To define the evaluation candidate, we need an intertwining operator $\cE$ of type 
$\binom{\cS_{1,1}}{\cS_{2,2}\, \cS_{2,2}}$. We start with the vertex operator map $Y_{\cS_{2,2}}$ and then apply the superalgebra versions of the skew-symmetry operator $\Omega_0$ from \cite[Equation 3.77]{HLZ2} and the contragredient operator $A_0$ from \cite[Equation~3.87]{HLZ2} to obtain an intertwining operator $A_0(\Omega_0(Y_{\cS_{2,2}}))$ of type $\binom{\cS_{1,1}'}{\cS_{2,2}\; \cS_{2,2}}$. If $t\neq\frac{p}{q}$ for $p,q\in\ZZ_{\geq 2}$ such that $\gcd(\frac{p-q}{2},q)=1$, then the vertex operator superalgebra $\cS(c,0) = \cS_{1,1}$ is simple, so $\cS_{1,1}'\cong\cS_{1,1}$. We fix the even non-degenerate invariant (and symmetric by Lemma \ref{lem:bilinear-form-symmetric}) bilinear form $\langle\cdot, \cdot \rangle$ on $\cS_{1,1}$ such that $\langle\one, \one\rangle = 1$. This allows us to define an intertwining operator $\cE=e^{-\pi i h}\cdot A_0(\Omega_0(Y_{\cS_{2,2}}))$ of type $\binom{\cS_{1,1}}{\cS_{2,2}\; \cS_{2,2}}$, which from the definitions satisfies 
\begin{align}\label{eqn:E}
  \langle v, \cE(w', x)w\rangle & = e^{-\pi i h}\langle A_0(\Omega_0(Y_{\cS_{2,2}}))(w',x)w, v\rangle\nonumber\\
  & =(-1)^{\vert w\vert\vert w'\vert}\langle w, \Omega_0(Y_{\cS_{2,2}})(e^{xL_1}e^{\pi i(L_0-h)} x^{-2L_0}w',x^{-1})v\rangle\nonumber\\
& = (-1)^{\vert w'\vert(\vert v\vert+\vert w\vert)}\langle e^{x^{-1} L_1} w, Y_{\cS_{2,2}}(v, -x^{-1})e^{xL_1}x^{-2 L_0}e^{\pi i (L_0-h)}w'\rangle \nonumber\\
& = \langle e^{x^{-1} L_1} w, Y_{\cS_{2,2}}(v, -x^{-1})e^{xL_1}x^{-2 L_0}e^{-\pi i (L_0-h)}w'\rangle
\end{align}
for $v \in \cS_{1,1}$, and $w, w' \in \cS_{2,2}$. The last step in the calculation uses the evenness of $\langle\cdot,\cdot\rangle$ together with $(-1)^{\vert w'\vert} w'= e^{-2\pi i(L_0-h)}w'$. In particular,
\begin{equation*}
    \langle\vac,\cE(v_{2,2},x)v_{2,2}\rangle =\langle v_{2,2},v_{2,2}\rangle x^{-2h} =x^{-2h}=\langle\vac,\vac\rangle x^{-2h},
\end{equation*}
and thus
\begin{equation}\label{eqn:E-v22-v22}
    \cE(v_{2,2},x)v_{2,2}\in x^{-2h}(\vac+x\cS_{1,1}[[x]]).
\end{equation}
We now define $\mathrm{ev}:\cS_{2,2}\boxtimes\cS_{2,2}\rightarrow\cS_{1,1}$ to be the unique $\cS(c,0)$-module homomorphism such that $\mathrm{ev}\circ\cY_\boxtimes=\cE$, where for brevity we now use $\cY_\boxtimes$ to denote all tensor product intertwining operators.

For the coevaluation candidate, we first fix some notation. If $\cY$ is an intertwining operator of type $\binom{W_3}{W_1\,W_2}$, where $W_1$, $W_2$, $W_3$ are modules in $\cO_c^{\mathrm{fin}}$, then substituting $x\mapsto 1$ in $\cY$ using the real-valued branch of logarithm $\ln 1 = 0$ yields an element
\begin{equation*}
    \cY(w_1,1)w_2\in \overline{W}_3=\prod_{h\in\CC}\prod_{i=0,1} (W_3)_{[h]}\cap W_3^{\bar{i}}
\end{equation*}
for any $w_1\in W_1$, $w_2\in W_2$. Then for $h\in\CC$ and $i=0,1$, we fix $\pi_h^{\bar{i}}: \overline{W}_3\rightarrow (W_3)_{[h]}\cap W_3^{\bar{i}}$ to be the projection. In the next lemma we take $W_3=\cS_{2,2}\boxtimes\cS_{2,2}$:
\begin{lemma}\label{lem:projections_zero}
If $t^{\pm 1}\notin\ZZ_{\leq -1}\cup\left\lbrace\frac{-(n-2)\pm \sqrt{n^2-4n}}{2}\mid n\in\ZZ_{\geq 1}\right\rbrace$, then
\begin{equation}\label{lowerboundcondition}
    \pi_1^{\bar{0}}(L_{-2-i}v)= \pi_0^{\bar{0}}(L_{-1-i}v) = \pi_1^{\bar{0}}(G_{-\frac{3}{2}-i}v) = 0
\end{equation}
for all $i\in\ZZ_{\geq 0}$ and all $v \in \cS_{2,2}\btimes \cS_{2,2}$.   
\end{lemma}
\begin{proof}
If one of the expressions in \eqref{lowerboundcondition} is non-zero, then there exists a non-zero $v$ in either $(\cS_{2,2}\boxtimes\cS_{2,2})_{[-i-1]}^{\bar{0}}$ or $(\cS_{2,2}\boxtimes\cS_{2,2})_{[-i-\frac{1}{2}]}^{\bar{1}}$ for some $i\in\ZZ_{\geq 0}$. By Proposition \ref{prop:upperbound}, such non-zero $v$ can exist only if $h_{3,1}=-n+\frac{1}{2}$, $h_{1,3}=-n+\frac{1}{2}$, or $h_{3,3}=-n$ for some $n\in\ZZ_{\geq 1}$.
Since
    \[
h^\NS_{3,1}= t-\frac{1}{2},\qquad h^\NS_{1,3}= t^{-1}-\frac{1}{2},\qquad  h^\NS_{3,3}= t-2+t^{-1},
    \]
     $h^\NS_{1,3}$ or $h^\NS_{3,1}$ is a negative strict half-integer only when $t^{\pm1}\in\ZZ_{\leq -1}$, and $h^\NS_{3,3}$ is a negative integer only when $t =\frac{-(n-2)\pm \sqrt{n^2-4n}}{2}$ for some $n \in \ZZ_{\geq 1}$. Thus if $t^{\pm 1}$ do not fall in these two cases, then all expressions in \eqref{lowerboundcondition} vanish.
 
\end{proof}

\begin{prop}\label{prop:coev}
If $t^{\pm 1}\notin\ZZ_{\leq -1}\cup\left\lbrace\frac{-(n-2)\pm \sqrt{n^2-4n}}{2}\mid n\in\ZZ_{\geq 1}\right\rbrace$, then there is an $\cS(c^\NS(t),0)$-module homomorphism $\mathrm{coev}: \cS(c^\NS(t),0) \rightarrow \cS_{2,2}\btimes \cS_{2,2}$ such that
    \begin{equation}\label{coev22}
    \mathrm{coev}(\one) = \pi_0^{\bar{0}}\left(\cY_\boxtimes(L_{-1}v_{2,2},1)v_{2,2}-\frac{2}{3}h_{2,2}^\NS(t) \cY_\btimes(v_{2,2},1)v_{2,2}\right).
    \end{equation}
    In particular, $\mathrm{coev}$ exists when $t^{\pm 1}\in\QQ\setminus\ZZ_{\leq 0}$.
\end{prop}
\begin{proof}
We first show that $L_{-1}\mathrm{coev}(\vac)=0$ using the singular vector \eqref{singularvector} in $M_{2,2}^\NS$, the commutator and associator formulas \eqref{eqn:comm1}, \eqref{eqn:comm2}, \eqref{eqn:asso1}, and \eqref{eqn:asso2}, as well as \eqref{lowerboundcondition}:
\begin{align*}
&   L_{-1}  \pi_0^{\bar{0}}\cY_\btimes(L_{-1}v_{2,2},1)v_{2,2}\\
   & = \pi_1^{\bar{0}}\left(\cY_\btimes(L_{-1}^2v_{2,2},1) v_{2,2}+\cY_\btimes(L_{-1}v_{2,2},1) L_{-1}v_{2,2}\right)\\
   & = \pi_1^{\bar{0}}\left(\frac{4}{3}h_{2,2}^\NS(t)\cY_\btimes(L_{-2}v_{2,2},1)v_{2,2} + \cY_\btimes(G_{-\frac{3}{2}}G_{-\frac{1}{2}}v_{2,2},1) v_{2,2}\right)\\
   & \quad + L_{-1}\pi_0^{\bar{0}}\cY_\btimes(v_{2,2},1) L_{-1}v_{2,2} - \pi_1^{\bar{0}}\cY_\btimes(v_{2,2},1) L_{-1}^2v_{2,2}\\
   & = \frac{4}{3}h_{2.2}^\NS(t)\pi_1^{\bar{0}}\left(\cY_\btimes(v_{2,2},1)L_{-1}v_{2,2} + h_{2,2}^\NS(t)\cY_\btimes(v_{2,2},1)v_{2,2}\right)-\pi_1^{\bar{0}}\cY_\btimes(G_{-\frac{1}{2}}v_{2,2},1)G_{-\frac{1}{2}}v_{2,2}\\
   &\quad + L_{-1}\pi_0^{\bar{0}}\cY_\btimes(v_{2,2},1)L_{-1}v_{2,2} -\pi_1^{\bar{0}}\left( \frac{4}{3}h_{2,2}^\NS(t)\cY_\btimes(v_{2,2},1)L_{-2}v_{2,2}+\cY_\boxtimes(v_{2,2},1)G_{-\frac{3}{2}}G_{-\frac{1}{2}}v_{2,2}\right)\\
   & = \pi_1^{\bar{0}}\left(\frac{4}{3}h_{2,2}^\NS(t)\cY_\btimes(v_{2,2},1)L_{-1}v_{2,2} + \frac{4}{3}h_{2,2}^\NS(t)^2\cY_\btimes(v_{2,2},1)v_{2,2}-\cY_\btimes(G_{-\frac{1}{2}}v_{2,2},1)G_{-\frac{1}{2}}v_{2,2}\right)\\
   &\quad + L_{-1}\pi_0^{\bar{0}}\cY_\btimes(v_{2,2},1)L_{-1}v_{2,2} + \frac{4}{3}h^\NS_{2,2}(t)\pi_1^{\bar{0}}\left(\cY_\boxtimes(L_{-1}v_{2,2},1)v_{2,2}-h^\NS_{2,2}(t)\cY_\btimes(v_{2,2},1)v_{2,2}\right)\\
   &\quad +\pi_1^{\bar{0}}\cY_\btimes(G_{-\frac{1}{2}}v_{2,2},1)G_{-\frac{1}{2}}v_{2,2}\\
   &=\frac{4}{3}h_{2,2}^\NS(t) L_{-1}\pi_0^{\bar{0}}\cY_\btimes(v_{2,2},1)v_{2,2}-L_{-1}\pi_0^{\bar{0}}\cY_\btimes(L_{-1}v_{2,2},1)v_{2,2},
\end{align*}
which implies
\begin{equation*}
    L_{-1}\pi_0^{\bar{0}}\left(\cY_\boxtimes(L_{-1}v_{2,2},1)v_{2,2}-\frac{2}{3}h_{2,2}^\NS(t) \cY_\btimes(v_{2,2},1)v_{2,2}\right) = 0,
\end{equation*}
as required.

We now show that $\mathrm{coev}(\vac)$ is (if non-zero) a highest-weight vector. Indeed, 
\[
L_n\,\mathrm{coev}(\vac)=G_{n-\frac{1}{2}}\mathrm{coev}(\vac)=0
\]
for all $n\in\ZZ_{\geq 1}$ because the conditions on $t$ imply that $(\cS_{2,2}\boxtimes\cS_{2,2})^{\bar{0}}$ has no negative integer conformal weights and $(\cS_{2,2}\boxtimes\cS_{2,2})^{\bar{1}}$ has no negative strict half-integer conformal weights, as in the proof of Lemma \ref{lem:projections_zero}. Moreover,
\begin{equation*}
    L_0\,\mathrm{coev}(\vac)=\frac{1}{2}[L_1,L_{-1}]\,\mathrm{coev}(\vac)=0,
\end{equation*}
so $\mathrm{coev}(\vac)$ is (if non-zero) an $L_0$-eigenvector. Thus by the universal property of Verma modules, there is a unique homomorphism $M_{1,1}^\NS\rightarrow\cS_{2,2}\boxtimes\cS_{2,2}$ sending $\vac$ to $\mathrm{coev}(\vac)$. 
Then because 
\begin{equation*}
    G_{-\frac{1}{2}}\mathrm{coev}(\vac) = [G_{\frac{1}{2}},L_{-1}]\,\mathrm{coev}(\vac)=0,
\end{equation*}
this homomorphism descends to the quotient $\cS(c^\NS(t),0)=M_{1,1}^\NS/\langle G_{-\frac{1}{2}}\vac\rangle$.

For the final statement, we need to show that $\left\lbrace \frac{-(n-2)\pm\sqrt{n^2-4n}}{2}\mid n\in\ZZ_{\geq 1}\right\rbrace\cap\QQ \subset\ZZ_{\leq -1}$. Indeed, $\frac{-(n-2)\pm\sqrt{n^2-4n}}{2}$ is rational if and only if $\sqrt{n^2-4n}\in\ZZ$, and in this case, $\sqrt{n^2-4n}=\sqrt{(n-2)^2-4}$ is smaller in absolute value then $n-2$ and also has the same parity as $n-2$. Thus $\frac{-(n-2)+\sqrt{n^2-4n}}{2}$ and $\frac{-(n-2)-\sqrt{n^2-4n}}{2}$ are negative integers in this case.

\end{proof}

We can now return to the rigidity composition $\mathfrak{R}$ in \eqref{eqn:rig-comp}. We fix $c=c^\NS(t)$ and $h=h^\NS_{2,2}(t)$ for $t$ such that $\mathrm{ev}$ and $\mathrm{coev}$ are both defined, that is, for
\begin{align*}
    t^{\pm 1}\notin\ZZ_{\leq -1}& \cup\left\lbrace\frac{-(n-2)\pm\sqrt{n^2-4n}}{2}\mid n\in\ZZ_{\geq 1}\right\rbrace\nonumber\\
    &\cup\left\lbrace\frac{p}{q}\mid p,q\in\ZZ_{\geq 2},\, p-q\in 2\ZZ,\,\gcd\left(\frac{p-q}{2},q\right)=1\right\rbrace.
\end{align*}
It is enough to show that $\langle v,\mathfrak{R}(v)\rangle
\neq 0$ for some $v\in\cS_{2,2}$. The definition \eqref{eqn:right_unit} of the right unit isomorphism $r:\cS_{2,2}\btimes\cS_{1,1}\rightarrow\cS_{2,2}$ implies that
\begin{equation*}
    \overline{r}(\cY_\btimes(v,1)\vac)= e^{L_{-1}} Y_{\cS_{2,2}}(\vac,-1)v =e^{L_{-1}}v,
\end{equation*}
where $\overline{r}$ is the extension of $r$ to algebraic completions. Thus if $v$ is homogeneous, then
\begin{equation*}
    r^{-1}(v) = (\pi_{\mathrm{wt}\,v}^{\vert v\vert}\cY_\btimes(v,1)\vac),
\end{equation*}
and then \eqref{coev22} implies that
\begin{equation*}
    (\id\btimes\mathrm{coev})\circ r^{-1}(v) =\pi_{\mathrm{wt}\,v}^{\vert v\vert} \cY_\btimes(v,1)\pi_0^{\bar{0}}\left(\cY_\btimes(L_{-1}v_{2,2},1)v_{2,2}-\frac{2}{3}h\cY_\btimes(v_{2,2},1)v_{2,2}\right).
\end{equation*}
This is equivalently the coefficient of $x^{-2h}(\log x)^0$ in
\begin{align*}
    \pi_{\mathrm{wt}\,v}^{\vert v\vert} \cY_\btimes(v,1) & \left(x\cY_\btimes(L_{-1}v_{2,2},x)v_{2,2}-\frac{2}{3}h\cY_\btimes(v_{2,2},x)v_{2,2}\right)\nonumber\\
    &= \left(x\frac{d}{dx}-\frac{2}{3}h\right)\pi_{\mathrm{wt}\,v}^{\vert v\vert}\cY_\btimes(v,1)\cY_\btimes(v_{2,2},x)v_{2,2}.
\end{align*}
Thus from the evenness and invariance of the bilinar form $\langle\cdot,\cdot\rangle$ on $\cS_{2,2}$, the definition \eqref{eqn:assoc} of the associativity isomorphisms in $\cO_c^{\mathrm{fin}}$, the definition of $\mathrm{ev}$, and the definition of the left unit isomorphism \eqref{eqn:left_unit}, we find that $\langle v,\mathfrak{R}(v)\rangle$ is the coefficient of $x^{-2h}(\log x)^0$ in
\begin{align}\label{eqn:rig-comp-series}
    \left(x\frac{d}{dx}-\frac{2}{3}h\right) & \left\langle v, \pi_{\mathrm{wt}\,v}^{\vert v\vert}(l\circ(\mathrm{ev}\btimes\id)\circ\cA\circ\cY_\btimes)(v,1)\cY_\btimes(v_{2,2},x)v_{2,2}\right\rangle\nonumber\\
    & = \left(x\frac{d}{dx}-\frac{2}{3}h\right)\left\langle v, (l\circ(\mathrm{ev}\btimes\id)\circ\cY_\btimes)(\cY_\btimes(v,1-x)v_{2,2},x)v_{2,2}\right\rangle\nonumber\\
    & =\left(x\frac{d}{dx}-\frac{2}{3}h\right)\left\langle v, Y_{\cS_{2,2}}(\cE(v,1-x)v_{2,2},x)v_{2,2}\right\rangle\nonumber\\
    & = \left(x\frac{d}{dx}-\frac{2}{3}h\right) x^{-2h}\left\langle v, Y_{\cS_{2,2}}\left(\cE\left(v,\frac{1-x}{x}\right)v_{2,2},1\right)v_{2,2}\right\rangle.
\end{align}
The right hand side should be viewed as a series in powers of $\frac{1-x}{x}$ which converges absolutely for real numbers $x$ such that $x > 1-x>0$. Thus we need to re-expand \eqref{eqn:rig-comp-series} as a series in powers of $x$ and $\log x$ in the region $1>x>0$ and extract the coefficient of $x^{-2h}(\log x)^0$.

\subsection{Analysis of matrix coefficients}

It seems difficult to explicitly compute the right side of \eqref{eqn:rig-comp-series} directly, since computations using the singular vector \eqref{singularvector} indicate that $\langle v, Y_{\cS_{2,2}}(\cE(v,1-x)v_{2,2},x)v_{2,2}\rangle$ satisfies at best a complicated fourth-order differential equation, even in the simplest cases $v=v_{2,2}$ and $G_{-\frac{1}{2}}v_{2,2}$. Thus instead, we will exploit the decomposition \eqref{modex2} of $\cS_{2,2}\otimes\cF(1)$ as an $L(c_a,0)\otimes L(c_b,0)$-module and use explicit formulas for Virasoro correlation functions to compute \eqref{eqn:rig-comp-series} in the case of $t\notin\QQ$. Then we will use analytic continuation of matrix coefficients to compute \eqref{eqn:rig-comp-series} for suitable $t\in\QQ$.

First we need to check that the coefficients of powers of $\frac{1-x}{x}$ on the right side of \eqref{eqn:rig-comp-series} are analytic functions of $t$:
\begin{prop}\label{analyticcoefficents}
Assume $t\neq\frac{p}{q}$ for $p,q\in\ZZ_{\geq 2}$ such that $p-q\in 2\ZZ$ and $\gcd(\frac{p-q}{2},q)=1$, and for $i=0,1,2,3$, fix $w_i\in\cS_{2,2}$ of the form $L_{-i_1-1}\cdots L_{-i_k-1}G_{-j_1-\frac{1}{2}}\cdots G_{-j_l-\frac{1}{2}} v_{2,2}$ where $i_1,\ldots,i_k,j_1,\ldots, j_l\in\ZZ_{\geq 0}$. Then as a series in powers of $\frac{1-x}{x}$,
    \begin{equation}\label{eqn:analytic}
        \langle w_0,Y_{\cS_{2,2}}(\cE(w_1,1-x)w_2,x)w_3\rangle = \sum_{n\in\ZZ} q_n(t)\left(\frac{1-x}{x}\right)^{-2h_{2,2}^\NS(t)+n}
    \end{equation}
    for certain rational functions $q_n(t) \in \CC(t)$ depending on $w_0,w_1,w_2,w_3$. 
\end{prop}
\begin{proof}
    For brevity, write $z=\frac{1-x}{x}$. Then similar to the last step of the calculation \eqref{eqn:rig-comp-series},
    \begin{equation*}
        \langle w_0,Y_{\cS_{2,2}}(\cE(w_1,1-x)w_2,x)w_3\rangle =(1+z)^{\mathrm{wt}\,w_1+\mathrm{wt}\,w_2+\mathrm{wt}\,w_3-\mathrm{wt}\,w_0}\langle w_0, Y_{\cS_{2,2}}(\cE(w_1,z)w_2,1)w_3\rangle.
    \end{equation*}
    The first factor on the right side has the form
    \begin{equation*}
        (1+z)^{2h+N} =\sum_{i=0}^\infty \binom{2h+N}{i} z^i
    \end{equation*}
    for some fixed $N\in\frac{1}{2}\ZZ_{\geq 0}$ depending on $w_0,w_1,w_2,w_3$. The coefficients $\binom{2h+N}{i}$ are polynomials in $h=h_{2,2}^\NS(t)=\frac{3}{8}(t-2+t^{-1})$ and thus rational functions in $t$. 
    
    For the second factor, let $\lbrace v_i\rbrace_{i\in I}$ be the basis of $\cS(c,0)$ consisting of PBW monomials $L_{-i_1-2}\cdots L_{-i_k-2}G_{-j_1-\frac{3}{2}}\cdots G_{-j_l-\frac{3}{2}}\vac$, where $i_1\geq \cdots\geq i_k\geq 0$ and $j_1>\cdots >j_l\geq 0$. Our assumption on $t$ guarantees that $\cS(c,0)$ has a non-degenerate invariant bilinear form $\langle\cdot,\cdot\rangle$ such that $\langle\vac,\vac\rangle=1$. Note from \eqref{eqn:contra} and the PBW theorem that $\langle v_i,v_j\rangle$ is a polynomial in $c$ for each $i,j\in I$. Thus each element of the dual basis $\lbrace v_i'\rbrace_{i\in I}$ of $\cS(c,0)$ with respect to $\langle\cdot,\cdot\rangle$ has the form $v_i'=\sum_j f_{i,j}(c)v_j$ for certain rational functions $f_{i,j}(c)$. Note that $f_{i,j}$ is also a rational function of $t$ since $c=c^\NS(t)=\frac{15}{2}-3(t+t^{-1})$.
    
    Now using \eqref{eqn:E}, we have
    \begin{align*}
        \langle w_0, Y_{\cS_{2,2}}( & \cE(w_1,z)w_2,1)w_3\rangle  =\sum_{i\in I} \langle w_0,Y_{\cS_{2,2}}(v_i,1)w_3\rangle\langle v_i',\cE(w_1,z)w_2\rangle\nonumber\\
        & =\sum_{i\in I} \langle w_0,Y_{\cS_{2,2}}(v_i,1)w_3\rangle \langle e^{z^{-1} L_1} w_2, Y_{\cS_{2,2}}(v_i', -z^{-1})e^{zL_1}z^{-2 L_0}e^{-\pi i (L_0-h)}w_1\rangle .
    \end{align*}
    This is a series in powers of $z$ of the form $\sum_{k=K}^\infty g_k\, z^{-2h+k}$ for some $K\in\ZZ$, where each coefficient $g_k$ is a finite linear combination of terms of the form
    \begin{align*}
        f_{i,j}(c)e^{-\pi i(\mathrm{wt}\,w_1-h)}\langle w_0, (v_i)_m w_3\rangle\langle L_1^r w_2, (v_j)_n L_1^s w_1\rangle.
    \end{align*}
    By the vertex operator formula \eqref{operator}, the invariance property \eqref{eqn:contra} of $\langle\cdot,\cdot\rangle$, and the assumption that $\langle v_{2,2},v_{2,2}\rangle=1$ (independent of $c$ and $h$), we conclude that $\langle w_0, (v_i)_m w_3\rangle$ and $\langle L_1^r w_2, (v_j)_n L_1^s w_1\rangle$ are polynomials in $c$ and $h$, and thus rational functions of $t$. Thus
    \begin{equation*}
       \langle w_0, Y_{\cS_{2,2}}(\cE(w_1,z)w_2,1)w_3\rangle =\sum_{k=K}^\infty g_k(t) z^{-2h+k}
    \end{equation*}
    where the coefficients $g_k(t)$ are rational functions of $t$. Now \eqref{eqn:analytic} holds with $q_n(t)=\sum_{i+k=n} \binom{2h^\NS_{2,2}(t)+N}{i} g_k(t)$.
    
\end{proof}

Since the coefficients $q_n(t)$ in \eqref{eqn:analytic} are rational functions, it is sufficient to compute them for $t\notin\QQ$. For this, we use the decomposition \eqref{modex2} of $\cS_{2,2}\otimes\cF(1)$ as an $L(c_a,0)\otimes L(c_b,0)$-module, where $a=\frac{1}{2}(t+1)$ and $b=\frac{1}{2}(t^{-1}+1)$. Since the free fermion vertex operator superalgebra $\cF(1)$ is simple and self-contragredient, it has a unique non-degenerate invariant bilinear form $\langle\cdot,\cdot\rangle$ such that $\langle\vac_{\cF(1)},\vac_{\cF(1)}\rangle=1$. Similar to the proof of Lemma \ref{lem:bilinear-form-symmetric}, $\langle\cdot,\cdot\rangle$ is also symmetric, and by \eqref{eqn:contra},
\begin{equation*}
    \langle \psi_n v',v\rangle = i(-1)^{\vert v'\vert}\langle v',\psi_{-n}v\rangle
\end{equation*}
for $v,v'\in\cF(1)$ and $n\in\ZZ+\frac{1}{2}$. In particular,
\begin{equation*}
    \langle\psi_{-\frac{1}{2}}\vac_{\cF(1)},\psi_{-\frac{1}{2}}\vac_{\cF(1)}\rangle = i\langle\vac_{\cF(1)},\psi_{\frac{1}{2}}\psi_{-\frac{1}{2}}\vac_{\cF(1)}\rangle = i\langle\vac_{\cF(1)},\vac_{\cF(1)}\rangle =i.
\end{equation*}
We now fix the non-degenerate symmetric invariant bilinear form on $\cS_{2,2}\otimes\cF(1)$ such that
\begin{equation*}
    \langle w'\otimes v', w\otimes v\rangle =\langle w',w\rangle\langle v',v\rangle
\end{equation*}
for $w,w'\in\cS_{2,2}$ and $v,v'\in\cF(1)$.

Now recall from \cite{FHL} that $Y_{\cS_{2,2}\otimes\cF(1)}=Y_{\cS_{2,2}}\otimes Y_{\cF(1)}$. We also define the intertwining operator $\widetilde{\cE}=\cE\otimes Y_{\cF(1)}$ of type $\binom{\cS_{1,1}\otimes\cF(1)}{\cS_{2,2}\otimes\cF(1)\,\,\cS_{2,2}\otimes\cF(1)}$. Recall also the $L(c_a,0)\otimes L(c_b,0)$-highest weight vector $u_{2,1}^a\otimes u_{2,1}^b$ from \eqref{u21timesu21}, as well as $u_{2,2}^a\otimes u_{2,2}^b=v_{2,2}\otimes\vac_{\cF(1)}$. Then using the definitions, we have
\begin{align*}
    &\left\langle u_{2,1}^a\otimes u_{2,1}^b, Y_{\cS_{2,2}\otimes\cF(1)}(\widetilde{\cE}(u_{2,1}^a\otimes u_{2,1}^b, 1-x)(u_{2,2}^a\otimes u_{2,2}^b),x)(u_{2,2}^a\otimes u_{2,2}^b)\right\rangle\nonumber\\
    & =\frac{4it}{(t-1)^2}\langle G_{-\frac{1}{2}}v_{2,2}, Y_{\cS_{2,2}}(\cE(G_{-\frac{1}{2}} v_{2,2},1-x)v_{2,2},x)v_{2,2}\rangle\langle\vac,Y_{\cF(1)}(Y_{\cF(1)}(\vac,1-x)\vac,x)\vac\rangle\nonumber\\
    &\qquad +\frac{2\sqrt{t}}{t-1}\langle G_{-\frac{1}{2}}v_{2,2}, Y_{\cS_{2,2}}(\cE( v_{2,2},1-x)v_{2,2},x)v_{2,2}\rangle\langle\vac,Y_{\cF(1)}(Y_{\cF(1)}(\psi_{-\frac{1}{2}}\vac,1-x)\vac,x)\vac\rangle\nonumber\\
    &\qquad +\frac{2\sqrt{t}}{t-1}\langle v_{2,2}, Y_{\cS_{2,2}}(\cE(G_{-\frac{1}{2}} v_{2,2},1-x)v_{2,2},x)v_{2,2}\rangle\langle\psi_{-\frac{1}{2}}\vac,Y_{\cF(1)}(Y_{\cF(1)}(\vac,1-x)\vac,x)\vac\rangle\nonumber\\
    & \qquad -i\langle v_{2,2}, Y_{\cS_{2,2}}(\cE( v_{2,2},1-x)v_{2,2},x)v_{2,2}\rangle\langle\psi_{-\frac{1}{2}}\vac,Y_{\cF(1)}(Y_{\cF(1)}(\psi_{-\frac{1}{2}}\vac,1-x)\vac,x)\vac\rangle.
\end{align*}
The second and third terms on the right vanish because the bilinear forms on $\cS_{2,2}$ and $\cF(1)$ are even, and in the last term we have
\begin{align*}
    \langle\psi_{-\frac{1}{2}}\vac,Y_{\cF(1)}(Y_{\cF(1)}(\psi_{-\frac{1}{2}}\vac,1-x)\vac,x)\vac\rangle & =\left\langle \psi_{-\frac{1}{2}}\vac,Y_{\cF(1)}\left(Y_{\cF(1)}\left(\psi_{-\frac{1}{2}}\vac,\frac{1-x}{x}\right)\vac,1\right)\vac\right\rangle\nonumber\\
    & =\left\langle\psi_{-\frac{1}{2}}\vac, e^{L_{-1}}e^{(\frac{1-x}{x})L_{-1}}\psi_{-\frac{1}{2}}\vac\right\rangle\nonumber\\
    & =\langle\psi_{-\frac{1}{2}}\vac,\psi_{-\frac{1}{2}}\vac\rangle =i.
\end{align*}
Thus
\begin{align}\label{eqn:SxF-to-S}
    &\left\langle u_{2,1}^a\otimes u_{2,1}^b, Y_{\cS_{2,2}\otimes\cF(1)}(\widetilde{\cE}(u_{2,1}^a\otimes u_{2,1}^b, 1-x)(u_{2,2}^a\otimes u_{2,2}^b),x)(u_{2,2}^a\otimes u_{2,2}^b)\right\rangle\nonumber\\
    &\qquad\qquad =  \langle v_{2,2}, Y_{\cS_{2,2}}(\cE( v_{2,2},1-x)v_{2,2},x)v_{2,2}\rangle\nonumber\\
    &\qquad\qquad\qquad +\frac{4it}{(t-1)^2}\langle G_{-\frac{1}{2}}v_{2,2}, Y_{\cS_{2,2}}(\cE(G_{-\frac{1}{2}} v_{2,2},1-x)v_{2,2},x)v_{2,2}\rangle\nonumber\\
    &\qquad\qquad = \langle v,Y_{\cS_{2,2}}(\cE(v,1-x)v_{2,2},x)v_{2,2}\rangle
\end{align}
for
\begin{equation*}
    v=v_{2,2}+ \frac{2e^{\pi i/4}\sqrt{t}}{t-1}G_{-\frac{1}{2}}v_{2,2} ,
\end{equation*}
where the second equality follows because the bilinear form on $\cS_{2,2}$ is even. Note that although we have been assuming $t\notin\QQ$, everything in the calculation \eqref{eqn:SxF-to-S}, including the $L(c_a,0)\otimes L(c_b,0)$-highest weight vectors $u_{2,1}^a\otimes u_{2,1}^b$ and $u_{2,2}^a\otimes u_{2,2}^b$, in fact makes sense for all $t$ except for $t=0,\pm 1$ and $t=\frac{p}{q}$ for $p,q\in\ZZ_{\geq 2}$ such that $p-q\in 2\ZZ$ and $\gcd(\frac{p-q}{2},q)=1$.

To compute the left side of \eqref{eqn:SxF-to-S} explicitly when $t\notin\QQ$, we view $\widetilde{\cE}$ and $Y_{\cS_{2,2}\otimes\cF(1)}$ as $L(c_a,0)\otimes L(c_b,0)$-module intertwining operators. First, by \cite[Theorem 5.2.5]{CJORY}, the image of the restriction of $\widetilde{\cE}$ to the submodules $L(c_a,h_{2,1}(a))\otimes L(c_b,h_{2,1}(b))$ and $L(c_a,h_{2,2}(a))\otimes L(c_b,h_{2,2}(b))$ is a quotient of
    \begin{align*}
        &(L(c_a,h_{2,1}(a))  \otimes L(c_b,h_{2,1}(b)))\btimes(L(c_a,h_{2,2}(a))\otimes L(c_b,h_{2,2}(b)))\nonumber\\
        &\qquad \cong L(c_a,h_{1,2}(a))\otimes L(c_b,h_{1,2}(b)) \oplus L(c_a,h_{1,2}(a))\otimes L(c_b,h_{3,2}(b))\nonumber\\
        &\qquad\qquad \oplus L(c_a,h_{3,2}(a))\otimes L(c_b,h_{1,2}(b))\oplus L(c_a,h_{3,2}(a))\otimes L(c_b,h_{3,2}(b)).
    \end{align*}
    Since only the first of these four direct summands is a submodule of $\cS_{1,1}\otimes\cF(1)$ by \eqref{eqn:irrational-dec}, we can identify the restriction of $\widetilde{\cE}$ with a tensor product intertwining operator $\cY_2^a\otimes\cY_2^b$, where $\cY_2^a$ is an $L(c_a,0)$-module intertwining operator of type $\binom{L(c_a,h_{1,2}(a))}{L(c_a,h_{2,1}(a))\,L(c_a,h_{2,2}(a))}$ and $\cY_2^b$ is an $L(c_b,0)$-module intertwining operator of type $\binom{L(c_b,h_{1,2}(b))}{L(c_b,h_{2,1}(b))\,L(c_b,h_{2,2}(b))}$.

    Next, we need to restrict $Y_{\cS_{2,2}}$ to $L(c_a,h_{1,2}(a))\otimes L(c_b,h_{1,2}(b))$ and $L(c_a,h_{2,2}(a))\otimes L(c_b,h_{2,2}(b))$, and then we need to project to $L(c_a,h_{2,1}(a))\otimes L(c_b,h_{2,1}(b))$. The result is a tensor product intertwining operator $\cY_1^a\otimes\cY_1^b$ where $\cY_1^a$ is an $L(c_a,0)$-module intertwining operator of type $\binom{L(c_a,h_{2,1}(a))}{L(c_a,h_{1,2}(a))\,L(c_a,h_{2,2}(a))}$ and $\cY_1^b$ is an $L(c_b,0)$-module intertwining operator of type $\binom{L(c_b,h_{2,1}(b))}{L(c_b,h_{1,2}(b))\,L(c_b,h_{2,2}(b))}$.

    Finally, the restriction of the non-degenerate invariant bilinear form on $\cS_{2,2}\otimes\cF(1)$ to $L(c_a,h_{2,1}(a))\otimes L(c_b,h_{2,1}(b))$ is the product of non-degenerate invariant bilinear forms on $L(c_a,h_{2,1}(a))$ and $L(c_b,h_{2,1}(b))$. Thus the left side of \eqref{eqn:SxF-to-S} factors as a product of Virasoro correlation functions,
    \begin{equation*}
\langle u_{2,1}^a,\cY_1^a(\cY_2^a(u_{2,1}^a,1-x)u_{2,2}^a,x)u_{2,2}^a\rangle \langle u_{2,1}^b,\cY_1^b(\cY_2^b(u_{2,1}^b,1-x)u_{2,2}^b,x)u_{2,2}^b\rangle.
    \end{equation*}
    By Theorem \ref{thmhyper}, this product of Virasoro correlation functions is some multiple of
    \begin{align*}
        &(1-x)^{-3(a+b-2)/2} x^{a+b-2-2(h_{2,2}(a)+h_{2,2}(b))} \cdot\nonumber\\
        &\qquad\qquad\cdot{}_2F_1\left(a,1-a;3-2a;-\frac{1-x}{x}\right) {}_2F_1\left(b,1-b;3-2b;-\frac{1-x}{x}\right)\nonumber\\
        & \qquad =(1-x)^{-2h_{2,2}^\NS(t)} x^{-2h_{2,2}^\NS(t)/3}\cdot\nonumber\\
        &\qquad\qquad \cdot {}_2F_1\left(\frac{1+t}{2},\frac{1-t}{2};2-t;-\frac{1-x}{x}\right) {}_2F_1\left(\frac{1+t^{-1}}{2},\frac{1-t^{-1}}{2};2-t^{-1};-\frac{1-x}{x}\right)\nonumber\\
        & \qquad =\left(\frac{1-x}{x}\right)^{-2h_{2,2}^\NS(t)} \left(1+\frac{1-x}{x}\right)^{8h_{2,2}^\NS(t)/3}\cdot\nonumber\\
        &\qquad \qquad \cdot {}_2F_1\left(\frac{1+t}{2},\frac{1-t}{2};2-t;-\frac{1-x}{x}\right) {}_2F_1\left(\frac{1+t^{-1}}{2},\frac{1-t^{-1}}{2};2-t^{-1};-\frac{1-x}{x}\right).
    \end{align*}
The multiple is just the coefficient of $(\frac{1-x}{x})^{-2h^\NS_{2,2}(t)}$ in the series expansion of 
\begin{align*}
    \langle v_{2,2}, Y_{\cS_{2,2}}(\cE( v_{2,2},1-x)v_{2,2},x)v_{2,2}\rangle +\frac{4it}{(t-1)^2}\langle G_{-\frac{1}{2}}v_{2,2}, Y_{\cS_{2,2}}(\cE(G_{-\frac{1}{2}} v_{2,2},1-x)v_{2,2},x)v_{2,2}\rangle.
\end{align*}
By \eqref{eqn:E-v22-v22} and the normalization $\langle v_{2,2},v_{2,2}\rangle =1$, the first term contributes $1$ to this coefficient. The second term contributes nothing since by the $L_0$-conjugation formula,
\begin{align*}
    &\langle G_{-\frac{1}{2}}v_{2,2},  Y_{\cS_{2,2}}(\cE(G_{-\frac{1}{2}} v_{2,2},1-x)v_{2,2},x)v_{2,2}\rangle\nonumber\\
    & =\left(1+\frac{1-x}{x}\right)^{2h_{2,2}^\NS(t)}\bigg\langle G_{-\frac{1}{2}}v_{2,2},Y_{\cS_{2,2}}\bigg(\left(\frac{1-x}{x}\right)^{L_0-2h_{2,2}^\NS(t)-\frac{1}{2}}\cE(G_{-\frac{1}{2}}v_{2,2},1)v_{2,2},1\bigg)v_{2,2}\bigg\rangle,
\end{align*}
and the lowest conformal weight of $\cS(c,0)^{\bar{1}}$ is $\frac{3}{2}$. Thus recalling \eqref{eqn:SxF-to-S}, we conclude that
\begin{align}\label{eqn:explicit-NS-corr-fun}
    &\langle v, Y_{\cS_{2,2}}( \cE(v,1-x)v_{2,2},x)v_{2,2}\rangle =\left(\frac{1-x}{x}\right)^{-2h_{2,2}^\NS(t)} \left(1+\frac{1-x}{x}\right)^{8h_{2,2}^\NS(t)/3}\cdot\nonumber\\
        & \qquad\cdot {}_2F_1\left(\frac{1+t}{2},\frac{1-t}{2};2-t;-\frac{1-x}{x}\right) {}_2F_1\left(\frac{1+t^{-1}}{2},\frac{1-t^{-1}}{2};2-t^{-1};-\frac{1-x}{x}\right)
\end{align}
as a series in powers of $\frac{1-x}{x}$, for $v=v_{2,2}+\frac{2e^{\pi i/4}\sqrt{t}}{t-1}G_{-\frac{1}{2}}v_{2,2}$ and $t\notin\QQ$.

Note that the left side of \eqref{eqn:explicit-NS-corr-fun} is defined for all $t\in\CC^\times$ except for $t=1$ and $t=\frac{p}{q}$ for $p,q\in\ZZ_{\geq 2}$ such that $p-q\in 2\ZZ$ and $\gcd(\frac{p-q}{2},q)=1$. Moreover, by Proposition \ref{analyticcoefficents}, the coefficients of powers of $\frac{1-x}{x}$ on the left side of \eqref{eqn:explicit-NS-corr-fun} are rational functions of $t$ whose poles can occur only at these exceptional values of $t$. On the other hand, the coefficients of powers of $\frac{1-x}{x}$ on the right side of \eqref{eqn:explicit-NS-corr-fun} are also rational functions of $t$, since $h_{2,2}^\NS(t)$ is a Laurent polynomial and since
\begin{equation*}
    {}_2F_1(\alpha,\beta;\gamma;z) =\sum_{n=0}^\infty \frac{(\alpha)_n(\beta)_n}{(\gamma)_n n!} z^n
\end{equation*}
for $\alpha,\beta,\gamma\in\CC$, where $(q)_n=q(q+1)\cdots(q+n-1)$ is the rising Pochhammer symbol for $q\in\CC$. Observe that the coefficients of powers of $\frac{1-x}{x}$ on the right side of \eqref{eqn:explicit-NS-corr-fun} can possibly have poles only at values of $t$ such that $(2-t^{\pm1})_n=0$ for some $n\in\ZZ_{\geq 1}$, that is, only when $t^{\pm 1}\in\ZZ_{\geq 2}$. However, there are actually no poles when $t^{\pm1}=2m+1$, $m\in\ZZ_{\geq 1}$, is an odd integer, since then the problematic factor in the Pochhammer symbol $(2-t^{\pm1})_n$, $n\geq 2m$, is canceled by a factor in $(\frac{1-t^{\pm1}}{2})_n$.

By the above discussion, \eqref{eqn:explicit-NS-corr-fun} holds for $t\in\QQ_{<0}$ as well as for all $t\notin\QQ$. Further, \eqref{eqn:explicit-NS-corr-fun} holds when $t^{\pm1}=2m+1$ for $m\in\ZZ_{\geq 1}$, as long as we interpret the hypergeometric function
\begin{equation*}
    {}_2 F_1\left(\frac{1+t^{\pm1}}{2},\frac{1-t^{\pm1}}{2};2-t;z\right) = {}_2 F_1(m+1,-m;-2m+1;z)
\end{equation*}
appropriately, where $z=-\frac{1-x}{x}$. Specifically, note that
\begin{equation*}
    \frac{(m+1)_n (-m)_n}{(-2m+1)_n n!} =\left\lbrace\begin{array}{ccc}
(-1)^n\binom{m}{n}\frac{(m+1)_n}{(-2m+1)_n} & \text{if} & 0\leq n\leq m\\
0 & \text{if} & m<n<2m\\
    \end{array}\right. ,
\end{equation*}
while if $n\geq 2m$, then the $\frac{1-t^{\pm 1}}{2}+m$ factor in $(\frac{1-t^{\pm1}}{2})_n$ cancels with the $2-t^{\pm1}+2m-1$ factor in $(2-t^{\pm1})_n$ to yield $\frac{1}{2}$, and thus 
\begin{align*}
    \left.\frac{(\frac{1+t^{\pm1}}{2})_n (\frac{1-t^{\pm1}}{2})_n}{(2-t^{\pm1})_n n!}\right\vert_{t^{\pm1}=2m+1} & =\frac{1}{2}\frac{(m+1)_n (-m)\cdots (-1)(1)\cdots(-m+n-1)}{(-2m+1)\cdots(-1)(1)\cdots(-2m+n) n!}\nonumber\\
    & \hspace{-4em} = -\frac{(-1)^m}{2}\frac{(m+1)_n m! (n-m-1)!}{(2m-1)!(n-2m)!n!}\nonumber\\
    & \hspace{-4em} = -\frac{(-1)^m}{2}\frac{(m+1)\cdots(3m)m!(m-1)!}{(2m-1)!(2m)!}\frac{(3m+1)_{n-2m}(m)_{n-2m}}{(2m+1)_{n-2m}(n-2m)!}.
\end{align*}
Thus when $t^{\pm1}=2m+1$, $m\in\ZZ_{\geq 1}$, we should replace ${}_2 F_1(\frac{1+t^{\pm1}}{2},\frac{1-t^{\pm1}}{2};2-t^{\pm1};-\frac{1-x}{x})$ with
\begin{align}\label{eqn:explicit-2m+1}
    \sum_{n=0}^m \binom{m}{n} & \frac{(m+1)_n}{(-2m+1)_n} \left(\frac{1-x}{x}\right)^n\nonumber\\
    &-\frac{(-1)^m}{2}\frac{(3m)!(m-1)!}{(2m-1)!(2m)!}\left(\frac{1-x}{x}\right)^{2m}{}_2F_1\left(3m+1,m;2m+1;-\frac{1-x}{x}\right)
\end{align}
in \eqref{eqn:explicit-NS-corr-fun}. By \cite[Equation 15.8.12]{DLMF}, we can simplify the hypergeometric series in this expression as follows:
\begin{align*}
    {}_2F_1\left(3m+1,m;2m+1;-\frac{1-x}{x}\right) & =\left(1+\frac{1-x}{x}\right)^{-2m} {}_2F_1\left(m+1,-m;2m+1;-\frac{1-x}{x}\right)\nonumber\\
    & = x^{2m} \sum_{n=0}^m \frac{(m+1)_n(-m)_n}{(2m+1)_n n!}\left(-\frac{1-x}{x}\right)^n\nonumber\\
    & =x^{2m}\sum_{n=0}^m \binom{m}{n}\frac{(m+1)_n}{(2m+1)_n}x^{-n}(1-x)^n.
\end{align*}
Thus \eqref{eqn:explicit-2m+1} becomes
\begin{align}\label{eqn:explicit-2m+1-better}
    x^{-m}(1-x)^{2m}&\sum_{n=0}^m \binom{m}{n}(m+1)_n x^{m-n}\cdot\nonumber\\
    &\cdot\left(\frac{(1-x)^{n-2m}}{(-2m+1)_n}-\frac{(-1)^m}{2}\frac{(3m)!(m-1)!}{(2m-1)!(2m)!}\frac{(1-x)^{n}}{(2m+1)_n}\right).
\end{align}

\subsection{Computing the rigidity composition}

We now complete the proof that $\cS_{2,2}$ is rigid:
\begin{theorem}\label{thm:S22-rigid}
    Assume that $t^{\pm1}\notin\ZZ_{\leq 0}$ and that $t\neq\frac{p}{q}$ for some $p,q\in\ZZ_{\geq 2}$ such that $p-q\in 2\ZZ$ and $\gcd(\frac{p-q}{2},q)=1$. Then $\cS_{2,2}$ is rigid and self-dual in $\cO_{c^\NS(t)}^{\mathrm{fin}}$.
\end{theorem}
\begin{proof}
    For $t\notin\QQ$, this is already proved in Theorem \ref{thm:t-irrational-properties}, and for $t=1$, this is already proved in Theorem \ref{thm:t=1-properties}. For the remaining cases, it is enough to show that the coefficient of $x^{-2h^\NS_{2,2}(t)}(\log x)^0$ in the expansion of 
    \begin{equation}\label{eqn:corr-fun-t-in-Q}
        \left(x\frac{d}{dx}-\frac{2}{3}h_{2,2}^\NS(t)\right)\langle v, Y_{\cS_{2,2}}(\cE(v,1-x)v_{2,2},x)v_{2,2}\rangle
    \end{equation}
    as a series in powers of $x$ and $\log x$ is non-zero, where we take $v=v_{2,2}+\frac{2e^{\pi i/4}\sqrt{t}}{t-1}G_{-\frac{1}{2}}v_{2,2}$. As previously, set $h=h_{2,2}^\NS(t)$ for brevity.

    For $t\in\QQ_{<0}$ such that $t^{\pm1}\notin\ZZ_{\leq-1}$, we use the explicit formula \eqref{eqn:explicit-NS-corr-fun}. The hypergeometric connection formula \cite[Equation 15.10.18]{DLMF} says that
    \begin{align}\label{eqn:connection-formula}
        &x^{-(1-t^{\pm1})/2}  (1-x)^{1-t^{\pm1}}  {}_2F_1\left(\frac{1+t^{\pm1}}{2},\frac{1-t^{\pm1}}{2};2-t^{\pm1};-\frac{1-x}{x}\right)\nonumber\\
        &\qquad= \frac{\Gamma(t^{\pm1})\Gamma(2-t^{\pm1})}{\Gamma(\frac{1+t^{\pm1}}{2})\Gamma(\frac{3-t^{\pm1}}{2})} {}_2F_1\left(\frac{1-t^{\pm1}}{2},-\frac{1-t^{\pm1}}{2};1-t^{\pm1};x\right)\nonumber\\
        &\qquad\qquad +\frac{\Gamma(-t^{\pm1})\Gamma(2-t^{\pm1})}{\Gamma(\frac{1-t^{\pm1}}{2})\Gamma(\frac{3-3t^{\pm1}}{2})} x^{t^{\pm1}} {}_2F_1\left(\frac{1+t^{\pm1}}{2},-\frac{1-3t^{\pm1}}{2};1+t^{\pm1};x\right);
    \end{align}
    note all gamma function values and hypergeometric series here are defined since $t^{\pm1}\notin\ZZ$. Thus using \eqref{eqn:explicit-NS-corr-fun}, the expansion of \eqref{eqn:corr-fun-t-in-Q} as a series in $x$ is a linear combination of four terms, whose leading powers of $x$ are:
    \begin{equation*}
        x^{-2h/3+(1-t)/2+(1-t^{-1})/2+C} = x^{-2h+C},\qquad C=0,t,t^{-1},t+t^{-1}.
    \end{equation*}
    Since $t^{\pm1},t+t^{-1}\notin\ZZ_{\leq 0}$ by our assumptions on $t$ (recall the end of the proof of Proposition \ref{prop:coev}), only the first term in the linear combination contributes to the coefficient of $x^{-2h}(\log x)^0$ in \eqref{eqn:corr-fun-t-in-Q}. In particular, we need the coefficient of the lowest power of $x$ in
    \begin{align*}
      \frac{\Gamma(t)\Gamma(2-t)}{\Gamma(\frac{1+t}{2})\Gamma(\frac{3-t}{2})} & \frac{\Gamma(t^{-1})\Gamma(2-t^{-1})}{\Gamma(\frac{1+t^{-1}}{2})\Gamma(\frac{3-t^{-1}}{2})}  \left(x\frac{d}{dx}-\frac{2h}{3}\right) x^{-2h}(1-x)^{2h/3}\cdot\nonumber\\
      &\cdot{}_2F_1\left(\frac{1-t}{2},-\frac{1-t}{2};1-t;x\right){}_2F_1\left(\frac{1-t^{-1}}{2},-\frac{1-t^{-1}}{2};1-t^{-1};x\right).
    \end{align*}
    Using gamma function properties (recorded in \cite[Section 5.5]{DLMF}), this coefficient is
    \begin{align*}
        \left(-2h-\frac{2h}{3}\right)& \frac{(1-t)\Gamma(t)\Gamma(1-t)}{\frac{1-t}{2}\Gamma(\frac{1+t}{2})\Gamma(\frac{1-t}{2})}\frac{(1-t^{-1})\Gamma(t^{-1})\Gamma(1-t^{-1})}{\frac{1-t^{-1}}{2}\Gamma(\frac{1+t^{-1}}{2})\Gamma(\frac{1-t^{-1}}{2})} \nonumber\\
        &=-\frac{4(t-1)^2}{t} \frac{\sin(\frac{\pi}{2}(1+t))}{\sin(\pi t)}\frac{\sin(\frac{\pi}{2}(1+t^{-1}))}{\sin(\pi t^{-1})}\nonumber\\
        & =-\frac{(t-1)^2}{t\sin(\pi t/2)\sin(\pi t^{-1}/2)}.
    \end{align*}
    This coefficient is defined and non-zero as long as $t\neq 1$ and $t^{\pm1}\notin 2\ZZ$. In particular, when $t\in\QQ_{<0}\setminus\ZZ_{\leq -1}$,
    \begin{equation*}
        \langle v,\mathfrak{R}(v)\rangle = -\frac{(t-1)^2}{t\sin(\pi t/2)\sin(\pi t^{-1}/2)}\neq 0,
    \end{equation*}
    where $\mathfrak{R}$ is the rigidity composition defined using $\mathrm{ev}$ and $\mathrm{coev}$. Since $\cS_{2,2}$ is simple, it follows that $\mathfrak{R}$ is a non-zero scalar multiple of $\id_{\cS_{2,2}}$, and then we can rescale either $\mathrm{ev}$ or $\mathrm{coev}$ to get the identity. Hence $\cS_{2,2}$ is rigid and self-dual in this case.

    For $t=2m+1$, $m\in\ZZ_{\geq 1}$, we use \eqref{eqn:explicit-NS-corr-fun} and \eqref{eqn:explicit-2m+1-better}. In this case $t^{-1}=\frac{1}{2m+1}\notin\ZZ$, so the $t^{-1}$ case of \eqref{eqn:connection-formula} is valid. Thus by \eqref{eqn:explicit-NS-corr-fun}, \eqref{eqn:explicit-2m+1-better}, and \eqref{eqn:connection-formula}, the expansion of \eqref{eqn:corr-fun-t-in-Q} as a series in powers of $x$ is a linear combination of two terms, whose leading powers of $x$ are:
    \begin{equation*}
        x^{-2h/3-m+(1-t^{-1})/2+C} = x^{-2h+C},\qquad C=0,t^{-1}.
    \end{equation*}
    Since $t^{-1}\notin\ZZ$, only the first term contributes to the power of $x^{-2h}(\log x)^0$ in \eqref{eqn:corr-fun-t-in-Q}. In particular, we need the coefficient of the lowest power of $x$ in
    \begin{align*}
        &\frac{\Gamma(t^{-1})\Gamma(2-t^{-1})}{\Gamma(\frac{1+t^{-1}}{2})\Gamma(\frac{3-t^{-1}}{2})}  \left(x\frac{d}{dx}-\frac{2h}{3}\right) x^{-2h}(1-x)^{2h/3}{}_2F_1\left(\frac{1-t^{-1}}{2},-\frac{1-t^{-1}}{2};1-t^{-1};x\right)\cdot\nonumber\\
        &\qquad\cdot\sum_{n=0}^m \binom{m}{n}(m+1)_n x^{m-n}\left(\frac{(1-x)^{n-2m}}{(-2m+1)_n}-\frac{(-1)^m}{2}\frac{(3m)!(m-1)!}{(2m-1)!(2m)!}\frac{(1-x)^{n}}{(2m+1)_n}\right).
    \end{align*}
    This coefficient is
    \begin{align*}
        &-\frac{8h}{3}\frac{(1-t^{-1})\Gamma(t^{-1})\Gamma(1-t^{-1})}{\frac{1-t^{-1}}{2}\Gamma(\frac{1+t^{-1}}{2})\Gamma(\frac{1-t^{-1}}{2})}(m+1)_m\cdot\nonumber\\
        &\qquad\qquad\cdot\left(\frac{1}{(-2m+1)_m}-\frac{(-1)^m}{2}\frac{(3m)!(m-1)!}{(2m-1)!(2m)!(2m+1)_m}\right)\nonumber\\
        &\qquad =-\frac{2(t-1)^2}{t}\frac{\sin(\frac{\pi}{2}(1+t^{-1}))}{\sin(\pi t^{-1})}\left((-1)^m\frac{(m+1)_m}{(m)_m}-\frac{(-1)^m}{2}\frac{(m+1)_m(m-1)!}{(2m-1)!}\right)\nonumber\\
        &\qquad =-\frac{(t-1)^2}{t\sin(\pi t^{-1}/2)}(-1)^m = -\frac{(t-1)^2}{t\sin(\pi t/2)\sin(\pi t^{-1}/2)},
    \end{align*}
    where the last equality follows because $t=2m+1$. As in the previous case, this coefficient is non-zero, and thus $\cS_{2,2}$ is rigid and self-dual when $t^{\pm1}=2m+1$, $m\in\ZZ_{\geq 1}$.
    
\end{proof}

Since $\cS_{2,2}$ is simple and self-dual (at least for most values of $t$), and since the unit object $\cS_{1,1}$ of $\cO_{c^\NS(t)}^{\mathrm{fin}}$ is simple, composing an evaluation of $\cS_{2,2}$ with a coevaluation yields a unique scalar multiple of $\id_{\cS_{1,1}}$. This scalar multiple is called the \textit{intrinsic dimension} of $\cS_{2,2}$ in $\cO_{c^\NS(t)}^{\mathrm{fin}}$. For $t=1$, the intrinsic dimension of $\cS_{2,2}$ is $1$ because $\cS_{2,2}=\cS_{1,1}$. For all other values of $t$ covered in Theorem \ref{thm:S22-rigid}, we use the proof of Theorem \ref{thm:S22-rigid} (as well as Theorem \ref{thm:t-irrational-properties} for the case $t\notin\QQ$) to calculate the intrinsic dimension of $\cS_{2,2}$:
\begin{prop}\label{prop:intrinsic-dim}
    Assume that $t^{\pm1}\notin\ZZ_{\leq 1}$ and that $t\neq\frac{p}{q}$ for some $p,q\in\ZZ_{\geq 2}$ such that $p-q\in 2\ZZ$ and $\gcd(\frac{p-q}{2},q)=1$. Then the intrinsic dimension of the self-dual object $\cS_{2,2}$ in $\cO^{\mathrm{fin}}_{c^\NS(t)}$ is $4\sin(\pi t/2)\sin(\pi t^{-1}/2)$.
\end{prop}
\begin{proof}
    For $t\notin\QQ$, \eqref{eqn:embedding-on-objects} and Theorem \ref{thm:t-irrational-properties}(1) imply that the intrinsic dimension of $\cS_{2,2}$ is the product of the intrinsic dimensions of $L(c_a,h_{2,1}(a))$ and $L(c_b,h_{2,1}(b))$ in the Virasoro tensor categories $\cC_{c_a}^{Vir}$ and $\cC_{c_b}^{Vir}$, where $a=\frac{1}{2}(t+1)$ and $b=\frac{1}{2}(t^{-1}+1)$. The $L(c_a,0)$-module $L(c_a,h_{2,1}(a))$ generates a tensor subcategory of $\cC_{c_a}^{Vir}$ which is tensor equivalent to the category of weight modules for the quantum group $U_q(\mathfrak{sl}_2)$ at $q=e^{\pi i a}$ (see \cite[Remark 8.4]{MY2}, or \cite[Proposition 5.5.2]{CJORY} where the equivalence was given earlier but less explicitly). The intrinsic dimension of the tensor generator of $\mathrm{Rep}\,U_q(\mathfrak{sl}_2)$ is $-q-q^{-1}$ (see for example \cite[Exercise 8.18.8]{EGNO}), so we see that the intrinsic dimension of $\cS_{2,2}$ in $\cO^{\mathrm{fin}}_{c^\NS(t)}$ for $t\notin\QQ$ is
    \begin{align*}
        (-e^{\pi i a}-e^{-\pi i a})(-e^{\pi i b}-e^{-\pi ib}) & = \left(-2\cos\left(\frac{\pi}{2}(t+1)\right)\right)\left(-2\cos\left(\frac{\pi}{2}(t^{-1}+1)\right)\right)\nonumber\\
        & = 4\sin(\pi t/2)\sin(\pi t^{-1}/2),
    \end{align*}
    as required.

    For $t^{\pm1}\in\QQ_{<0}\setminus\ZZ_{\leq -1}$ or $t^{\pm1}=2m+1$, $m\in\ZZ_{\geq 1}$, the proof of Theorem \ref{thm:S22-rigid} shows that the composition of correctly normalized evaluation and coevaluation morphisms for $\cS_{2,2}$ is
    \begin{equation*}
        -\frac{\langle v,v\rangle\, t\sin(\pi t/2)\sin(\pi t^{-1}/2)}{(t-1)^2} \mathrm{ev}\circ\mathrm{coev},
    \end{equation*}
    where $v=v_{2,2}+\frac{2e^{\pi i/4}\sqrt{t}}{t-1} G_{-\frac{1}{2}} v_{2,2}$. Thus
    \begin{align*}
        \langle v, v\rangle & = \langle v_{2,2},v_{2,2}\rangle +\frac{4it}{(t-1)^2}\langle G_{-\frac{1}{2}} v_{2,2},G_{-\frac{1}{2}}v_{2,2}\rangle\nonumber\\
        & = 1 +\frac{4it}{(t-1)^2}(-i\langle v_{2,2}, G_{\frac{1}{2}}G_{-\frac{1}{2}}v_{2,2}\rangle) = 1 +\frac{3}{2h^\NS_{2,2}(t)} 2h^\NS_{2,2}(t)\langle v_{2,2},v_{2,2}\rangle = 4.
    \end{align*}
    Also, the definitions of $\mathrm{ev}$ and $\mathrm{coev}$ imply that $\mathrm{ev}\circ\mathrm{coev}(\vac)$ is the coefficient of $x^{-2h^\NS_{2,2}(t)}$ in
    \begin{align*}
        \left(x\dfrac{d}{dx}-\frac{2}{3}h^\NS_{2,2}(t)\right)\cE(v_{2,2},x)v_{2,2}),
    \end{align*}
    and this coefficient is $-\frac{8}{3}h^\NS_{2,2}(t)$ by \eqref{eqn:E-v22-v22}. It follows that the intrinsic dimension of $\cS_{2,2}$ is
    \begin{equation*}
        \left(-\frac{4t\sin(\pi t/2)\sin(\pi t^{-1}/2)}{(t-1)^2}\right)\left(-\frac{8}{3}h^\NS_{2,2}(t)\right) =4 \sin(\pi t/2)\sin(\pi t^{-1}/2),
    \end{equation*}
    as required.
    
\end{proof}

\appendix
\section{The image of \texorpdfstring{$u_{2,1}^a\otimes u_{2,1}^b$}{ua(2,1) x ub(2,1)} in \texorpdfstring{$\cS_{2,2}\otimes \mathcal{F}(1)$}{S(2,2) x F(1)}} \label{N1FERM}

Let $t\neq 0,\pm1$, and set $a=\frac{1}{2}(t+1)$ and $b=\frac{1}{2}(t^{-1}+1)$. Here we show explicitly how the Virasoro algebras of central charge $c_a$ and $c_b$ embed as commuting subalgebras of $\cS(c^\NS(t),0)\otimes\cF(1)$. We will then identify a simultaneous highest-weight vector for these commuting Virasoro algebras of highest weights $(h_{2,1}(a),h_{2,1}(b))$ in $\cS_{2,2}\otimes\cF(1)$. This will determine an embedding
\begin{equation*}
 L(c_a,h_{2,1}(a))\otimes L(c_b,h_{2,1}(b))\hookrightarrow\cS_{2,2}\otimes\cF(1)   
\end{equation*}
of $L(c_a,0)\otimes L(c_b,0)$-modules at least when $t\notin\QQ$.

First, the conformal weight $2$ space of $S(c^\NS_t, 0)\otimes \mathcal{F}(1)$ has a basis given by the $\cS(c^\NS(t),0)$-conformal vector $L^\NS=L_{-2}\vac_\NS\otimes\vac_{\cF(1)}$, the normal-ordered product $:\psi G: =-G_{-\frac{3}{2}}\vac_\NS\otimes\psi_{-\frac{1}{2}}\vac_{\cF(1)}$, and the $\mathcal{F}(1)$-conformal vector $L^\psi=\vac_\NS\otimes\frac{1}{2}\psi_{-\frac{3}{2}}\psi_{-\frac{1}{2}}\vac_{\cF(1)}$. Thus the conformal vectors $L^a, L^b\in\cS(c^\NS(t),0)\otimes\cF(1)$ which generate $L(c_a,0)$ and $L(c_b,0)$ are given by linear combinations
\begin{align*}
&L^a=a_1(t)L^\NS+a_2(t):\psi G:+a_3(t)L^\psi\\
&L^b=b_1(t)L^\NS+b_2(t):\psi G:+b_3(t)L^\psi
\end{align*}
for certain functions $\{a_i(t), b_i(t)| 1\leq i\leq 3\}$ of $t$. Using Thielemans' \texttt{Mathematica} package \cite{T}, one can show that 
\begin{align*}
&a_1(t)=\frac{t}{1 + t},\;\;\; a_2(t)=\pm \frac{\sqrt{t}}{\sqrt{-1 - 2 t - t^2}},\;\;\; a_3(t)=\frac{2 -  t}{1 + t},\\
&b_1(t)=\frac{1}{1+t},\;\;\; b_2(t)= \mp \frac{\sqrt{t}}{\sqrt{-1-2t-t^2}},\;\;\; b_3(t) =\frac{2t-1}{1+t},
\end{align*}
where we need $a_2=-b_2$ for the conformal vectors $L^a$ and $L^b$ to commute and add to $L^\NS+L^\psi$. Thus we fix choices for the square roots $\sqrt{t}$ and $\sqrt{-1}=i$, set $\sqrt{-(t+1)^2}=i(t+1)$, and then set $a_2(t)=+\frac{\sqrt{t}}{i(t+1)}=-b_2(t)$.

Now consider $\cS_{2,2}\otimes\cF(1)$ and fix a highest-weight vector $v_{2,2}\in\cS_{2,2}$. Let
\begin{align*}
v_1=G_{-\frac{1}{2}} v_{2,2} \otimes \vac_{\mathcal{F}(1)}, \qquad v_2=v_{2,2}\otimes \psi_{-\frac{1}{2}}\vac_{\mathcal{F}(1)},
\end{align*}
so that $\lbrace v_1,v_2\rbrace$ is a basis for the conformal weight $h^\NS_{2,2}(t)+\frac{1}{2}$ space of $\cS_{2,2}\otimes\cF(1)$. 
We want to find a linear combination of $v_1$ and $v_2$ that is a simultaneous eigenvector for $L^a_0$ and $L^b_0$ with eigenvalues $h_{2,1}(a)$ and $h_{2,1}(b)$, respectively.
We calculate
\begin{align*}
&L^a_0v_1=\left(a_1(t)\left(h_{2,2}(t)+\frac{1}{2}\right)\right)v_1+ 2a_2(t)h_{2,2}(t)v_2\\
&L^a_0v_2=- a_2(t)v_1+\left(a_1(t)h_{2,2}(t)+\frac{1}{2}a_3(t)\right)v_2\\
&L^b_0v_1=\left(b_1(t)\left(h_{2,2}(t)+\frac{1}{2}\right)\right)v_1+ 2b_2(t)h_{2,2}(t)v_2\\
&L^b_0v_2=-b_2(t)v_1+ \left(b_1(t)h_{2,2}(t)+\frac{1}{2}b_3(t)\right)v_2,
\end{align*}
so that the matrices of $L_0^a$ and $L_0^b$ in the basis $\lbrace v_1,v_2\rbrace$ are 
\begin{align*}
&[L_0^a] = \left[ \begin{array}{cc} 
a_1(t)\left(h_{2,2}(t)+\frac{1}{2}\right) & -a_2(t)  \\
 2a_2(t)h_{2,2}(t)  & a_1(t)h_{2,2}(t)+\frac{1}{2}a_3(t)  
\end{array} \right], \\
& [L_0^b] = \left[ \begin{array}{cc} 
b_1(t)\left(h_{2,2}(t)+\frac{1}{2}\right) & -b_2(t)  \\
 2b_2(t)h_{2,2}(t)  & b_1(t)h_{2,2}(t)+\frac{1}{2}b_3(t)  
\end{array} \right].
\end{align*}
Then \texttt{Mathematica} calculates the eigenvalues of $[L^a_0]$ to be $h_{2,1}(a)$ and $h_{2,3}(a)$ as expected, with respective eigenvectors
\begin{equation*}
    w_{2,1} =\left[\begin{array}{c}
\frac{2i\sqrt{t}}{t-1}\\
1\\
        \end{array}\right],\qquad w_{2,3} =\left[\begin{array}{c}
-\frac{2i\sqrt{t}}{3(t-1)}\\
1\\
        \end{array}\right].
\end{equation*}
Similarly, the eigenvalues of $[L^b_0]$ are $h_{2,1}(b)$ and $h_{2,3}(b)$ with the same two eigenvectors. For technical reasons that are relevant in Section \ref{sec:rigidity},  we multiply $w_{2,1}$ by $e^{-\pi i/4}$, so that when $t\neq 0,1,-1$,
\begin{equation*}
 \frac{2e^{\pi i/4}\sqrt{t}}{t-1} G_{-\frac{1}{2}}v_{2,2}\otimes\vac_{\cF(1)} +e^{-\pi i/4}\, v_{2,2}\otimes\psi_{-\frac{1}{2}}\vac_{\cF(1)}   
\end{equation*}
is a simultaneous highest-weight vector in $\cS_{2,2}\otimes\cF(1)$ for the two commuting Virasoro algebras, of highest weights $(h_{2,1}(a), h_{2,1}(b))$.

\section{Differential equations for Virasoro correlation functions}

In this appendix, we assume $\ell \notin \QQ$. From Theorem 5.2.2 and Theorem 5.2.5 in \cite{CJORY}, there are non-zero intertwining operators $\cY_1$ and $\cY_2$ of types
\[
\binom{L(c_\ell, h_{2,1}(\ell))}{L(c_\ell, h_{1,2}(\ell))\; L(c_\ell, h_{2,2}(\ell))}\qquad\text{and}\qquad\binom{L(c_\ell, h_{1,2}(\ell))}{L(c_\ell, h_{2,1}(\ell))\; L(c_\ell, h_{2,2}(\ell))},
\]
respectively. Let $u_{2,1}$ and $u_{2,2}$ be highest-weight vectors in $L(c_\ell, h_{2,1}(\ell))$ and $L(c_\ell, h_{2,2}(\ell))$, respectively, let $\langle\cdot,\cdot\rangle$ be an invariant bilinear form on $L(c_\ell,h_{2,1}(\ell))$, and define 
\[
\psi(z)=\langle u_{2,1}, \cY_1(\cY_2(u_{2,1}, 1-z)u_{2,2}, z)u_{2,2}\rangle,
\]
a multivalued analytic function defined on the region $\vert z\vert>\vert 1-z\vert>0$.
For brevity, denote $h_{2,2}(\ell)$ by $h$. We will prove:
\begin{theorem} \label{teodiff} 
The analytic function $\psi(z)$ satisfies the differential equation
\begin{align}\label{eqn:psi-diff-eqn}
&z(1-z) \psi''(z)+\left((4h+2-\ell)(1-z)-\ell\right)\psi'(z)\nonumber\\
& \ \ \ \ \ +\frac{1}{z(1-z)}\left((2h(2h+1)-3\ell h)(1-z)^{2}-\ell h\right)\psi(z)=0.
\end{align}
\end{theorem}

We prove Theorem \ref{teodiff} using the following lemma. First define 
\begin{align*}
  \Psi(x_0,x_2) & = \langle u_{2,1}, \cY_1(\cY_2(u_{2,1}, x_0)u_{2,2}, x_2) u_{2,2}\rangle,
\end{align*}
which we can view as a formal series in powers of $x_0$ and $x_2$.
\begin{lemma} \label{diffeqv1}
The formal series $\Psi(x_0,x_2)$ satisfies the formal differential equation
\begin{align}
\partial^2_{x_0}\Psi(x_0, x_2)=&\ell(x_0^{-1}-(x_2+x_0)^{-1})\partial_{x_2}\Psi(x_0, x_2) \nonumber\\
&\ \ \ -\ell x_0^{-1}\partial_{x_0}\Psi(x_0, x_2) + \ell h(x_0^{-2}+(x_2+x_0)^{-2})\Psi(x_0, x_2). \label{equlem1}
\end{align}
\end{lemma}

\begin{proof}
It is easy to check that $(L_{-1}^2-\ell L_{-2}) u_{2,1}$ is a singular vector in the Verma $L(c_\ell,0)$-module with lowest conformal weight $h_{2,1}(\ell)$ and thus vanishes in $L(c_\ell, h_{2,1}(\ell))$. Thus
\begin{align} \label{rhlll}
\partial^2_{x_0} \Psi(x_0,x_2)= \ell\langle u_{2,1}, \cY_1(\cY_2(L_{-2}u_{2,1}, x_0)u_{2,2}, x_2)u_{2,2}\rangle.
\end{align} 
by the $L_{-1}$-derivative property.
Using the associator formula \eqref{eqn:asso1},
\begin{align} 
\langle u_{2,1}, \cY_1(\cY_2(L_{-2}u_{2,1}, x_0)u_{2,2}, x_2)u_{2,2}\rangle & =\sum_{i \geq 0}\left(x_0^i\langle u_{2,1}, \cY_1(L_{-2-i}\cY_2(u_{2,1},x_0)u_{2,2}, x_2)u_{2,2}\rangle\right.\nonumber\\
&\quad\quad \left.+ x_0^{-1-i} \langle u_{2,1}, \cY_1(\cY_2(u_{2,1}, x_0)L_{i-1}u_{2,2}, x_2) u_{2,2}\rangle\right). \label{rhs1}
\end{align}
Since $L_{i-1}u_{2,2}=0$ for $i\geq 2$ and $L_0 u_{2,2}= hu_{2,2}$, the right side of \eqref{rhs1} equals
\begin{align}
&\sum_{i \geq 0}x_0^i\langle u_{2,1}, \cY_1(L_{-2-i} \cY_2(u_{2,1},x_0)u_{2,2}, x_2)u_{2,2}\rangle + x_0^{-1}\langle u_{2,1}, \cY_1(\cY_2(u_{2,1}, x_0)L_{-1}u_{2,2}, x_2) u_{2,2} \rangle \nonumber\\
&\qquad  + hx_0^{-2} \langle u_{2,1}, \cY_1(\cY_2(u_{2,1}, x_0)u_{2,2}, x_2) u_{2,2}\rangle. \label{fito}
\end{align}
Using the associator formula \eqref{eqn:asso1} on the first term of the equation (\ref{fito}) together with the fact that $u_{2,2}$ is a highest-weight vector of conformal weight $h$, we obtain 
\begin{align}
\sum_{i \geq 0}\,&x_0^i\langle u_{2,1}, \cY_1(L_{-2-i} \cY_2(u_{2,1},x_0), x_2)u_{2,2}\rangle \nonumber \\ 
&=\sum_{i\geq 0}\sum_{j \geq 0}(-1)^{i+j}\binom{-i-1}{j}x_0^ix_2^{-i-j-1}\langle u_{2,1}, \cY_1(\cY_2(u_{2,1}, x_0)u_{2,2}, x_2)L_{j-1}u_{2,2}\rangle \nonumber \\
&=\sum_{i\geq 0}  (-1)^{i}x_0^ix_2^{-i-1}\langle u_{2,1},\cY_1(\cY_2(u_{2,1}, x_0)u_{2,2}, x_2)L_{-1}u_{2,2}\rangle \nonumber\\
&\ \ \ \ + \sum_{i\geq 0} (-1)^i(i+1)x_0^ix_2^{-i-2}\langle u_{2,1}, \cY_1(\cY_2(u_{2,1},x_0)u_{2,2}, x_2)L_0u_{2,2}\rangle \nonumber\\
&=(x_2+x_0)^{-1} \langle u_{2,1}, \cY_1(\cY_2(u_{2,1},x_0)u_{2,2}, x_2)L_{-1}u_{2,2}\rangle \nonumber \\
&\ \ \ \ + h(x_2+x_0)^{-2} \langle u_{2,1}, \cY_1(\cY_2(u_{2,1},x_0)u_{2,2}, x_2)u_{2,2}\rangle. \label{caj}
\end{align}
Moreover, the $L_{-1}$-commutator formula (the $n=-1$ case of \eqref{eqn:comm1}) implies that the second term of equation \eqref{fito} equals
\begin{align}
x_0^{-1} \left(\langle u_{2,1}, \cY_1(L_{-1}\cY_2(u_{2,1}, x_0)u_{2,2}, x_2) u_{2,2} \rangle-\langle u_{2,1}, \cY_1(\cY_2(L_{-1}u_{2,1}, x_0)u_{2,2}, x_2) u_{2,2} \rangle\right), \label{can1}
\end{align}
and that the first term of the right side of \eqref{caj} can be rewritten as
\begin{align}
& - (x_2+x_0)^{-1} \langle u_{2,1}, \cY_1(L_{-1} \cY_2(u_{2,1}, x_0)u_{2,2}, x_2)u_{2,2}\rangle \label{can2}
\end{align}
Thus, using equations \eqref{fito} through \eqref{can2} in \eqref{rhlll}, we obtain 
\begin{align*}
& \partial^2_{x_0} \Psi(x_0,x_2)\nonumber \\
& = -\ell (x_2+x_0)^{-1} \langle u_{2,1}, \cY_1(L_{-1} \cY_2(u_{2,1}, x_0)u_{2,2}, x_2)u_{2,2}\rangle \nonumber \\
& \ \ \ \ \ + \ell h (x_2+x_0)^{-2} \langle u_{2,1}, \cY_1(\cY_2(u_{2,1}, x_0)u_{2,2}, x_2)u_{2,2}\rangle \\
& \ \ \ \ \ +\ell x_0^{-1} \left(\langle u_{2,1}, \cY_1(L_{-1}\cY_2(u_{2,1}, x_0)u_{2,2}, x_2) u_{2,2} \rangle - \langle u_{2,1}, \cY_1(\cY_2(L_{-1}u_{2,1}, x_0)u_{2,2}, x_2) u_{2,2} \rangle \right)\nonumber\\
& \ \ \ \ \ + \ell h x_0^{-2}  \langle u_{2,1}, \cY_1(\cY_2(u_{2,1}, x_0) u_{2,2}, x_2) u_{2,2} \rangle\\
&=-\ell (x_2+x_0)^{-1}\partial_{x_2} \Psi (x_0, x_2)+ \ell h (x_2+x_0)^{-2} \Psi (x_0, x_2)\\
& \ \ \ \ \ + \ell x_0^{-1}\partial_{x_2} \Psi(x_0, x_2)
 - \ell x_0^{-1}\partial_{x_0}\Psi(x_0, x_2)+ \ell h x_0^{-2} \Psi(x_0, x_2)\\
 &= \ell (x_0^{-1}-(x_2+x_0)^{-1}) \partial_{x_2} \Psi(x_0, x_2)-\ell x_0^{-1} \partial_{x_0}\Psi(x_0, x_2)\\
 & \ \ \ \ \ + \ell h (x_0^{-2}+(x_2+x_0)^{-2}) \Psi(x_0, x_2),
\end{align*}
which proves the Lemma.

\end{proof}

Next, we define the auxiliary function (or formal series)
\begin{equation*}
    \widetilde{\psi}(x) = \langle u_{2,1},\cY_1(\cY_2(u_{2,1},x)u_{2,2},1)u_{2,2}\rangle.
\end{equation*}
It follows from the $L_0$-conjugation formula \cite[Equation 5.2.37]{FHL} that
\begin{equation}\label{eqn:Psitilde-relations}
    \Psi(x_0,x_2) = x_2^{-2h}\widetilde{\psi}\left(\frac{x_0}{x_2}\right),\qquad \psi(z) = z^{-2h}\widetilde{\psi}\left(\frac{1-z}{z}\right).
\end{equation}
We use these relations to prove Theorem \ref{teodiff}:
\begin{proof}
    We use the second relation in \eqref{eqn:Psitilde-relations} as well as the product and chain rules to get
    \begin{align*}
        \psi'(z) & =-2h z^{-2h-1}\widetilde{\psi}\left(\frac{1-z}{z}\right)- z^{-2h-2}\widetilde{\psi}'\left(\frac{1-z}{z}\right),\nonumber\\
        \psi''(z) & =2h(2h+1)z^{-2h-2}\widetilde{\psi}\left(\frac{1-z}{z}\right) +(4h+2)z^{-2h-3}\widetilde{\psi}'\left(\frac{1-z}{z}\right) + z^{-2h-4}\widetilde{\psi}''\left(\frac{1-z}{z}\right).
    \end{align*}
Solving these equations for $\widetilde{\psi}'\left(\frac{1-z}{z}\right)$ and $\widetilde{\psi}''\left(\frac{1-z}{z}\right)$ yields
\begin{align*}
    \widetilde{\psi}'\left(\frac{1-z}{z}\right) & = -z^{2h+1}\left(z\psi'(z)+2h\psi(z)\right),\nonumber\\
    \widetilde{\psi}''\left(\frac{1-z}{z}\right) & = z^{2h+2}\left(z^2\psi''(z)+(4h+2)z\psi'(z)+2h(2h+1)\psi(z)\right).
\end{align*}
Combining these relations with the first relation in \eqref{eqn:Psitilde-relations} yields
\begin{align*}
    (\partial_{x_0}\Psi)\big\vert_{(x_0,x_2)=(1-z,z)} & = x_2^{-2h-1}\widetilde{\psi}'\left(\frac{x_0}{x_2}\right)\bigg\vert_{(x_0,x_2) =(1-z,z)} =-z\psi'(z)-2h\psi(z),\nonumber\\
    (\partial_{x_0}^2\Psi)\big\vert_{(x_0,x_2)=(1-z,z)} & = x_2^{-2h-2}\widetilde{\psi}''\left(\frac{x_0}{x_2}\right)\bigg\vert_{(x_0,x_2)=(1-z,z)}\nonumber\\
    &=z^2\psi''(z)+(4h+2)\psi'(z)+2h(2h+1)\psi(z)\nonumber\\
    (\partial_{x_2}\Psi)\big\vert_{(x_0,x_2)=(1-z,z)} & =\left(-2h x_2^{-2h-1}\widetilde{\psi}\left(\frac{x_0}{x_2}\right)-x_0x_2^{-2h-2}\widetilde{\psi}'\left(\frac{x_0}{x_2}\right)\right)\bigg\vert_{(x_0,x_2)=(1-z,z)}\nonumber\\
    & = (1-z)\psi'(z)-2h\psi(z).
\end{align*}
Thus substituting $(x_0,x_2)\mapsto(1-z,z)$ in \eqref{equlem1} yields
\begin{align*}
    z^2\psi''(z) & +(4h+2)z\psi'(z) +2h(2h+1)\psi(z)\nonumber\\
    & = \ell\left((1-z)^{-1}-1\right)\left((1-z)\psi'(z)-2h\psi(z)\right)\nonumber\\
    &\qquad -\ell(1-z)^{-1}(-z\psi'(z)-2h\psi(z))+\ell h\left((1-z)^{-2}+1\right)\psi(z),
\end{align*}
which simplifies to \eqref{eqn:psi-diff-eqn}. 

\end{proof}

To solve the differential equation \eqref{eqn:psi-diff-eqn}, straightforward but tedious calculations using the identity $h\ell =\frac{3}{4}(\ell-1)^2$ show:
\begin{cor} \label{corab}
Let $f(z)$ be the analytic function on the region $\vert z\vert>\vert 1-z\vert>0$ such that
\begin{align*}
\psi(z)=(1-z)^{(\ell-1)/2}z^{-2h}f(z).
\end{align*} 
Then $f(z)$ satisfies the hypergeometric differential equation
\begin{align} \label{feq}
z(1-z)f''(z)+(\gamma-(\a+\b+1)z) f'(z)
-\a\b f(z)=0,
\end{align}
where $\a = -\b = 1-\ell$ and $\gamma = 2-2\ell$.
\end{cor}

If $\gamma-\a-\b$ is not an integer, that is, $2\ell\notin\ZZ$, then by \cite[Equation 15.10.4]{DLMF}, the hypergeometric differential equation \eqref{feq} has the following basis of solutions on the region $1>\vert 1-z\vert >0$:
\begin{align*}
&f_1(z)={}_2F_1(\a, \b; \a+\b+1-\gamma; 1-z),\\
&f_2(z)=(1-z)^{\gamma-\a-\b}{}_2F_1(\gamma-\a, \gamma-\b; \gamma-\a-\b+1; 1-z).
\end{align*}
By \cite[Equations~15.10.13 and 15.10.14]{DLMF}, these basis solutions agree with the following basis of solutions on the overlapping region $\vert z\vert>\vert 1-z\vert >0$:
\begin{align*}
&g_1(z)=z^{-\a}{}_2F_1\left(\a, \a-\gamma+1; \a+\b -\gamma+1; -\frac{1-z}{z}\right),\\
&g_2(z)=z^{\a-\gamma}(1-z)^{\gamma-\a-\b} {}_2F_1\left(1-\a, \gamma-\a; \gamma-\a-\b+1; -\frac{1-z}{z}\right).
\end{align*}

\begin{theorem} \label{thmhyper}
There exists $\mu_\ell\in\CC$ such that 
\begin{align*}
\psi (z)
=\mu_\ell \,(1-z)^{-3(\ell-1)/2}z^{\ell-1-2h}{}_2F_1\left(\ell, 1-\ell; 3-2\ell; -\frac{1-z}{z}\right)
\end{align*}
\end{theorem}

\begin{proof}
 Since $\ell\notin\QQ$, $\gamma-\a-\b= 2-2\ell$ is not an integer. Thus by Corollary \ref{corab}, $\psi(z)$ is a linear combination of 
\begin{align}
&\psi_1(z)=(1-z)^{(\ell-1)/2}z^{\ell-1-2h}{}_2F_1\left(1-\ell, \ell; 2\ell -1; -\frac{1-z}{z}\right), \\
&\psi_2(z)=(1-z)^{-3(\ell-1)/2}z^{\ell-1-2h} {}_2F_1\left(\ell , 1-\ell ; 3-2\ell ; -\frac{1-z}{z}\right). \label{phi2a}
\end{align}
Since $\cY_2$ is an intertwining operator of type $\binom{L(c_\ell , h_{1,2}(\ell ))}{L(c_\ell , h_{2,1}(\ell ))\,\, L(c_\ell , h_{2,2}(\ell ))}$,  the $L_0$-conjugation formula \cite[Equation 5.2.37]{FHL} implies that
\[
\cY_2 (u_{2,1},x)u_{2,2}\in x^{h_{1,2}(\ell )-h_{2,1}(\ell )-h_{2,2}(\ell )} L(c_\ell , h_{1,2}(\ell ))[[x]],
\]
and then because $z^{-2h}=\left(1+\frac{1-z}{z}\right)^{2h}\in\CC[[\frac{1-z}{z}]]$,
\begin{align*}
\psi(z)&=\langle u_{2,1}, \cY_1(\cY_2(u_{2,1}, 1-z)u_{2,2}, z)u_{2,2}\rangle\\
&=z^{-2h}\bigg{\langle} u_{2,1}, \cY_1\left(\cY_2\left(u_{2,1}, \frac{1-z}{z}\right)u_{2,2}, z\right)u_{2,2}\bigg{\rangle}\\
&\in \left(\frac{1-z}{z}\right)^{h_{1,2}(\ell)-h_{2,1}(\ell )-h_{2,2}(\ell )}\CC \left[\left[\frac{1-z}{z}\right]\right].
\end{align*}
Since $h_{1,2}(\ell )-h_{2,1}(\ell )-h_{2,2}(\ell)=-\frac{3(\ell -1)}{2}$, and since $\frac{\ell-1}{2}-\left(-\frac{3(\ell-1)}{2}\right)=2\ell-2\notin\ZZ$, $\psi(z)$ must be a multiple of \eqref{phi2a}.

\end{proof}

\section*{Declarations}
The authors have no relevant financial or non-financial interests to disclose.
\section*{Data availability statement}
All data supporting the findings of this work are presented within the paper and its supplementary materials.
\bibliographystyle{plain}

\end{document}